\documentclass[a4paper,reqno]{amsart}

\usepackage{amsmath}
\usepackage{amsthm}
\usepackage{amsfonts}
\usepackage{amssymb}
\usepackage{mathrsfs}
\usepackage[shortlabels]{enumitem}
\usepackage{graphicx}
\usepackage{subfigure}
\usepackage[dvipsnames]{xcolor}
\usepackage{setspace}

\usepackage[update,prepend]{epstopdf}
\usepackage{tikz}

\usepackage{mathtools} 
\usepackage{courier} 
\usepackage[T1]{fontenc}  

\newtheorem{theorem}{Theorem}[section]
\numberwithin{theorem}{section}
\newtheorem{prop}[theorem]{Proposition}
\newtheorem{lemma}[theorem]{Lemma}
\newtheorem{corollary}[theorem]{Corollary}
\theoremstyle{definition}
\newtheorem{definition}[theorem]{Definition}
\newtheorem{rem}[theorem]{Remark}

\newtheorem{example}[theorem]{Example}

\numberwithin{equation}{section}

\newcommand{\N}{\mathbb{N}}

\newcommand{\R}{\mathbb{R}}

\newcommand{\C}{\mathbb{C}}

\newcommand\id{\operatorname{id}}
\newcommand\e{\mathrm{e}}
\newcommand\I{\mathrm{i}}
\newcommand\re{\operatorname{Re}}
\newcommand\im{\operatorname{Im}}

\newcommand{\cV}{\mathcal V}
\newcommand{\sV}{\mathscr V}

\newcommand{\rd}{\mathrm{d}}

\newcommand\dom{\operatorname{dom}}
\newcommand\ran{\operatorname{ran}}
\DeclareMathOperator{\Div}{div}
\DeclareMathOperator{\curl}{curl}

\newcommand\diag{{\mathrm{diag}}}

\newcommand\cP{\mathcal P}

\newcommand\cT{\mathcal T}
\newcommand\cU{\mathcal U}
\newcommand\cA{\mathcal A}
\newcommand\cB{\mathcal B}

\newcommand\cD{\mathcal D}

\newcommand\cJ{\mathcal J}
\newcommand\cL{\mathcal L}

\newcommand\cH{\mathcal H}
\newcommand\cK{\mathcal K}
\newcommand\cW{\mathcal W}

\newcommand\ov\overline

\newcommand\eps\varepsilon
\renewcommand\epsilon\varepsilon
\renewcommand\rho\varrho
\newcommand\al\alpha
\newcommand\la\lambda

\newcommand\ds\displaystyle

\newcommand\p\partial

\newcommand{\sto}{\stackrel{s}{\rightarrow}}

\newcommand{\s}{\stackrel{s}{\rightarrow}}

\newcommand{\dist}{\operatorname{dist}}

\newcommand{\supp}{\operatorname{supp}}

\newcommand{\beq}{\begin{equation}}
\newcommand{\eeq}{\end{equation}}
\newcommand{\be}{\begin{equation*}}
\newcommand{\ee}{\end{equation*}}
\newcommand{\bmat}{\begin{pmatrix}}
\newcommand{\emat}{\end{pmatrix}}

\colorlet{DarkOlive}{OliveGreen!70!Black}

\newcounter{counter_a}
\newenvironment{myenum}{\begin{list}{{\rm \roman{counter_a})}}%
{\usecounter{counter_a}
\setlength{\itemsep}{0.5ex}\setlength{\topsep}{0.7ex}
\setlength{\leftmargin}{4ex}\setlength{\labelwidth}{5ex}}}{\end{list}}

\DeclarePairedDelimiter{\norma}{\lVert}{\rVert}

\author[S.\ B\"ogli]{Sabine B\"ogli}
\address[S.\,B.]{\!
Department of Mathematical Sciences, Durham University,
Upper Mount\-joy, South Road, Durham DH1 3LE, UK}
\email{sabine.boegli@durham.ac.uk}

\author[F.\ Ferraresso]{Francesco Ferraresso}
\address[F.\,F.]{\!
School of Mathematics,
Cardiff University, Abacws,
Senghennydd Road, Cathays, 
Cardiff, Wales, CF24 4AG, UK}
\email{FerraressoF@cardiff.ac.uk}

\author[M.\ Marletta]{Marco Marletta}
\address[M.\,M.]{\!
School of Mathematics,
Cardiff University, Abacws,
Senghennydd Road, Cathays,
Cardiff CF24 4AG, UK}
\email{MarlettaM@cardiff.ac.uk}

\author[C.\ Tretter]{Christiane Tretter}
\address[C.\,T.]{\!
Mathematisches Institut,
Universit\"at Bern, Sidlerstrasse 5, 3012 Bern, Switzerland}
\email{tretter@math.unibe.ch}


\date{\today}

\thanks{}
\usepackage[update,prepend]{epstopdf}

\title[Spectral approximation for Maxwell's equations]%
{Spectral analysis and domain truncation methods for Maxwell's equations}

\begin{document}

\begin{abstract}
We analyse how the spectrum of the anisotropic Maxwell system with bounded conductivity $\sigma$ on a Lipschitz domain $\Omega$ is approximated by domain truncation. First we prove a new non-convex enclosure for the spectrum of the Maxwell system, with weak assumptions on the geometry of $\Omega$ and none on the behaviour of the coefficients at infinity. We also establish a simple criterion for non-accumulation of eigenvalues at $\I \R$ as well as resolvent estimates. For asymptotically constant coefficients, we describe the essential spectrum and show that spectral pollution may occur only in the essential numerical range $W_e(L_\infty)\!\subset\!\R$~of the quadratic pencil $L_\infty(\omega) \!=\! \mu_\infty^{-1}\curl^2 \!-\! \omega^2\eps_\infty\!$, acting on divergence-free vector fields. Further, every isolated spectral point of the Maxwell system lying outside $W_e(L_\infty)$ and outside the part of the essential spectrum on $\I\R$ is approximated by spectral points of the Maxwell system on the trun\-cated domains. Our analysis is based on two new abstract results on~the (limiting) essential spectrum of polynomial pencils and triangular block operator matrices, which are of general interest. We believe our strategy of proof could be~used to establish domain truncation spectral exactness for more general classes of non-self-adjoint differential operators and systems with non-constant  \vspace{-7.8mm} coefficients.
\end{abstract}

\maketitle

\section{\bf Introduction}
\label{section:1}

Given a possibly unbounded domain $\Omega \!\subset\! \R^3$ with Lipschitz bound\-ary, and an increasing sequence of bounded Lipschitz  
domains $\Omega_n \!\subset\! \Omega$ exhausting $\Omega$, we are interested in the spectral properties of the anisotropic Maxwell  system
\begin{alignat}{3}
\label{intro: Max1}
& \ \ \begin{array}{rll}
-\I\sigma E + \I\curl H \!\!\!\!& = \,\omega \eps E & 
\\
-\I\curl E \!\!\!\! & =\omega \mu H& 
\end{array}
\textup{in $\Omega$},
&& \quad \nu \times E = 0 \ \ \textup{on $\partial \Omega$}, \\
\intertext{and in its spectral approximation via the sequence of \vspace{-0.1mm}problems}
\label{intro: Maxn}
& \begin{array}{rll}
-\I \sigma E_n \!+ \I\curl H_n \!\!\!\!\!& = \,\omega \eps E_n & 
\\
-\I\curl E_n \!\!\!\!\!&  =  \omega \mu H_n & 
\end{array}
\!\!\!\!\textup{in $\Omega_n$,}
&& \quad \nu \times \!E_n \!= 0 \ \, \textup{on $\partial \Omega_n$}, \ \ n\in\N.
\end{alignat}
Here $\omega$ is the spectral parameter, $\eps$ the electric permittivity, $\mu$ the magnetic permeability and $\sigma$ the conductivity; $\nu$ is the outward unit normal vector to the boundary. We allow the coefficients $\eps$, $\mu$, $\sigma$ to be non-constant and tensor-valued, bounded on $\Omega$ with non-negative matrix values; for some~results, e.g.\  involving the essential spectrum, we assume $\eps \!=\! \eps_\infty \id$, $\mu\!=\! \mu_\infty \id$ and $\sigma \!=\! 0$ at infinity.

We denote by $V(\cdot)$ and $V_n(\cdot)$ the operator pencils associated with problem~\eqref{intro: Max1} and \eqref{intro: Maxn} in $L^2(\Omega,\mathbb{C}^3) \oplus L^2(\Omega,\mathbb{C}^3)$ and $L^2(\Omega_n,\mathbb{C}^3) \oplus L^2(\Omega_n,\mathbb{C}^3)$, respectively, given by the same matrix differential \vspace{-0.6mm} expression
\begin{equation}
\label{eq:Vdef0}
V(\omega) = \begin{pmatrix} -\I\sigma & \I\curl \\ -\I\curl_0 & 0 \end{pmatrix}
- \omega \begin{pmatrix} \eps & 0 \\ 0 & \mu \end{pmatrix}, \quad \omega \in \C,
\vspace{-0.5mm}
\end{equation}
on their respective domains which are independent of $\omega$, see \eqref{eq:Vdef} below.

An important feature of our Maxwell systems is that the conductivity $\sigma$ is assumed to be non-trivial, making the problem dissipative rather than self-adjoint, see e.g.\ \cite{MR2202183}, \cite{MR3959447}, \cite{MR4029529}, \cite{MR3942228}. 
Furthermore, we avoid any hypotheses on the permeability, permittivity and conductivity, or upon the geometry, which would allow the use of TE- and TM- mode reductions to second order operators of Schr\"{o}dinger or conductivity type. This lack of simplifying hypotheses introduces significant additional hurdles in the analysis compared to the self-adjoint case, some of which were already apparent in the paper of Alberti et al.\  \cite{MR3942228} on the essential spectrum (see also Lassas \cite{MR1637445} for bounded domains). The non-convexity of the essential spectrum, consisting of a part which is purely real and a part which is purely imaginary, might be expected to lead to much more {\em spectral pollution}.

In the self-adjoint case, this phenomenon is well known when variational approximation methods are used, see e.g.\ Rappaz et al.\ \cite{MR1431212}: following discretisation, the spectral gaps may fill up with eigenvalues of the discretised problem which are so closely spaced that it may be impossible to distinguish the spectral bands from the spectral gaps. For finite element approximations to Maxwell systems on bounded domains, this may be avoided by the use of appropriate {\em conforming} elements, see N\'{e}d\'{e}lec \cite{MR592160}. The study of which finite element bases pollute for a given class of problems has been taken up by many authors: see, e.g.\ Barrenechea et al.\ \cite{MR3285902} for self-adjoint Maxwell systems on bounded, convex domains; Costabel et al.\ \cite{MR2032869} for an application to Maxwell resonances; and Lewin and S\'{e}r\'{e} \cite{MR2640293} for self-adjoint Dirac and Schr\"{o}dinger equations. Unfortunately, in our non-self-adjoint context, elegant techniques such as quadratic relative spectrum \cite{MR2068829} or residual-minimisation algorithms \cite{MR2068830,MR1337263} are not available.

For particular differential operators on infinite domains or with singularities, spectral pollution caused by domain truncation is also well studied. To avoid it one may, for instance, devise {\em non-reflecting boundary conditions}  \cite{MR1005207}, or resort to the complex scaling method \cite{MR0345551}, which reappeared as the {\em perfectly matched layer} (PML) method in the computational literature \cite{MR1294924}. In fact this technique replaces a self-adjoint problem by a non-self-adjoint one.

In our opinion the clearest way to think about these methods, and about dissipative barrier methods more generally, is that they replace the underlying operator by one whose {\em essential numerical range} \cite{MR4083777} does not contain the eigenvalues of interest. The results in \cite{MR4083777} then give a unified explanation of why such methods work, within a wide operator-theoretic framework which also allows a uniform
treatment of many of the finite element approximation schemes.

The Maxwell system, however, presents some additional challenges: for a start, \eqref{eq:Vdef0} defines a {\em pencil} of operators, for which fewer
results on spectral pollution~are available. We generalise the concept of \emph{limiting essential spectrum}, pre\-sent\-ed in \cite{MR3831156}, to sequences of pencils of closed operators $T_n : \C \!\to\! C(\cH)$ with domains $\dom(T_n)$ constant for each $n\in\N$, by means of \vspace{-0.6mm} the~formula
\begin{align*}
\sigma_{e}\big((T_n)_{n\in\N}\big)\!:=\!
\big\{\la\!\in\!\C:\,\exists\, I\!\subset\!\N,\,I\ \mbox{infinite},\, \exists\,&x_n\!\in\!\dom(T_n), \, \|x_n\|\!=\!1,\,n\!\in\! I, \\[-0.5mm]
&\!\!\!\text{with }x_n \rightharpoonup 0,\,\|T_n(\la)x_n\|\to 0\big\}.\\[-6mm]
\vspace{-3mm}
\end{align*}
%
This generalisation is important because the best results on spectral pollution come not from considering the linear Maxwell pencil (\ref{eq:Vdef0}), but rather by eliminating the magnetic field $H$ to obtain a {\em quadratic pencil} $\cL(\cdot)$ whose numerical range is not convex. 
Another key ingredient is the operator matrix structure of the pencil $\cL(\cdot)$ induced by the Helmholtz decomposition.

Our main result, see Theorem \ref{theorem: main res}, establishes a surprisingly small enclosure for the set of spectral pollution of the domain truncation method for \eqref{intro: Max1}, which is much smaller than the one given by the essential numerical range $W_e(V)$, a convex set enclosing the essential spectrum, see \cite{MR4083777}, \cite{BM}. In fact, Theorem~\ref{theorem: main res} goes beyond what can be achieved using essential numerical ranges, whether for pencils or operators: it relies on new results which we develop
in Sections \ref{section:6} and~\ref{sec: triangular} on limiting essential spectra of sequences of polynomial operator pencils
and operator~matrices. 

To the best of our knowledge, the domain truncation results we present here for the Maxwell system in unbounded domains are new even in the
self-adjoint setting. Of course, there are good physical reasons why many problems for Maxwell systems on infinite domains have coefficients
which are constant outside a compact set, and in these cases domain truncation can be avoided by using domain decomposition and boundary
integral techniques. These approaches are extensively researched, see e.g.\ \cite{MR1935806}, \cite{MR2059447}, \cite{MR1742312}.

Much of the proof of Theorem \ref{theorem: main res} relies on new, non-convex enclosures for the spectra of Maxwell problems, which we present
in Theorem \ref{thm:spec-encl}. These are valid for the original problem \eqref{intro: Max1} on $\Omega$, for all the truncated problems \eqref{intro: Maxn} on $\Omega_n$ and, if they exist, for corresponding `limiting problems at $\infty$'. In particular, they provide what are, to our knowledge,
the first enclosures for the essential spectrum if the coefficients do \emph{not} have limits at $\infty$, and novel bounds for the non-real eigenvalues.
The non-convexity of our enclosures allows them to be much tighter than bounds obtained from the numerical range, which is a horizontal
strip below the real axis. In fact, apart from the imaginary axis, the new spectral enclosures are contained in a strip whose width is \emph{half} that of the numerical range. They also provide an incredibly simple criterion for non-accumulation of the spectrum at $\I\R$, including non-accumulation at $0$.

The paper is organised as follows. In Section \ref{section:2} we present our main results,~illu\-strate our new spectral enclosure and give some examples showing e.g\ that the~latter is sharp. Section \ref{sec:proof1} contains the proof of the spectral enclosure theorem and some auxiliary results such as resolvent estimates. In Section \ref{sec:V-L} we study the relations between the spectra and essential spectra of the Maxwell pencil $V(\cdot)$ and  the quadratic operator pencil $\cL(\cdot)$. This enables us to explicitly characterise the essential spectrum of the Maxwell pencil in terms of the asymptotic limits of the coefficients $\eps$, $\mu$ and $\sigma$ in Section \ref{section:5}. In Section \ref{section:6}
we prove abstract results on spectral pollution, limiting approximate point spectrum and limiting essential spectrum for polynomial operator pencils. In Section \ref{sec: limiting sigmae} we investigate the limiting essential spectrum of the Maxwell pencil $V(\cdot)$ via the associated quadratic operator pencil $\cL(\cdot)$. As a consequence, we prove absence of spectral pollution for domain truncation outside the union of two sets on the  real and the imaginary axis, the essential numerical range of the self-adjoint limiting quadratic operator pencil $L_\infty(\cdot)$ on the real axis and the convex hull of the essential spectrum on the imaginary axis. Section \ref{sec: triangular} and the Appendix contain the abstract results on essential spectra for upper triangular operator matrices and computational details for the example in Section 2, respectively.


\section{\bf Main results and examples}
\label{section:2}

As explained in the introduction, we are interested in domain truncation methods for the anisotropic Maxwell system (\ref{intro: Max1}).
We assume that the coefficients $\eps$, $\mu$ and $\sigma$ are non-negative symmetric matrix valued functions in
$L^\infty(\Omega, \R^{3 \times 3})$ such that, for some constants $\eps_{\min}$, $\eps_{\max}$, $\mu_{\min}$, $\mu_{\max}$,
\vspace{-0.5mm}  $\sigma_{\max}$,
\begin{align}
\label{conditions eps mu sigma}
\hspace{-3mm}
0\!<\!\eps_{\min} \!\leq\! \eta \!\cdot\! \eps \eta \!\leq\! \eps_{\max}, \ \
0\!<\!\mu_{\min}  \!\leq\! \eta \!\cdot\! \mu \eta \!\leq\! \mu_{\max}, \ \
0 \!\leq\! \eta \!\cdot\! \sigma \eta \!\leq\! \sigma_{\max},
\quad \eta \!\in\! \R^3, \,|\eta|\!=\!1. \hspace{-3mm}
\vspace{-1mm}
\end{align}

The magnetic field $H$ and electric field $E$  lie respectively in the function \vspace{-0.5mm} spaces
\[
H(\curl, \Omega) := \{u \in L^2(\Omega)^3: \curl u \in L^2(\Omega)^3 \},\]
\[H_0(\curl, \Omega) := \{ u \in H(\curl, \Omega) : \nu \times u|_{\partial \Omega} = 0\},
\]
with the canonical norm $\norma{u}_{H(\curl, \Omega)} \!:=\! ( \norma{u}^2 \!+\! \norma{\curl u}^2 )^{1/2}\!$.
Unless stated otherwise, our function spaces consist of complex-valued functions and so we write, for example, $L^2(\Omega)=L^2(\Omega,\C)$ for short.

We asso\-ciate two operators with the symmetric differential expression $\curl$ in $L^2(\Omega)^3$, first,
the operator $\curl$ on its maximal domain $\dom \curl\!=\!H(\curl, \Omega)$ and, secondly, the adjoint $\curl_0\!=\!\curl^*$
of the operator $\curl$, given by $\curl$ on the domain~$\dom \curl_0\!=\!H_0(\curl, \Omega)$.

We now recall the definitions of other function spaces used in the sequel.
The ho\-mogeneous Sobolev spaces $\dot{H}^1_0(\Omega)$ and $\dot{H}^1(\Omega)$ are defined as the completions of the Schwartz spaces
 $\cD(\Omega)$ and $\cD(\overline{\Omega})$, respectively, with respect to the seminorm $\norma{u}_{\dot{H}^1(\Omega)} \!:=\! \norma{\nabla u}_{L^2(\Omega)}$. These spaces are in general strictly bigger~than the usual Sobolev spaces $H^1_0(\Omega)$ and $H^1(\Omega)$ if $\Omega$ does not have
finite measure~or if $|\Omega| \!<\! \infty$ but fails to have quasi-resolved boundary in the sense of Burenkov-Maz\-'ya, see \cite[Sect.\,4.3, p.\,148-150]{MR1622690} (note that Lipschitz domains have quasi-resolved~boundary).

The spaces $\nabla \dot{H}^1_0(\Omega)$ and $\nabla \dot{H}^1(\Omega)$ are the images of $\dot{H}^1_0(\Omega)$ and
$\dot{H}^1(\Omega)$, respectively, under the gradient. Further,
we~define
\begin{align}
H(\Div, \Omega) &:= \{ u \in L^2(\Omega)^3 : \Div u \in L^2(\Omega) \}, \\
H(\Div0, \Omega) &:= \{ u \in L^2(\Omega)^3  :  \Div u = 0 \: \textup{in $\Omega$} \}.
\end{align}
Here we equip $H(\Div, \Omega)$ with the canonical norm $\norma{u}_{H(\Div,\Omega)} \!:=\! ( \norma{u}^2 + \norma{\Div u}^2 )^{1/2}\!$ and $H(\Div0, \Omega)$ is considered as a closed subspace of $L^2(\Omega)^3$ with the $L^2$-norm~which coincides with $\|\!\cdot\!\|_{H(\Div,\Omega)}$ on $H(\Div0, \Omega)$. Finally, the space $H(\curl,\Omega)\cap H(\Div,\Omega)$ is equipped with the norm
$\norma{u}_{H(\curl, \Omega)\cap H(\Div,\Omega)} \!:=\!\norma{u}_{H(\curl, \Omega)}+\norma{u}_{H(\Div, \Omega)}$.
%
%

We are now able to state our first new result, which yields non-convex spectral enclosures for dissipative Maxwell systems.
This enclosure yields the first bounds for both the essential spectrum and the non-real eigenvalues.

\begin{theorem}
\label{thm:spec-encl}
The Maxwell operator pencil in $L^2(\Omega)^3 \oplus L^2(\Omega)^3$ given by
\begin{equation}
\label{eq:Vdef}
  V(\omega) \!:=\!
  \begin{pmatrix}
   -\I\sigma \!\!&\!\! \I\curl \\ -\I\curl_0 \!\!&\!\! 0
  \end{pmatrix}
  \!-\omega \!\begin{pmatrix}
  \eps \!&\! 0 \\ 0 \!&\! \mu
  \end{pmatrix}\!,
  \, \
	\dom (V(\omega)) \!:=\! H_0(\curl,\Omega) \oplus H(\curl,\Omega),
\end{equation}
for $\omega\in\C$ satisfies the spectral \vspace{-1mm} enclosure
\begin{align*}
  \sigma(V) \subset \I\,
	\Big[\!-\!\frac{\sigma_{\max}}{\eps_{\min}},0 \Big] \cup 	
	\Big\{ \omega\!\in\! \C \!\setminus\! \I\R: 	& \im \omega \!\in\! \Big[\!-\! \frac 12 \frac{\sigma_{\max}}{\eps_{\min}},
	- \frac 12 \frac{\sigma_{\min}}{\eps_{\max}}\Big],   \\[-1mm]
 &(\re \omega)^2 \!\!-\! 3 (\im \omega)^2 \!+\! 2 \frac{\sigma_{\max}}{\eps_{\min}}  |\im \omega| \!\geq\! \frac{\lambda_{\min}^\Omega}{{\eps_{\max}}\mu_{\max}}
  \Big\}
 \\[-6.5mm]
\end{align*}
where $\lambda_{\min}^\Omega\!:=\!\min\sigma(\curl \curl_0|_{H(\Div 0,\Omega)}) \ge 0$.
In particular, if $\lambda_{\min}^\Omega \!>\!0$, \vspace{-1mm} then
\begin{equation}
\label{gaps}
\sigma(V) \cap \bigg( \Big( \!-\!\Big(\frac{\lambda_{\min}^\Omega}{{\eps_{\max}}\mu_{\max}}\Big)^{\!\!1\!/2}\!,0\Big) \cup \Big(0,\Big(\frac{\lambda_{\min}^\Omega}{{\eps_{\max}}\mu_{\max}}\Big)^{\!\!1\!/2}\Big) \bigg) =\emptyset,
\vspace{-1.2mm}
\end{equation}
and if $\lambda_{\min}^\Omega \!>\! \frac 13 \frac{\sigma_{\max}^2\eps_{\max} \mu_{\max}}{\eps_{\min}^2}$, then
$\sigma(V) \cap \I\R \!\subset\! \I\big[\!-\!\frac{\sigma_{\max}}{\eps_{\min}},0 \big]$ is isolated from~$\sigma(V) \setminus \I\R$.
\end{theorem}


\begin{rem}
The enclosure in Theorem \ref{thm:spec-encl} becomes larger when the domain $\Omega$ does, provided we choose the optimal values $\eps_{\min}^\Omega$, $\mu_{\min}^\Omega$, $\sigma_{\min}^\Omega$ and  $\eps_{\max}^\Omega$,  $\mu_{\max}^\Omega$,  $\sigma_{\max}^\Omega$ for $\Omega$ as the bounds in \eqref{conditions eps mu sigma}. In this case, $\eps_{\min}^\Omega$, $\mu_{\min}^\Omega$, $\sigma_{\min}^\Omega$ and $\lambda_{\min}^\Omega$ are decreasing with $\Omega$, while $\eps_{\max}^\Omega$,  $\mu_{\max}^\Omega$,  $\sigma_{\max}^\Omega$ are increasing.
\end{rem}

The possible different shapes of the above non-convex spectral enclosure are illustrated in Figure~\ref{fig:enclosure} below, see Remark \ref{rem:spec-encl} for details. While in all cases \vspace{-0.8mm}  accumulation of spectrum at $\I\R$ is excluded at the complex interval \vspace{-1mm} $\I\Big[\!\!-\!\!\tfrac{\sigma_{\max}}{\eps_{\min}},\!-\!\tfrac 12\tfrac{\sigma_{\max}}{\eps_{\min}}\Big]$, accumulation is also excluded successively i) near $0$,  ii) near $-\I \frac 12\frac{\sigma_{\max}}{\eps_{\min}}$ and iii) everywhere at $\I\R$  \vspace{-1mm} at the following thresholds for $\lambda_{\min}^\Omega$,
\begin{align}
\label{eq:thresholds}
   {\rm i)} \ \lambda_{\min}^\Omega\!>\!0, \quad
   {\rm ii)} \  \lambda_{\min}^\Omega\!>\!\frac 14 \frac{\sigma_{\max}^2 \eps_{\max} \mu_{\max}}{\eps_{\min}^2}, \quad
   {\rm iii)} \ \lambda_{\min}^\Omega\!>\!\frac 13 \frac{\sigma_{\max}^2 \eps_{\max} \mu_{\max}}{\eps_{\min}^2}. \quad
\end{align}%
The proof of Theorem~\ref{thm:spec-encl} is given in  Section \ref{sec:proof1}.

\begin{figure}
\includegraphics[width=0.58 \textwidth]{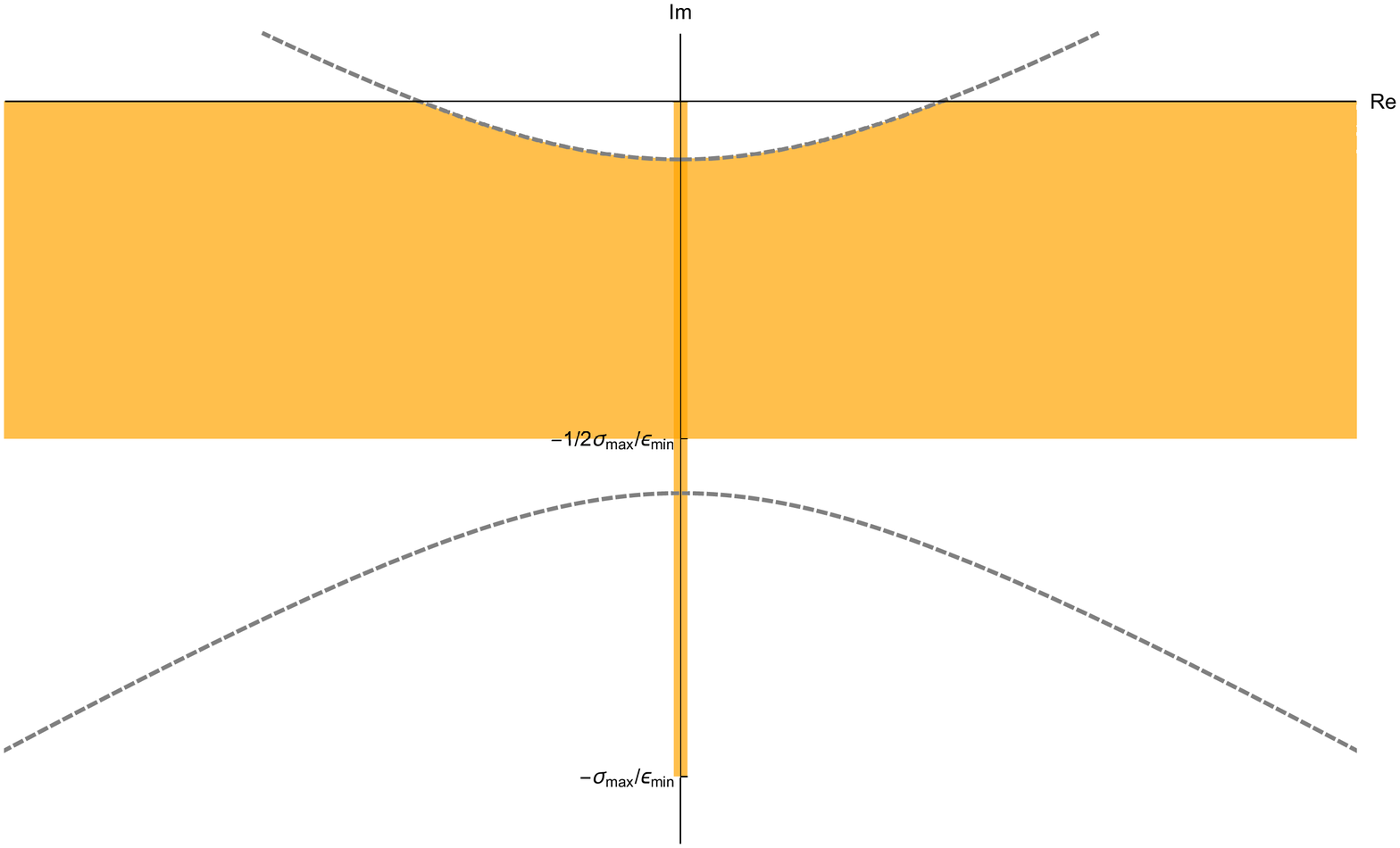}  
\\[2mm]
\includegraphics[width=0.58 \textwidth]{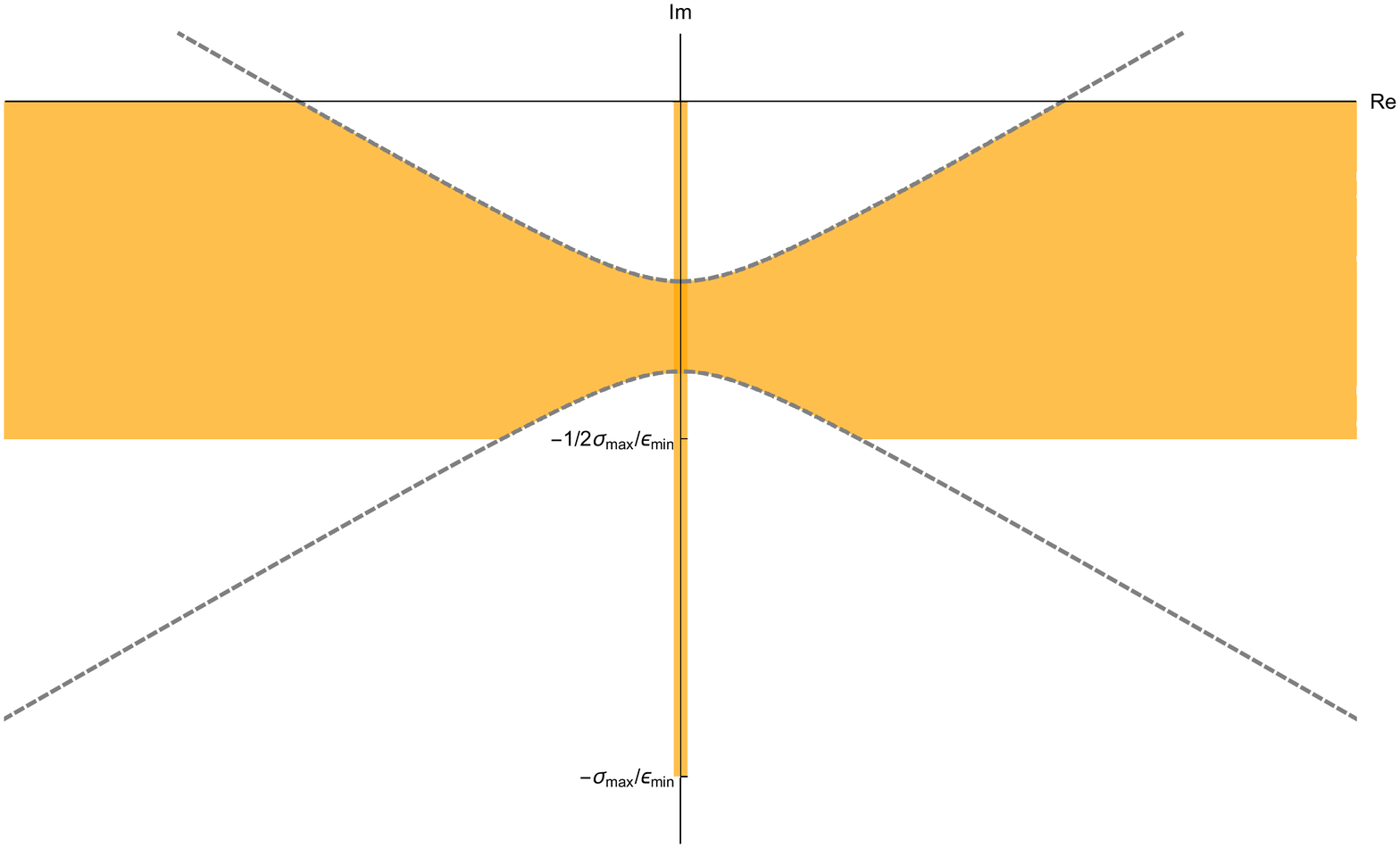}
\\[2mm]
\ \includegraphics[width=0.58 \textwidth]{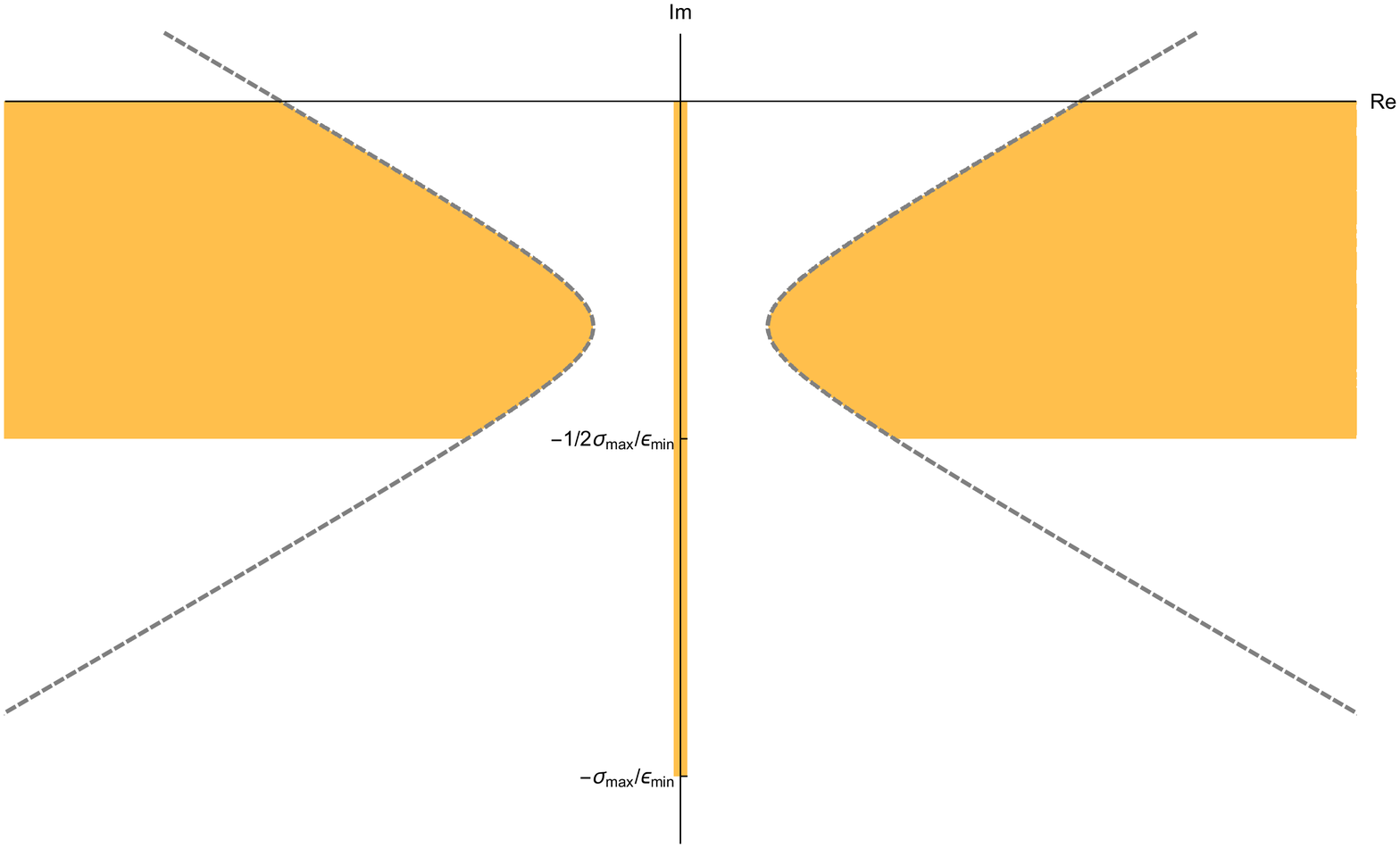}
\vspace{-2mm}
\caption{\small Spectral enclosure in Theorem \ref{thm:spec-encl} (yellow)
in cases i) (top), ii) (middle), iii) (bottom) of \eqref{eq:thresholds} for $\sigma_{\rm min}\!=\!0$; the dashed lines~are the boundary curves $(\re \omega)^2 \!\!-\! 3 (\im \omega)^2 \!+\! 2 \frac{\sigma_{\max}}{\eps_{\min}}  |\im \omega| \!=\!\lambda_{\min}^\Omega/\eps_{\max}\mu_{\max}$.}
\vspace{-2mm}
\label{fig:enclosure}
\end{figure}

For non-self-adjoint problems, it is crucial not only to establish spectral enclosures, but also resolvent estimates. The following resolvent bounds which we prove in Section  \ref{sec:proof1} also apply in the double  semistrip $\{z\!\in\!\C\!\setminus\!\I\R:
-  \frac {\sigma_{\max}}{\eps_{\min}} \!\le\! \im z \!<\! - \frac 12 \frac {\sigma_{\max}}{\eps_{\min}} \}$ inside of the closure of the numerical range of the Maxwell pencil.

\begin{theorem}
\label{thm:res-est-V}
For $\omega \in \{z\!\in\!\C\!\setminus\!\I\R:  \im z <  -  \frac 12 \frac {\sigma_{\max}}{\eps_{\min}} \}$, we \vspace{-1mm} have
\[
 \|V(\omega)^{-1}\| \!\leq\! \frac 1 {\min\{\eps_{\min},\mu_{\min}\}} \frac 1{|\im \omega|\!-\!\frac12 \frac{\sigma_{\max}}{\eps_{\min}}} \left( 1 \!+\! \frac {(\frac12 \frac{\sigma_{\max}}{\eps_{\min}})^2}{(\re\omega)^2} \right),
\vspace{-1mm} 
\]
and hence, for \vspace{-1mm}  $\omega \in \{z\!\in\!\C:  \im z <  - \frac {\sigma_{\max}}{\eps_{\min}} \}$
\[
   \|V(\omega)^{-1}\| \!\leq\! \frac 1 {\min\{\eps_{\min},\mu_{\min}\}\!} \min\left\{ \frac 1{|\im \omega|\!-\!\frac12 \frac{\sigma_{\max}}{\eps_{\min}}\!} \!\left( \!1 \!+\! \frac {(\frac12 \frac{\sigma_{\max}}{\eps_{\min}})^{2}}{(\re\omega)^2} \right)\!, \frac 1{|\im \omega|\!-\!\frac{\sigma_{\max}}{\eps_{\min}}}  \right\}\!.
\]
\end{theorem}

Note that the second resolvent bound in Theorem \ref{thm:res-est-V} follows since in the half-plane $\{z\!\in\!\C:  \im z <  - \frac {\sigma_{\max}}{\eps_{\min}} \}$ also the classical resolvent bound in terms of the numerical range of $V(\cdot)$ applies.

\begin{figure}[h]
\hspace{8mm}
\includegraphics[width=0.89\textwidth]{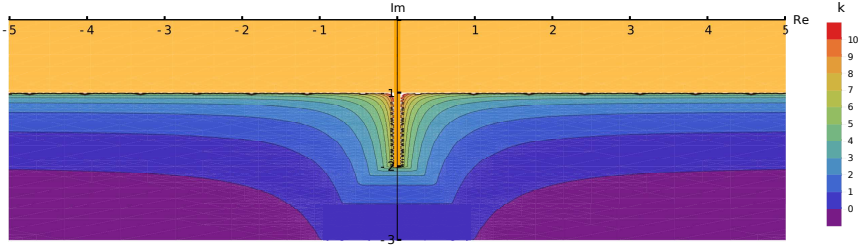}
\caption{\small\!Level curves of the resolvent norm bound in Theorem~\ref{thm:res-est-V}
for the case~$\eps_{\min}\!=\!\mu_{\min}\!=\!1$, $\sigma_{\max}\!=\!2$, also for regions inside the numerical range $W(V)=\!\{z\!\in\!\C: \im z \!\in\! [-2,0]\}$ of the Maxwell pencil.}
\end{figure}


The next group of new results concerns approximations of the Maxwell pencil.
Since $\Omega$ is a Lipschitz domain, we may assume that there exists a strictly increasing sequence of
bounded Lipschitz domains $(\Omega_n)_{n\in \N}$ such that $\bigcup_{n \in \N} \Omega_n = \Omega$.

It is clear that if $\Omega = \R^3$, or $\Omega$ has smooth boundary, we may choose $\Omega_n$ to be smooth domains for every $n \in \N$. We note that sequences of domains $(\Omega_n)_{n \in \N}$ as described above can always be constructed by setting $ \Omega_n = \Omega \cap B(0,n)$, $n \in \N$.

Define $V_n(\cdot)$  to be the Maxwell pencil in $L^2(\Omega_n)^3\oplus L^2(\Omega_n)^3$ with~domain
\[
\dom(V_n(\omega)) = H_0(\curl, \Omega_n) \oplus H(\curl, \Omega_n),  \ \omega\in\C, \quad n \in \N,
\]
and the set of spectral pollution for the domain truncation method $(V_n)_{n \in \N}$ as
\begin{equation}
\label{eq:poll}
\sigma_{\rm poll}((V_n)_{n\in\N}) := \{ \omega \in \C : \omega \in \rho(V), \exists \, \omega_n \in \sigma(V_n) : \omega_n \to \omega \}.
\end{equation}
For approximations of an abstract linear pencil $A - \la B$, $\la\in\C$, spectral pollution for the domain truncation method was localised inside its essential numerical range in \cite[Thm.\ 3.5]{BM}. For the Maxwell pencil $V(\cdot)$, it is not difficult to show that the es\-sent\-ial numerical range $W_e(V)$ is 
contained in the closed horizontal strip $\{z \!\in\! \C : - \frac{\sigma_{\max}}{\eps_{\min}}\!\leq\! \im z \!\leq\! 0 \}$.

Our second main result improves this enclosure substantially if we assume that the coefficients $\eps$, $\mu$, $\sigma$ have limits at $\infty$. It shows that, in fact, spectral pollution is confined to the real axis, with possible gaps on either side of $0$.

\begin{theorem}
\label{theorem: main res}
Suppose that $\Omega$ is an unbounded domain and that $\eps\!-\!\eps_\infty\id$, $\mu\!-\!\mu_\infty\id$ and $\sigma$ vanish at infinity
for some $\eps_\infty$, $\mu_\infty>0$, \vspace{-1.5mm} i.e.
\begin{equation}
\label{eq:coeffs-infty}
 \lim_{R\to\infty}\left\{ \sup_{\| x \| > R} \max \, \big(\| \eps(x)-\eps_\infty\id\|,\|\mu(x)-\mu_\infty\id\|,\|\sigma(x)\|\big)\right\} = 0.
\end{equation}
Let $L_\infty$ be the operator pencil in the subspace $H(\Div 0,\Omega)$ of $L^2(\Omega)^3$ defined by
\begin{align*}
& L_\infty(\omega) := \curl \mu_\infty^{-1} \curl_0 - \omega^2 \eps_\infty, \\
& \dom(L_\infty(\omega)) := \{ E \in H_0(\curl, \Omega){\cap H(\Div 0,\Omega)} : \curl E \in H(\curl, \Omega) \},
\end{align*}
and let $\mathcal W(\cdot)$ be the operator pencil in $L^2(\Omega)^3$ defined by $\mathcal W(\omega)\!:=\!-\omega(\omega\eps+\I\sigma)$, $\omega\!\in\!\C$.
Then,
with \vspace{-1.5mm} $\lambda_{e,\min}^\Omega\!:=\!\min\sigma_e(\curl \curl_0|_{H(\Div 0,\Omega)}) \ge 0$,
\[
\sigma_{\rm poll}\big((V_n)_{n\in\N}\big) \subset  W_e(L_\infty)
\,=\Big(\!\!-\!\infty,\!-\Big(\frac{\lambda_{e,\min}^\Omega}{\eps_\infty\mu_\infty}\Big)^{\!\!1\!/2}\,\Big]\cup \Big[\Big(\frac{\lambda_{e,\min}^\Omega}{\eps_\infty\mu_\infty}\Big)^{\!\!1\!/2},\infty\Big)
\\\subseteq \mathbb{R};
\]
and for every isolated $\omega \in \sigma_p(V)$ outside $W_e(L_\infty) \cup \sigma_e(P_\nabla \cW(\cdot) |_{\nabla \!\dot H^1_0(\Omega)})$,
and hence outside the set $ W_e(L_\infty) \!\cup\! \I\big[\!-\!\frac{\sigma_{\max}}{\eps_{\min}},0\big]$,
%
there exists a sequence $\omega_n \in \sigma(V_n)$, $n\in\N$, such that $\omega_n \to \omega$ as $n \to \infty$.
\end{theorem}

The proof of Theorem \ref{theorem: main res} which relies on a combination of analytic and operator theoretic tools is given at the end of Section \ref{sec: limiting sigmae}.

\begin{rem}
The enclosure for spectral pollution in Theorem \ref{theorem: main res}
is a subset of the spectral enclosure in Theorem \ref{thm:spec-encl} on the real axis, see~\eqref{gaps}, since $\lambda_{e,\min}^\Omega \ge \lambda_{\min}^\Omega \ge 0$ and $\eps_\infty \!\le\! \eps_{\max}$, $\mu_\infty \!\le\! \mu_{\max}$.\\
Note that, depending on $\Omega$, it may happen that $\lambda_{\min}^\Omega \!>\! 0$ or $\lambda_{e,\min}^\Omega \!>\! \lambda_{\min}^\Omega \!\ge\! 0$;~in the former case, both enclosures for the spectrum and spectral pollution have a~gap on either side of $0$, in the latter case, the enclosure for spectral pollution has a gap on either side of $0$ and thus eigenvalues in these gaps are safe from spectral~pollution.
\end{rem}

As far as we know, Theorem \ref{theorem: main res} is new even in the self-adjoint case, see also Theorem \ref{thm: main res self-adj}.
In the general case, it yields spectral exactness for every non-real, isolated eigenvalue of the Maxwell system  and, if $\lambda_{e,\min}^\Omega \!>\!0$, also for the real eigenvalues in the gaps of the essential spectrum to either side of $0$.

The following examples illustrate our results on spectral enclosure, the essential spectrum  and spectral pollution. The first example also
provides an idea of~the complex spectral structure that may arise even for rather simple Maxwell systems~\eqref{intro: Max1}.




\begin{example}
\label{ex:waveguide}
We consider the semi-infinite cylinder $\Omega=(0, \infty)\times(0,L_2)\times(0,L_3)$
and suppose that $\eps=\mu = \id$ everywhere, and $\sigma = \id$ if $x_1\in (0,1)$, else $\sigma=0$,
i.e.\ $\sigma \!=\! \chi_K \id$ with $K:=(0,1)\times(0,L_2)\times(0,L_3)$, 
so that the Maxwell pencil $V(\cdot)$ is non-self-adjoint with piecewise constant coefficients.

In the Appendix we show how Fourier expansion for $E$ together with \cite[Thm.~6]{MR3942228}, or Theorem \ref{sigma-ess} below, can be used to deduce that the essential spectrum of $V$ in the infinite half-cylinder $\Omega$ coincides with the essential spectrum for the infinite cylinder $\R \times (0,L_2) \times (0, L_3)$ and hence satisfies
\begin{equation}
\label{eq:sigmaessV}
\sigma_{e}(V) = (-\infty, -\pi/L] \cup [\pi/L, + \infty) \cup( -{\rm i} \{0,1/2,1\}), \quad L = \max\{L_2, L_3\}.
\end{equation}
Now we truncate the domain to $\Omega_n:= (0,X_n)\times(0,L_2)\times(0,L_3)$, with $X_n\gg 1$
and let $V_n(\cdot)$ be the corresponding Max\-well pencil in \eqref{intro: Maxn}.
It turns out that $\omega\in\C$ is an eigenvalue of $V_n(\cdot)$ if and only if, for some ${\bf n}=(n_2,n_3)\in{\mathbb N}^2$ with $|{\bf n}|>0$,
\begin{equation}
\label{eq:alphabeta}
 \alpha_{\bf n}(\omega)\coth(\alpha_{\bf n}(\omega))  + \beta_{\bf n}(\omega)\coth(\beta_{\bf n}(\omega)(X_n-1)) = 0, \quad n\!\in\!\N;
\end{equation}
the construction of the eigenfunctions is given in the appendix.
Here
\begin{equation}
\label{al-n-beta-n}
\begin{aligned}
&\alpha_{\bf n}(\omega) := \sqrt{\pi^2n_2^2/L_2^2 + \pi^2n_3^2/L_3^2 - \omega(\omega+\I)},
\\[-0.6mm]
&\beta_{\bf n}(\omega) := \sqrt{\pi^2n_2^2/L_2^2 + \pi^2n_3^2/L_3^2 - \omega^2},
\end{aligned}
\end{equation}
where the branch of the square root is taken with non-negative real part.
Note~that there are no square root singularities since $z\mapsto z\coth(z)$ is a meromorphic function.

A little change in the Fourier ansatz allows us to also compute the eigenvalues of the problem in the whole domain $\Omega=(0, \infty) \times (0, L_2) \times (0, L_3)$; the eigenvalue equation for $\omega \in \sigma_p(V)$ becomes
\begin{equation}
\label{true-ev-eq}
\alpha_{\bf n}(\omega)\coth(\alpha_{\bf n}(\omega))  + \beta_{\bf n}(\omega) = 0, \quad {\bf n} = (n_2, n_3) \in \N^2, \,|{\bf n}| > 0,
\end{equation}
which is also obtained from \eqref{eq:alphabeta} in the limit \vspace{-2mm} $X_n\to\infty$.

\begin{figure}[h]
\includegraphics[width= 0.75\textwidth]{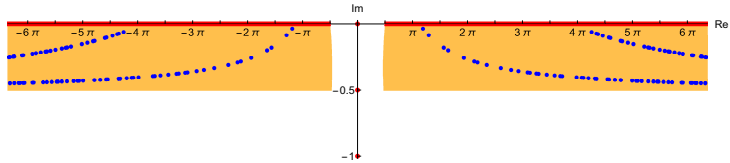}
\vspace{-3mm}
\caption{\small Spectrum of $V$ in $\Omega \!=\! (0, \infty) \!\times\! (0, 1) \!\times\! (0,2)$. The essential spec\-trum is in red, the eigenvalues in blue.
The yellow region is the enclosure in Theorem \ref{thm:spec-encl} for the eigenvalues away from ${\rm i}\R$ and $\R$.}
\label{fig: infinite cylinder}
\vspace{-2mm}
\end{figure}

\begin{figure}[h]
\includegraphics[width= 0.75 \textwidth]{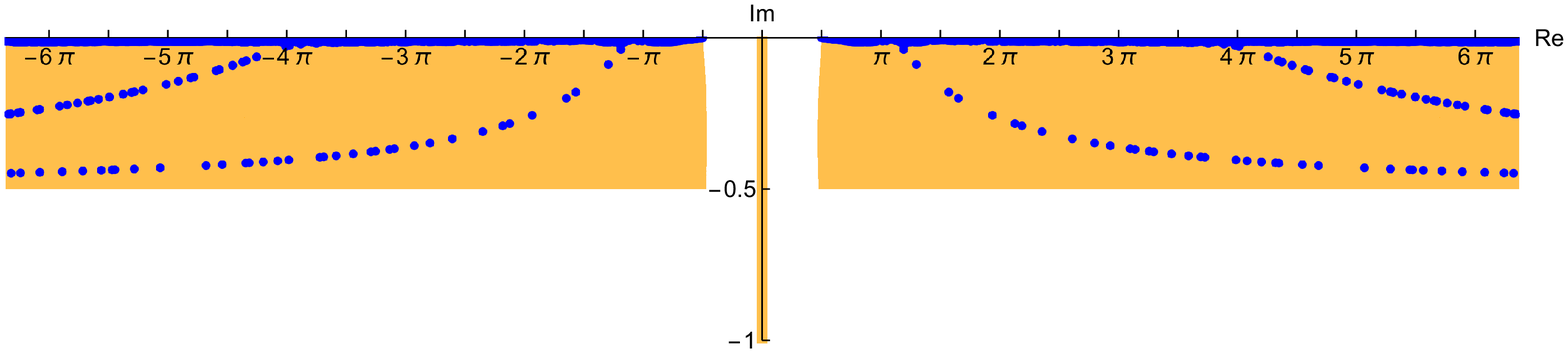}
\\[\smallskipamount]
\includegraphics[width= 0.75\textwidth]{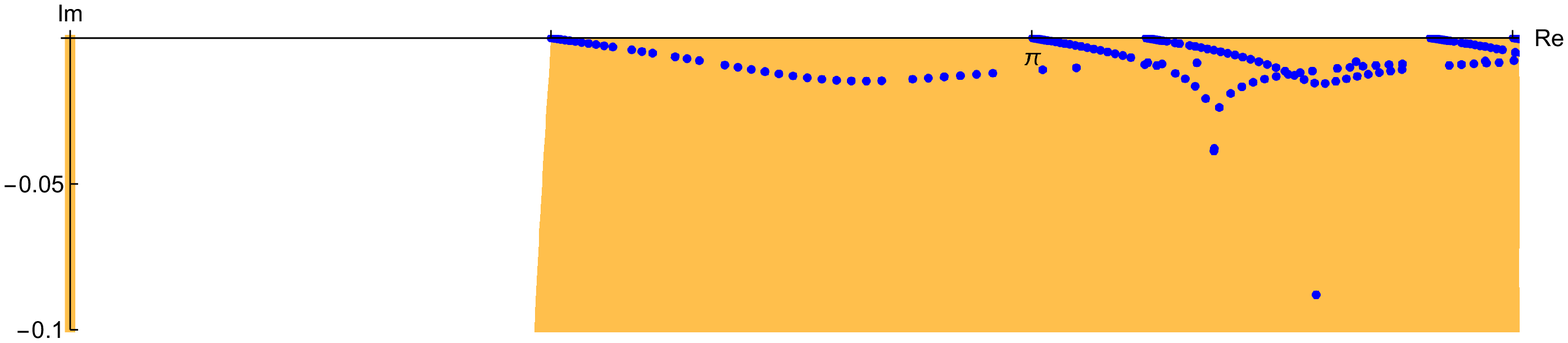}
\vspace{-2mm}
\caption{\small Eigenvalues of $V$ in $\Omega_{50} = (0,50) \times (0, 1) \times (0,2)$, $\eps = \mu = \id$, $\sigma(x) = \chi_{(0,1)}(x_1)$.
The yellow area is the spectral enclosure given by Theorem \ref{thm:spec-encl}. The picture below is a zoom in the area
$[0, \frac{3}{2}\pi] + \I[-0.1,0]$.}
\label{fig: waveguide}
\vspace{-2mm}
\end{figure}
%
The solutions to equations \eqref{true-ev-eq} and  \eqref{eq:alphabeta} can be plotted using a standard computational routine, see Figures~\ref{fig: infinite cylinder} and \ref{fig: waveguide}.
There are many isolated eigenvalues in the region $\R \times -\I [0,1/2]$ that seem to lie along determined curves, see Figure~\ref{fig: infinite cylinder}. Let us give a brief idea of what these curves are. Provided that $\alpha_{\bf n}(\omega) \neq 0$ and $\omega \neq 0$, we rewrite the eigenvalue equation \eqref{true-ev-eq} in the form
\[
\coth(\alpha_{\bf n}(\omega))
= - \frac{\beta_{\bf n}(\omega)}{\alpha_{\bf n}(\omega)}
= -\sqrt{1 + \frac{\I \omega}{\pi^2 n_2^2/L_2^2 + \pi^2 n_3^2/L_3^2 - \omega^2 - \I \omega}}.
\]
We follow an eigenvalue branch $(\omega_k)_k$ which we write as $\omega_k=\mu_k+\I(-1/2+\delta_k)$ with $\mu_k\in\R$ and
$\delta_k\in [0,1/2]$. We show that, if $|\mu_k|\to\infty$, then there exists a subsequence for which $\delta_k\to 0$ as $k\to\infty$.
Without loss of generality, let $\mu_k\to \infty$. We assume that $\liminf_{k\to\infty}\delta_k\!>\!0$ and show that this leads to a contradiction.~Clearly,
$$
\omega_k(\omega_k+\I)=(\omega_k+\I/2)^2+1/4=\mu_k^2+1/4-\delta_k^2+\I 2\mu_k\delta_k;
$$
note that the corresponding ${\bf n}$ for which $\omega_k$ satisfies \eqref{true-ev-eq} may depend on $k$. If we set
$c_k:=\pi^2 n_2^2/L_2^2 + \pi^2 n_3^2/L_3^2>0$, then
\begin{align*}
\alpha_{\bf n}(\omega_k)&=\sqrt{c_k-\mu_k^2-1/4+\delta_k^2-\I 2\mu_k\delta_k},\\
\coth(\alpha_{\bf n}(\omega_k)) &= -\sqrt{1 + \frac{1/2-\delta_k+\I \mu_k}{c_k-\mu_k^2-1/4+\delta_k^2-\I 2\mu_k\delta_k}}.
\end{align*}
If $|c_k-\mu_k^2| \gg \mu_k$, then $\coth(\alpha_{\bf n}(\omega_k))\to -1$ as $k\to\infty$ and
$$\re \alpha_{\bf n}(\omega_k)\sim
\begin{cases}
\sqrt{c_k-\mu_k^2} \ & \text{if }c_k-\mu_k^2\to \infty,\\
\ \frac{\mu_k\delta_k}{\sqrt{|c_k-\mu_k^2|}} \ & \text{if }c_k-\mu_k^2\to -\infty;
\end{cases}
$$
note that $\coth(\alpha_{\bf n}(\omega_k))\to -1$ requires $\re \alpha_{\bf n}(\omega_k)\to -\infty$, but in both cases we have $\re \alpha_{\bf n}(\omega_k)>0$ asymptotically.
It remains to consider the case $|c_k-\mu_k^2|={\rm O}(\mu_k)$.
By the assumption $\liminf_{k\to\infty}\delta_k\neq 0$, there is a subsequence on which $c_k-\mu_k^2-1/4+\delta_k^2-\I 2\mu_k\delta_k\sim C \mu_k$ with $\im C\neq 0$. Then $\coth(\alpha_{\bf n}(\omega_k))\to -\sqrt{1+\I/C}$.
But $\re \alpha_{\bf n}(\omega_k)\sim \re(\sqrt{C}) \sqrt{\mu_k}\to \pm\infty$ implies that $\coth(\alpha_{\bf n}(\omega_k))\to \pm 1$. The obtained contradiction proves $\liminf_{k\to\infty}\delta_k=0$.

For this example, we therefore see that the presence of the compactly supported conductivity generates infinitely many eigenvalues, both in unbounded and bounded domains. These eigenvalues are approximated without spectral pollution due to our result Theorem \ref{theorem: main res}, since in this example $W_e(L_\infty)$ and $ \sigma_e(P_\nabla \cW(\cdot) |_{\nabla \!\dot H^1_0(\Omega)})$ are subsets of the essential spectrum of $V$.

Moreover, one can verify that $\lambda_{\min}^\Omega \!=\! \pi^2/L^2$. This and
the fact that the eigenvalues approach the line $\im\omega = -1/2$ as $|x| \to \infty$ 
show that our spectral enclosure in Theorem~\ref{thm:spec-encl} is sharp.
\end{example}


\begin{example}
In the case of zero conductivity the Maxwell pencil is self-adjoint. Taking the same domain $\Omega$ as in Example \ref{ex:waveguide},
but now with coefficients $\mu \!=\! \id$, $\sigma \!=\!0$ and $\eps \! = \! (1+\delta)\id $ with constant $\delta>0$ if $x_1 \in(0,1)$, else $\eps \! = \id$, i.e.\ $\eps=(1+\delta\,\chi_K)\id$ with $K$ as in Example \ref{ex:waveguide}, we
lose the imaginary part of the essential spectrum from Example \ref{ex:waveguide}, leaving just
\begin{equation}
\label{eq:sigmaessV-sa}
\sigma_{e}(V) = (-\infty, -\pi/L] \cup \{0\} \cup [\pi/L, + \infty), \quad L = \max\{L_2, L_3\}.
\end{equation}
By calculations similar to
those which led to equation (\ref{true-ev-eq}), the eigenvalues are the real zeros of the set of analytic functions
\begin{equation} \label{eq:truemm}
\omega \mapsto \tilde{\alpha}_{\bf n}(\omega)\coth(\tilde{\alpha}_{\bf n}(\omega))  + \beta_{\bf n}(\omega),
\quad {\bf n} = (n_2, n_3) \in \N^2, \,|{\bf n}| > 0,
\end{equation}
in which now $\tilde{\alpha}_{\bf n}(\omega) = \sqrt{\pi^2 n_2^2/L_2^2 + \pi^2 n_3^2/L_3^2 - (1+\delta)\omega^2}$. Taking $L_2 = 1$, $L_3 = 2$ and
$\delta = 10$, we have $\sigma_{e}(V) = (-\infty,-\pi/2]\cup \{0\} \cup [\pi/2,+\infty)$. Elementary numerics show that the gap $(-\pi/2,\pi/2)$ contains
four eigenvalues, given approximately by $\pm 1.4622$ (both simple) and $\pm 1.5643$ (both multiplicity 2). These eigenvalues can be approximated
without pollution using a domain truncation method: this follows immediately from Theorem \ref{theorem: main res}, by verifying that
$\lambda_{e,\min}^\Omega = \pi^2/4$ and since $ \sigma_e(P_\nabla \cW(\cdot) |_{\nabla \!\dot H^1_0(\Omega)})=\{0\}\subset\sigma_e(V)$.
 It may also be seen from the fact that, just as in Example \ref{ex:waveguide}, the functions (\ref{eq:truemm}), whose zeros
are the eigenvalues, are the locally uniform limits as $n\to\infty$ of the functions
\[ \omega \mapsto
 \tilde{\alpha}_{\bf n}(\omega)\coth(\tilde{\alpha}_{\bf n}(\omega))  + \beta_{\bf n}(\omega)\coth(\beta_{\bf n}(\omega)(X_n-1)),
\]
whose zeros are the eigenvalues for the truncated domains. Thus we have a total absence of spectral pollution in this self-adjoint example
despite the fact that, by \cite[Thm.\ 3.8]{MR4083777}, it has $W_e(V) = \R$.
\end{example}


\section{Proofs of the spectral enclosure result and resolvent estimate}
\label{sec:proof1}

In this section we prove the spectral enclosure in Theorem~\ref{thm:spec-encl} and the resolvent estimate in Theorem \ref{thm:res-est-V}. We also
show some auxiliary results that are used for the spectral pollution result.

Since $\eps$ and $\mu$ are bounded and uniformly positive, the linear Maxwell pencil $V(\cdot)$ in \eqref{eq:Vdef}
admits the factorisation
\begin{equation}
\label{Vfact1}
 V(\omega)= \begin{pmatrix} \eps^{1/2} & 0 \\ 0 & \mu^{1/2}\end{pmatrix}(\cA-\omega I) \begin{pmatrix} \eps^{1/2} & 0 \\ 0 & \mu^{1/2}\end{pmatrix},
\end{equation}
in which
\begin{equation}
\begin{aligned}
\label{def:calA}
 & \cA  :=
  \begin{pmatrix}
	-\I \epsilon^{-\frac 12} \sigma \epsilon^{-\frac 12}  & - \I \epsilon^{-\frac 12} \curl \mu^{-1/2} \\
	\I \mu^{-1/2} {\curl_0} \epsilon^{-\frac 12} &  0
	\end{pmatrix},\\
& \dom (\cA) := \eps^{1/2} H_0(\curl, \Omega) \oplus \mu^{1/2}H(\curl, \Omega).	
\end{aligned}
\end{equation}

\begin{proof}[Proof of Theorem {\rm \ref{thm:spec-encl}.}]
Since the matrix multiplication operators $\eps$ and $\mu$ are bound\-ed and uniformly positive, $V(\omega)$ is bijective if and only if so is ${\mathcal A}-\omega$, and hence $\sigma(V)=\sigma(\cA)$. Observe that
\begin{align}
\label{eq:cA}
  \cA-\omega  = \begin{pmatrix}
	-\I \eps^{-1/2} \sigma \eps^{-1/2} -\omega \!&\! \I \eps^{-1/2} \curl \mu^{-1/2} \\
	-\I \mu^{-1/2} \curl_0 \eps^{-1/2} \!&\! -\omega
	\end{pmatrix}
	=:
	\begin{pmatrix}
	-\I Q  & B \\ B^* & 0
	\end{pmatrix} - \omega; 
\end{align}
note that $(\mu^{-1/2} \curl_0 \eps^{-1/2} )^* = \eps^{-1/2} \curl \mu^{-1/2} $ since $\mu^{-1/2}$ is bounded and $\eps^{-1/2}$ is bounded
with range equal to the whole space, see \cite{MR230116}.
Since $\cA$ is a bounded per\-turbation of the self-adjoint off-diagonal part of $\cA$, it is obvious that both the upper and lower half-plane contain at least one point of the resolvent set of $\cA$. 
Hence it suffices to prove the claimed enclosures 
for the approximate point spectrum~$\sigma_{app}(\cA)$.

So let $\omega \!\in\! \sigma_{app}(\cA)$. Then there exists a sequence $((f_n,g_n)^t)_{n\in\N} \!\subset\!  \dom(B^*) \oplus \dom(B)$,
 $\norma{f_n}^2 + \norma{g_n}^2 = 1$, with
\begin{align}
  (-\I Q - \omega) f_n + B g_n & =: h_n \to 0,  \quad n\to\infty, \label{proof:eigeneq1} \\
  B^*f_n - \omega g_n & =: k_n \to 0, \quad n\to\infty. \label{proof:eigeneq2}
\end{align}
If $\omega\!=\!0$, there is nothing to show. Hence we can suppose that $\omega \!\ne\! 0$.
In this case $f_n \!\ne\! 0$ for sufficiently large $n\!\in\!\N$ since otherwise \eqref{proof:eigeneq2} would imply the contradiction $g_n\!\to\! 0$, $n\!\to\!\infty$; hence, without loss of generality we can assume that $f_n \!\ne\! 0$, $n\!\in\!\N$.

If we decompose $g_n = g_n^1 + g_n^2$ with $g_n^1 \in (\ker B)^\perp$, $g_n^2 \in \ker B = (\ran B^*)^\perp$, then $\norma{g_n^1}^2+\norma{g_n^2}^2=\norma{g_n}^2\le 1$, $n\in\N$. Now we take the scalar products with $g_n^1$ and $g_{n}^2$, respectively, in \eqref{proof:eigeneq2}, to conclude \vspace{-0.5mm} that
\begin{align}
   \langle B^*f_n,g_n^1\rangle - \omega \langle g_n^1,g_n^1\rangle & = \langle k_n,g_n^1\rangle \to 0,  \quad n\to\infty, \label{proof:eqnum2}\\
                  - \omega \langle g_n^2,g_n^2\rangle & = \langle k_n,g_n^2\rangle \to 0,  \quad n\to\infty. \label{proof:eigeneq2a}
\end{align}
Taking the scalar product with $f_n$ in \eqref{proof:eigeneq1}, we arrive at
\begin{equation}
\label{proof: eqnum1}
 \langle (-\I Q - \omega) f_n, f_n \rangle + \langle Bg_n^1, f_n\rangle  = \langle h_n,f_n\rangle \to 0, \quad n\to\infty.
\end{equation}
If we subtract the real part of \eqref{proof: eqnum1} from the real part of \eqref{proof:eqnum2}, it follows that
\[
   - \re \omega \norma{g_n^1}^2 - \re \langle (-\I Q - \omega) f_n, f_n \rangle = \re \big( \langle k_n,g_n^1\rangle-\langle h_n,f_n\rangle \big)
	 \to 0, \quad n\to\infty.
\]
Since $Q=\eps^{-1/2} \sigma \eps^{-1/2}$ is a self-adjoint matrix multiplication operator, this implies
\begin{align}
\label{proof:eqnum3}
  \re \omega \big( \norma{f_n}^2 - \norma{g_n^1}^2 \big) \to 0, \quad n\to\infty.
\end{align}
If we add the imaginary parts of \eqref{proof:eqnum2} and \eqref{proof: eqnum1}, we obtain
\[
   -\im \omega \norma{g_n^1}^2 + \im \langle (-\I Q - \omega) f_n, f_n\rangle = \im \big( \langle k_n,g_n^1\rangle+\langle h_n,f_n\rangle \big)\to 0, \quad n\to\infty,
\]
and   hence
\begin{align}
\label{proof:eqnum4}
   (\im \omega) \big(\norma{g_n^1}^2 + \norma{f_n}^2 \big) + \langle Qf_n,f_n\rangle  \to 0, \quad n\to\infty.
\end{align}
Since $\norma{f_n}^2 + \norma{g_n^1}^2 + \norma{g_n^2}^2 =\norma{f_n}^2 + \norma{g_n}^2 = 1$ and $\norma{g_n^2}^2 \to 0$, $n\to\infty$, by
\eqref{proof:eigeneq2a}, we have $\norma{f_n}^2 + \norma{g_n^1}^2\to 1$, $n\to\infty$; hence we can assume without loss of generality that
$\norma{f_n}^2 + \norma{g_n^1}^2 \ge c_1 >0$ with $c_1\in(0,1]$.

Since $\omega\ne 0$, either \eqref{proof:eqnum3}  or \eqref{proof:eqnum4} shows that $f_n \to 0$, $n\to\infty$, implies the contradiction $\norma{g_n^1} \to 0$, $n\to\infty$. Hence, if $\omega \ne 0$, we can assume without loss of generality that $\norma{f_n} \ge c_2 >0$ with $c_2\in(0,1]$. Then \eqref{proof:eqnum4} can be equivalently written \vspace{-2mm} as
\begin{align}
\label{proof:eqnum5}
  \underbrace{\frac{\norma{f_n}^2}{\norma{g_n^1}^2 + \norma{f_n}^2}}_{\in [0,1]} \underbrace{\frac{\langle Qf_n,f_n\rangle }{\norma{f_n}^2}}_{\in W(Q)}
	\to - \im \omega, \quad n\to \infty.
\end{align}
Since $Q$ is self-adjoint, its numerical range $W(Q)\!:=\!\{\langle Qf,f\rangle : f\in L^2(\Omega)^3, \|f\|=1\}$ 
satisfies $W(Q) \!=\! {\rm conv} \,\sigma(Q) \!\subset\! \big[0,\frac{\sigma_{\max}}{\eps_{\min}}\big]$ \vspace{-1mm} and
thus \eqref{proof:eqnum5} implies
\[
  \im \omega \in - {\rm conv}\, \big( \overline{W(Q) \cup \{0\} \big)} = \Big[-\frac{\sigma_{\max}}{\eps_{\min}}, 0 \Big],
\]
which proves the claimed estimate on the imaginary axis.

Now suppose that $\omega \in \C\setminus \I\R$, i.e.\ $\re \omega \neq 0$. Then \eqref{proof:eqnum3} implies that
\begin{equation}
\label{proof: equalitynorms}
   \norma{f_n}^2 - \norma{g_n^1}^2 \to 0, \quad n\to\infty.
\end{equation}
Noting that $\big(\norma{f_n}^2 + \norma{g_n^1}^2\big) - 2 \norma{f_n}^2 = \norma{g_n^1}^2 - \norma{f_n}^2 \to 0$, $n\to \infty$
(due to \eqref{proof: equalitynorms}) and using this in \eqref{proof:eqnum4}, we obtain that
\begin{align}
\label{eq:last-but-one}
   2 \im \omega \norma{f_n}^2  + \langle Qf_n,f_n\rangle \to 0, \quad n\to\infty,
\end{align}
and \vspace{-1.5mm} hence
\begin{align*}
   - \frac 12 \underbrace{\frac {\langle Qf_n,f_n\rangle }{\norma{f_n}^2}}_{\in W(Q)} \to \im \omega, \quad n\to\infty.
\\[-9mm]	
\end{align*}
This proves that
\begin{align}
\label{eq:last}
  \omega \in \C\setminus \I\R \ \implies \ \im \omega \in -\frac 12\overline{W(Q)} \subset \Big[-\frac 12 \frac{\sigma_{\max}}{\eps_{\min}}, - \frac 12 \frac{\sigma_{\min}^\Omega}{\eps_{\max}} \Big].
\end{align}

In order to prove the second inequality for $\omega \in \C\setminus \I\R$, we use the \emph{reduced minimum modulus} of a closed linear operator $T$, defined by
\[
   \gamma(T) := \inf_{x\in \dom T} \frac {\|Tx\|}{{\rm dist}(x,\ker T)},
\]
see e.g.\ \cite[Thm.\ IV.5.2, p.\ 231]{MR1335452}. Note that $\gamma(T) > 0$ if and only if $\ran T$ is closed;  in this case
$\gamma(T) = \norma{T^+}^{-1}$ where $T^+$ is the \emph{Moore-Penrose inverse} of $T$,
$\gamma(T)=\gamma(T^*)$, see \cite[Cor.\ IV.1.9]{MR0200692}, and, if $T\not\equiv 0$,
\begin{align}
\label{eq:rmm-spec}
  \gamma(T)^2 = \min \big( \sigma(T^*T) \setminus \{0\} \big) = \min \sigma( T^*T|_{\dom (T^*T) \cap(\ker T)^\perp}),
\end{align}
comp.\ \cite{MR803833} for the bounded case. In the unbounded case, $T^*T$ is self-adjoint
and its dense domain $\dom (T^*T)$ is a core for $T$, see \cite[Thm.\ V.3.24]{MR1335452}.~Hence
\begin{align*}
   \gamma(T)^2 &= \inf_{x\in \dom T \cap (\ker T)^\perp} \frac {\|Tx\|^2}{\|x\|^2}
	 = \inf_{x\in \dom T^*T \cap (\ker T)^\perp} \frac {\|Tx\|^2}{\|x\|} \\
		&= \inf_{x\in \dom T^*T \cap (\ker T)^\perp} \frac {(T^*Tx,x)}{\|x\|^2}
	 = \min \sigma( T^*T|_{ \dom (T^*T) \cap(\ker T)^\perp}).
	\vspace{-1cm}
\end{align*}%
For $B\!=\!\I \eps^{-1/2} \curl \mu^{-1/2}$, we have $\dom B\!=\!\mu^{1/2} H(\curl,\Omega)$, $\ker B \!=\! \mu^{1/2} \ker \curl$ and thus
\begin{align*}
\gamma(B) 
&= \!\! \inf_{x\in \mu^{1/2} H(\curl,\Omega)} \frac{\|\eps^{-1/2}\curl \mu^{-1/2}x\|}{{\rm dist}(x,\mu^{1/2}\ker\curl)}
= \!\! \inf_{u\in H(\curl,\Omega)} \frac{\|\eps^{-1/2}\curl u\|}{{\rm dist}(\mu^{1/2}u,\mu^{1/2}\ker\curl)} \\
& \ge \frac 1{{\eps_{\max}}^{\!\!\!\!\!1/2}} \inf_{u\in H(\curl,\Omega)} \frac{\|\curl u\|}{{\rm dist}(\mu^{1/2}u,\mu^{1/2}\ker\curl)}\\
& \ge \frac 1{{\eps_{\max}}^{\!\!\!\!\!1/2}{\mu_{\max}}^{\!\!\!\!\!1/2}} \inf_{u\in H(\curl,\Omega)}
\frac{\|\curl u\|}{{\rm dist}(u,\ker\curl)} \\
& = \frac 1{{\eps_{\max}}^{\!\!\!\!\!1/2}{\mu_{\max}}^{\!\!\!\!\!1/2}} \, \gamma(\curl)
= \frac 1{{\eps_{\max}}^{\!\!\!\!\!1/2}{\mu_{\max}}^{\!\!\!\!\!1/2}} \, \gamma(\curl_0).
\end{align*}
Here, we have used $\gamma(T)=\gamma(T^*)$ to replace $\curl$ by $\curl_0$ at the last step. Also, in the second estimate,
we have used the equality
\[
   {\rm dist}\big(\mu^{1/2}u,\mu^{1/2}\ker\curl\big)
	= \inf\limits_{y\in \ker \curl}\big\|\mu^{1/2}u - \mu^{1/2}y\big\| \le {\mu_{\max}}^{\!\!\!\!\!1/2} {\rm dist}(u,\ker\curl).
\]
If $\lambda_{\min}^\Omega=0$, then $\ran \curl_0$ is not closed and hence $\gamma(\curl_0)=0$.
If $\lambda_{\min}^\Omega>0$, then $\ran \curl_0$ is closed and thus $(\ker \curl_0)^\perp = \ran \curl \subset H(\Div 0,\Omega)$.
Hence, by \eqref{eq:rmm-spec}, in both cases, it follows \vspace{-1mm}that
\begin{align}
\label{eq:est-gamma}
  \gamma(B) \ge \frac 1{{\eps_{\max}}^{\!\!\!\!\!1/2}{\mu_{\max}}^{\!\!\!\!\!1/2}} {\lambda_{\min}^\Omega}^{\!\!\!\!\!1/2}.
\\[-8mm]
\nonumber
\end{align}

Now we can estimate
\begin{align*}
  \gamma(B)^2 \norma{f_n}^2 & \!\!=\! \gamma(B)^2 \norma{g_{n}^1}^2 \!-\! \gamma(B)^2 \big(\norma{g_{n}^1}^2 \!-\! \norma{f_n}^2 \big)
  \!\le\! \| B g_n^1 \|^2 \!-\! \gamma(B)^2 \big(\norma{g_{n}^1}^2 \!-\! \norma{f_n}^2 \big) 
\end{align*}
and further, since $g_n^2 \in \ker B$ and $\|f_n\| \ge c_2$ because $\omega \ne 0$,
\begin{align}
\label{proof: ineq red modulus}
  \hspace{-2mm} 
	\gamma(B)^2 &
	\le \frac{\|Bg_n\!-\!(-\I Q\!-\!\omega)f_n\|^2}{\|f_n\|^2}  
	\!+\! \frac{\|(-\I Q\!-\!\omega)f_n\|^2}{\|f_n\|^2}
	\!-\! \gamma(B)^2 \frac{\norma{g_{n}^1}^2 \!-\! \norma{f_n}^2}{\|f_n\|^2}. \ \
\end{align}
For the middle term on the right-hand side we have
\begin{align}
\label{proof:ineq_M_gamma-2a}
	\frac{\|(-\I Q-\omega)f_n\|^2}{\|f_n\|^2} = \frac{\|Qf_n\|^2}{\|f_n\|^2} + 2 \im \omega \frac{\langle Qf_n,f_n\rangle }{\|f_n\|^2} + |\omega|^2.
\end{align}
Using that $Q$ is self-adjoint, we can estimate
\begin{align}
\label{proof:ineq_M_gamma-2}
  \frac{\| Qf_n \|^2}{\|f_n\|^2} = \frac{\langle Qf_n,Qf_n\rangle }{\|f_n\|^2} \le \|Q\| \frac{\langle Qf_n,f_n\rangle}{\|f_n\|^2}, \quad n\in\N.
\end{align}
Altogether, by \eqref{proof: ineq red modulus}, \eqref{proof:ineq_M_gamma-2a}, \eqref{proof:ineq_M_gamma-2} and since $\im \omega \le 0$ by \eqref{eq:last}, we arrive at
\begin{align*}
  \gamma(B)^2 	\!\le\! \big( \|Q\| \!-\! 2 |\im \omega| \big) \frac{\langle Qf_n,f_n\rangle }{\|f_n\|^2} \!+\! |\omega|^2
	& \!+\! \frac{\|Bg_n\!-\!(-\I Q\!-\!\omega)f_n\|^2}{\|f_n\|^2} \\
	&\!-\! \gamma(B)^2 \frac{\norma{g_{n}^1}^2\!-\! \norma{f_n}^2}{\|f_n\|^2}.
\end{align*}
If we use \eqref{eq:last-but-one} and that by \eqref{proof:eigeneq1}and \eqref{proof: equalitynorms}, the last two terms tend to $0$, together with
$\|Q\| =\| \eps^{-1/2}\sigma \eps^{-1/2}\| \le \frac{\sigma_{\max}}{\eps_{\min}}$, we obtain
\[
  \gamma(B)^2 \le \big( \|Q\| - 2 |\im \omega| \big) 2 |\im \omega|  + |\omega|^2 = 2 \|Q\| |\im \omega| - 3 |\im \omega|^2 + |\re\omega|^2.
\]
Now the remaining claimed inequality follows from \eqref{eq:est-gamma}.
\end{proof}

The following remark details the three different possible shapes of the spectral enclosure near the imaginary axis and the corresponding thresholds of $\lambda_{\min}^\Omega$.

\begin{rem}
\label{rem:spec-encl}
Theorem \ref{thm:spec-encl} shows that $\sigma(V) \setminus \I\R$ cannot approach  $\sigma(V) \cap \I\R \subset$ $\I\Big[\!\!-\!\!\tfrac{\sigma_{\max}}{\eps_{\min}}\!,0 \Big]$ in the lower half $\I\Big[\!\!-\!\!\tfrac{\sigma_{\max}}{\eps_{\min}},\!-\!\tfrac 12\tfrac{\sigma_{\max}}{\eps_{\min}}\Big]$
and that there are 3 thresholds \vspace{-1mm} of $\lambda_{\min}^\Omega$
for where $\sigma(V) \!\setminus\! \I\R$ may approach the upper half $\I\Big[\!-\!\frac 12\frac{\sigma_{\max}}{\eps_{\min}}\!,0 \Big]$,~see Figure~\ref{fig:enclosure}:

 \hspace{-4mm}i)  if $\lambda_{\min}^\Omega\!>\!0$, then $\sigma(V) \!\setminus\!\I\R$ does not approach $\I\Big[\!\!-\!\!\frac 12\frac{\sigma_{\max}}{\eps_{\min}}\!,0 \Big]$ near $0$;

 \hspace{-4mm}ii) if $\lambda_{\min}^\Omega\!>\!\frac 14 \frac{\!\sigma_{\max}^2 \eps_{\max} \mu_{\max}\!}{\eps_{\min}^2}\!$, then
$\sigma(V) \!\setminus \I\R$ does not approach $\I\Big[\!-\!\frac 12\frac{\sigma_{\max}}{\eps_{\min}}\!,\!0 \Big]$\,near~$-\I \frac 12\frac{\sigma_{\max}}{\eps_{\min}}$;

 \hspace{-4mm}iii) \!if $\lambda_{\min}^\Omega\!>\!\frac 13 \frac{\sigma_{\max}^2 \eps_{\max} \mu_{\max}}{\eps_{\min}^2}$, then
$\sigma(V) \!\setminus\! \I\R$ does not approach $\I\Big[\!-\!\frac 12\frac{\sigma_{\max}}{\eps_{\min}},0 \Big]$ at all.
\end{rem}

The following special case in Theorem \ref{thm:spec-encl}  of constant matrix functions  $\eps$, $\sigma$, but still varying $\mu$,
is useful e.g.\ for `limiting problems at $\infty$' if they exist.

\begin{corollary}
\label{cor:spec-encl}
If the matrix functions $\eps$, $\sigma$ are constant multiples of the identity,
$\eps\!\equiv\! \eps_\infty \id \!>\!0$, $\sigma\!\equiv\! \sigma_\infty \id \!\ge\! 0$, then $\eps_{\min} \!=\! \eps_{\max} \!=\! \eps_\infty$, $\sigma_{\min} \!=\! \sigma_{\max} \!=\! \sigma_\infty$ \vspace{-1mm} and~thus
\begin{align*}
  \sigma(V) \subset \I\,
	\Big[\!-\!\frac{\sigma_\infty}{\eps_\infty},0 \Big] \cup 	
	\Big\{ \omega\!\in\! \C \!\setminus\! \I\R: \im \omega = - \frac 12 \frac{\sigma_\infty}{\eps_\infty}, \
(\re \omega)^2 \!+\! \frac 14  \frac{\sigma_\infty^2}{\eps_\infty^2} \!\geq\! \frac{\lambda_{\min}^\Omega}{\eps_\infty\mu_{\max}}
 \Big\};
\end{align*}%
in particular, $\sigma(V) \cap \I\R \!\subset\! \I \big[\!-\!\frac{\sigma_\infty}{\eps_\infty},0 \big]$ is isolated from~$\sigma(V) \setminus \I\R$
if $\lambda_{\min}^\Omega \!>\! \frac 14 \frac{\sigma_\infty \mu_{\max}}{\eps_\infty}$.
\end{corollary}

Next we prove Theorem \ref{thm:res-est-V} providing a resolvent norm estimate of $V(\cdot)$.

\begin{proof}[Proof of Theorem {\rm \ref{thm:res-est-V}.}]
Let $\omega\!\in\!\C\!\setminus\!\I\R$, $\im \omega \!<\!  -  \frac 12 \frac {\sigma_{\max}}{\eps_{\min}}$
or $\omega \in \I\R$, $\im\omega<-\frac {\sigma_{\max}}{\eps_{\min}}$. Then $\omega \!\in\! \rho(V)=\rho(\cA)$ by Theorem \ref{thm:spec-encl}~and
$\|V(\omega)^{-1}\| \!\le\! \frac 1{\min\{\eps_{\min},\mu_{\min}\}} \|(\cA\!-\!\omega)^{-1}\|$ due to the factorisation \eqref{Vfact1} where $\cA$ is the oper\-ator matrix in \eqref{eq:cA}.
In order to estimate the resolvent of $\cA$, we continue to use the notation $Q\!=\!\eps^{-1/2}\sigma\eps^{-1/2}$, $B\!=\!\eps^{-1/2}\curl\mu^{-1/2}$ introduced in  \eqref{eq:cA}.

Since $\cA$ is a self-adjoint operator perturbed by the bounded operator $\diag(-\I Q,0)$ and $W(Q)\!\subset\![0,q_{\max}]$ with
$q_{\max}\!:=\!\frac {\sigma_{\max}}{\eps_{\min}}$, a numerical range argument  for $\cA$ yields the resolvent estimate $\|(\cA\!-\!\omega)^{-1} \|\le \frac 1 {|\im\omega|-q_{\max}}$ for all $\omega \in \C$, $\im\omega<-q_{\max}$.

Now let  $\omega\!=\!x\!+\!\I y$  with $x\!>\!0$, $y\!<\! -\frac {q_{\max}}2$; the proof is analogous if $x\!<\!0$.
Let $\varphi\in (0,\frac\pi 2)$ be the argument of $x+\I \frac{q_{\max}}2$.
Let $\mathcal B=\diag(\e^{-\I \varphi},\e^{\I\varphi})$ in $L^2(\Omega)^3\oplus L^2(\Omega)^3$, and let $S\!:=\!\{z\in\C: {\rm arg}(z-\gamma)\in (-\pi+\varphi,-\varphi)\}$
be the open sector with vertex $\gamma\!:=\!x-\I q_{\max}/2$ and semi-angle $\frac \pi 2 \!-\! \varphi$. Note that $\omega\in S$.

We claim that $S\cap W(\mathcal B\cA,\mathcal B)\!=\!\emptyset$,
where $W(\mathcal B\cA,\mathcal B)\!:=\!\{z\!\in\!\C:\,0\!\in\!\overline{W(\mathcal B\cA\!-\!z\mathcal B)}\}$.
Then \cite[Thm.\ 4.1 ii)]{BM} implies $S\subset \rho(\cA)$ \vspace{-1mm} with
$$
 \|(\cA-z)^{-1}\|\leq \frac{1}{\cos(\varphi)\, \dist(z,\partial S)}, \quad z\in S.
\vspace{-1.5mm}
 $$

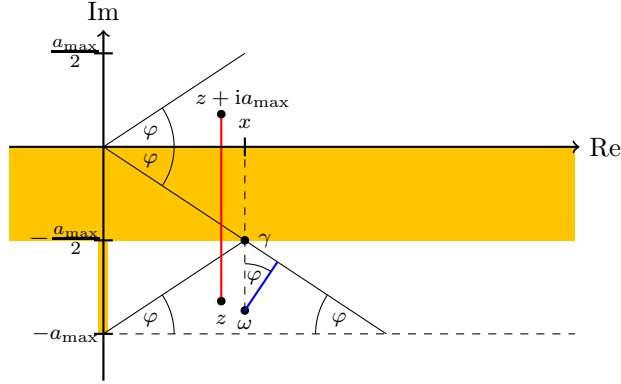
\begin{figure}[htb]
\begin{center}
\begin{tikzpicture}[scale=0.62]
\filldraw[fill=orange!45!yellow, draw=none] (-2,-2) to  (10,-2) to (10,0) to (-2,0);
\filldraw[fill=orange!45!yellow, draw=none] (-0.1,-4) to  (0.1,-4) to (0.1,0) to (-0.1,0);
\draw[thick] (-0.2,-4) -- ++ (0.4,0);
\draw[thick] (-0.2,-2) -- ++ (0.4,0);
\draw[thick] (-0.2,2) -- ++ (0.4,0);
\node at (-0.6,2) {$\frac{a_{\max}}{2}$};
\node at (-0.8,-2) {$-\frac{a_{\max}}{2}$};
\node at (-0.8,-4) {\footnotesize $-a_{\max}$};
\draw[thick] (3,-0.2) -- ++ (0,0.4) node[above] {\footnotesize$x$};
\draw (0,0)--(3,2);
\draw[dashed] (0,-4) to (10,-4);
\draw  (0,0) to (6,-4);
\draw (1.5,0) arc [start angle=0, end angle=33.69, x radius=1.5, y radius=1.5];
\node at (1,0.3) {\footnotesize$\varphi$};
\draw (1.5,0) arc [start angle=0, end angle=-33.69, x radius=1.5, y radius=1.5];
\node at (1,-0.3) {\footnotesize$\varphi$};
\draw (0,-4)--(3,-2);
\draw (1.5,-4) arc [start angle=0, end angle=33.69, x radius=1.5, y radius=1.5];
\node at (1,-3.7) {\footnotesize$\varphi$};
\draw (4.5,-4) arc [start angle=180, end angle=146.31, x radius=1.5, y radius=1.5];
\node at (5,-3.7) {\footnotesize$\varphi$};
\node[fill, draw,circle,label={[label distance=0mm]right:\footnotesize $\gamma$}, inner sep=-1] at (3,-2) {};
\draw[red, thick] (2.5,-3.3) to (2.5, 0.7);
\node[fill, draw,circle,label=below: \footnotesize$z$, inner sep=-1] at (2.5,-3.3) {};
\node[fill, draw,circle,label={[label distance=-1mm]above right:\footnotesize \hspace{-4.3mm}$z+\I a_{\max}$}, inner sep=-1] at (2.5,0.7) {};
\node[fill, draw,circle,label={[label distance=-0.6mm]below:\footnotesize $\omega$}, inner sep=-1] at (3,-3.5) {};
\draw[dashed] (3,-3.5) to (3,0);
\draw[blue, thick] (3,-3.5) to (3.692,-2.462);
\draw (3,-2.5) arc [start angle=90, end angle=56.31, x radius=1, y radius=1];
\node at (3.2,-2.8) {\footnotesize$\varphi$};
\draw[->,thick] (-2,0)--(10.1,0) node[right]{Re};
\draw[->,thick] (0,-5)--(0,2.5) node[above]{Im};
\end{tikzpicture}
\vspace{-5mm}
\end{center}
\caption{\small The geometry in the proof of Theorem \ref{thm:res-est-V}: The blue line measures the distance of $\omega\!=\!x\!+\!\I y$ to $\partial S$; since the lines meet at a right angle, the angle between the blue and the vertical dashed lines is also~$\varphi$.}
\label{fig:proof}
\vspace{-2mm}
\end{figure}

By means of Figure \ref{fig:proof}, one can check that for $z\!=\!\omega$ we have $\dist(z,\partial S)\!=\!(|y|\!-\!\frac{q_{\max}}2) \cos(\varphi)$ and $\cos^2(\varphi)=x^2/(x^2+(\frac{q_{\max}}2)^2)$, which \vspace{-1mm} implies
$$
   \|(\cA-\omega)^{-1}\|\leq \frac 1{|y|-\frac{q_{\max}}2} \Big( 1+ \frac{(\frac{q_{\max}}2)^2}{x^2} \Big).
\vspace{-1mm}
$$

To prove that $S\cap W(\mathcal B\cA,\mathcal B)=\emptyset$, assume that there exists $z\in S\cap W(\mathcal B\cA,\mathcal B)$.
This implies that there  is a normalised sequence $((f_n,g_n)^t)_{n\in\N}\subset\dom(\cA)$ with $\big\langle \mathcal B(\cA-z)(f_n,g_n)^t,(f_n,g_n)^t\big\rangle\to 0$ as $n\!\to\!\infty$; in particular, the sequence also converges to $0$ if we take imaginary parts.
Let $t_n:=\|f_n\|^2\in [0,1]$. Then one can write $\langle Qf_n,f_n\rangle=a_n t_n$ for some $a_n\in W(Q)\subset [0,q_{\max}]$.
We obtain
$$
  \im \big(t_n\e^{-\I\varphi}(\I a_n+z)+(1-t_n)\e^{\I\varphi}z\big) \to 0, \quad n\to\infty.
$$
Note that we take convex linear combinations of points in $\e^{-\I\varphi}(z\!+\!\I [0,q_{\max}])$ and $\{\e^{\I\varphi}z\}$.
Using that $z\!\in\! S$, one can see that both of these compact sets are in the open low\-er complex half-plane, so no sequence of convex linear combinations of points there\-in can converge to the real line.\,This contradiction proves $S\cap W(\mathcal B\cA,\mathcal B)\!=\!\emptyset$.\!\!%
\end{proof}


\section{Spectral relations between $V$ and $\cL$}
\label{sec:V-L}

In this section we establish the intimate relations between the spectra of the linear Maxwell pencil $V(\cdot)$ in the product space $L^2(\Omega,\C^3) \oplus L^2(\Omega,\C^3)$ and of a quadratic operator pencil $\cL$ in the first component $L^2(\Omega,\C^3)$. They will be used~later for our description of the essential spectrum and for our results on spectral pollution for the original Maxwell problem.

The quadratic operator pencil  $\cL$ in $L^2(\Omega)^3$ appears naturally in the matrix re\-presentation of the resolvent of $V(\cdot)$, see Theorem \ref{Prop: spectra S_1 and L}, and is defined \vspace{-0.5mm} by
\begin{equation}
\label{eq:defcL}
\begin{aligned}
\cL(\omega)&:=\curl\mu^{-1}{\curl_0}-\omega(\omega\eps+\I\sigma), \\
 \dom(\cL(\omega))&:=\{E\in H_0(\curl,\Omega):\,\mu^{-1}\curl E\in H(\curl,\Omega)\}.
\end{aligned}
\end{equation}
For studying the relations between the Maxwell pencil $V(\cdot)$ and $\cL$ we require some technical lemmas.


\begin{lemma}
\label{lem:squareroot}
In $ L^2(\Omega)^3$ define the operators $T_0:=\mu^{-1/2}{\rm curl}_0$, $\dom T_0 \!=\! H_0(\curl,\Omega)$, and
$\cW(\omega):=-\omega(\omega\epsilon+\I\sigma)$, $\omega\in\C$.
Then $C_c^{\infty}(\Omega)^3$ is a core of $(T_0^*T_0+I)^{1/2}$, $\dom((T_0^*T_0+I)^{1/2})\!=\!H_0({\rm curl},\Omega)$, and, for all
\vspace{-0.5mm} $\omega \!\in\!\C$,
$$
 \mathcal L(\omega)
 =(T_0^*T_0\!+\!I)^{1/2}\left(I\!+\!(T_0^*T_0\!+\!I)^{-1/2}\big(\cW(\omega)\!-\!I\big)(T_0^*T_0\!+\!I)^{-1/2}\right)(T_0^*T_0\!+\!I)^{1/2};
\vspace{-1mm}
$$
further, for all $t\geq \epsilon_{\min}^{-1/2}$, $\cL(\I t)$ is boundedly \vspace{-1.5mm} invertible,
\begin{equation}
\label{eq:resest}
\begin{aligned}
& \Big\|\Big(I\!+\!(T_0^*T_0+I)^{-1/2}\big(\cW(\I t)\!-\!I\big)(T_0^*T_0\!+\!I)^{-1/2}\Big)^{-1}\Big\|\leq 1,
\\
& \|\cL(\I t)^{-1}\|\leq 1.
\end{aligned}
\end{equation}
\end{lemma}

\begin{proof}
Since $T_0^*T_0$ is self-adjoint and non-negative, the square-root  $(T_0^*T_0\!+\!I)^{1/2} \ge\!I$ is self-adjoint, uniformly positive and boundedly invertible with
\begin{equation}
\label{est-leq1}
 \|(T_0^*T_0+I)^{-1/2}\|\leq 1
\end{equation}
and, e.g.\ by \cite[Prop.\ 3.1.9]{Haa}, $\dom((T_0^*T_0\!+\!I)^{1/2}) \!=\! \dom((T_0^*T_0)^{1/2}) \!=\! \dom (|T_0|)=\dom T_0 \!=\! H_0({\rm curl},\Omega)$. By the second representation theorem \cite[Thm.~VI.2.23]{MR1335452}, a subspace of $ L^2(\Omega)^3$ is a core of $(T_0^*T_0\!+\!I)^{1/2}$ if and only if it is a core of the associated quadratic~form
$$
    {\mathfrak t}[u,v] = \langle T_0u,T_0v\rangle+\langle u,v\rangle, \quad \dom {\mathfrak t} = \dom T_0 = H_0(\curl,\Omega).
$$
Since $C_c^{\infty}(\Omega)^3$ is a core of $T_0$, the first claim follows. The second claim, i.e.\ the operator factorisation of $\mathcal L(\omega)$, is obvious since $\cW(\omega)$ is bounded.

For all $\omega=\I t$ with $t\geq \epsilon_{\min}^{-1/2}$, we have $\cW(\I t)-I \ge 0$ and hence
\begin{align}
\label{eq:ge1}
  I\!+\!(T_0^*T_0\!+\!I)^{-1/2}\big(\cW(\I t)\!-\!I\big)(T_0^*T_0\!+\!I)^{-1/2} \ge I,
\end{align}
which implies the first estimate in \eqref{eq:resest};  the latter and \eqref{est-leq1} yield the last  estimate.
\end{proof}

\vspace{-1mm}

\begin{lemma}
\label{lemma: boundedness}
Let $\omega \!\in\! \rho(\cL)$. Then $\curl_0 \cL(\omega)^{-1}$ is a bounded operator in $ L^2(\Omega)^3$ and  $\cL(\omega)^{-1} \curl$, $\curl_0 \cL(\omega)^{-1}\!\curl$ are closable operators with bounded closures in~$ L^2(\Omega)^3\!$.
\end{lemma}

\begin{proof}
The operator $\curl_0 \cL(\omega)^{-1}$ is bounded in $ L^2(\Omega)^3$ since $\dom(\cL) \!\subset\! \dom(\curl_0)$ $= H_0(\curl, \Omega)$ and $\cL(\omega)$ is a closed operator. Since $(\cL(\omega)^{-1} \curl)^*\!=\!\curl_0 \cL(\omega)^{-*}$ is bounded by the same argument, $\cL(\omega)^{-1} \curl$ has a bounded closure in $ L^2(\Omega)^3$. The boundedness of $\curl_0 \cL(\omega_0)^{-1}\curl$ for $\omega_0=\I t$ with $t\geq \epsilon_{\min}^{-1/2}$  follows from Lemma~\ref{lem:squareroot} using that $\curl_0 (T_0^*T_0+I)^{-1/2}$ and $(T_0^*T_0+I)^{-1/2}\curl$ are bounded.
For a general $\omega\in \rho(\cL)$ the boundedness then follows from
$$\curl_0 \cL(\omega)^{-1}\curl=\curl_0 \cL(\omega_0)^{-1}\curl+\curl_0 \cL(\omega)^{-1}(\cW(\omega)-\cW(\omega_0))\cL(\omega_0)^{-1}\curl$$
since $\cW(\omega)=-\omega(\omega\eps+\I\sigma)$ is a bounded operator.
\end{proof}


\begin{rem}
\label{rem-for-e4}
The claims in Lemma \ref{lemma: boundedness} continue to hold if we replace $\cL(\omega)^{-1}$ by $(\cL(\omega)\!+\!K)^{-1}$ for any bounded operator $K$ in $ L^2(\Omega)^3$. In fact, if we choose $t \!\ge\! \eps_{\min}^{-1/2}(1\!+\!\|K\|)^{1/2}$ in Lemma~\ref{lem:squareroot}, then $\re \,( \cW(\I t) \!+\! K \!-\! I ) \ge t^2 \eps_{\min} I  \!+\! \re K \!-\! I \ge (1\!+\! \|K\|)I - (I - \re K)\ge 0$. Hence the numerical range of the modified operator on the left-hand side of \eqref{eq:ge1} satisfies
\[
 \dist\big(0,W\big( I\!+\!(T_0^*T_0\!+\!I)^{-1/2}\big(\cW(\I t)\!+\!K\!-\!I\big)(T_0^*T_0\!+\!I)^{-1/2} \big)\big) \ge 1
\]
which implies the first estimate in \eqref{eq:resest} with $\cW(\I t)$ replaced by  $\cW(\I t)\!+\!K$. Now the proof of Lemma \ref{lemma: boundedness} can be completed if we note that $\cW(\omega)\!+\!K$ is still bounded.
\end{rem}

\begin{rem}
The resolvent estimate  in Lemma~\ref{lem:squareroot}  for the quadratic operator pencil $\cL$ can be made more precise and extended to the whole region
$\big\{ z\!\in\!\C: \im z \!\notin\! \big[ 0, -\frac 12 \frac{\sigma_{\max}}{\eps_{\min}}, \big] \big\} \setminus \I \big[0, -\frac{\sigma_{\max}}{\eps_{\min}}\big]$, e.g.\ on $\I (0,\infty)$ \vspace{-1.5mm}by
\begin{equation}
    \|\cL(\I t)^{-1} \| \leq \frac 1{t^2 \eps_{\min}}, \quad t \in (0,\infty).
\end{equation}
Since we focus on the Maxwell pencil $V(\cdot)$ in this paper, we restrict ourselves to the properties in Lemmas \ref{lem:squareroot} and \ref{lemma: boundedness} which we need in order to investigate absence of spectral pollution for $V(\cdot)$.
\end{rem}

\begin{theorem}
\label{Prop: spectra S_1 and L}
The Maxwell pencil $V(\cdot)$ in \eqref{eq:Vdef} and the quadratic pencil $\cL$ in \eqref{eq:defcL} \vspace{-2mm} satisfy
\begin{equation}
\label{eq:spec-V-L}
   \rho(V) \setminus \{0\} = \rho(\cL) \setminus \{0\},\quad \sigma(V) \setminus \{0\} = \sigma(\cL) \setminus \{0\},
\end{equation}
and the resolvent of $V(\cdot)$ is given \vspace{-1mm} by
\begin{equation}
\label{eq:res-V}
 V(\omega)^{-1} \!\!=\!
 \begin{pmatrix}
 \omega \cL(\omega)^{-1} & \I \overline{\cL(\omega)^{-1}\!\curl} \mu^{-1} \\
  \!-\I \mu^{-1} \!\curl_0 \cL(\omega)^{-1} \!&\! \omega^{-1}  ( -\mu^{-1} \!\!+\! \mu^{-1} \overline{\curl_0 \cL(\omega)^{-1} \curl} \mu^{-1})
  \end{pmatrix}
\vspace{-1mm}	
\end{equation}
for $\omega\!\in\!\rho(V)$. \vspace{-0.5mm} Moreover,
$$
   \sigma_p(V) \setminus\{0\} = \sigma_p(\cL) \setminus\{0\}, \quad \sigma_c(V) \setminus\{0\} = \sigma_c(\cL) \setminus\{0\}, \quad \sigma_r(V) = \sigma_r(\cL) = \emptyset,
$$
and $\sigma_{ek}(V) = \sigma_{ek}(\cL)$ and $0\in \sigma_{ek}(V)$ for \vspace{-0.2mm} $k =1,2,3,4$.
\end{theorem}

\begin{proof}
Suppose that $\omega\!\in\!\rho(\cL) \setminus \{0\}$. Then, by Lemma \ref{lemma: boundedness}, all entries in the operator matrix on the right-hand side of \eqref{eq:res-V} are bounded and it is easy to check that the latter is a two-sided inverse for $V(\omega)$. This proves $\omega \!\in\!\rho(V) \setminus \{0\}$. Vice~versa, let $\omega \in\rho(V) \setminus \{0\}$. Then, for arbitrary $f \in  L^2(\Omega)^3$, there is a unique $(u,v)^t \!\in\! \dom(V(\omega)) $ $\!=\! H_0(\curl,\Omega) \oplus H(\curl,\Omega)$ such that
$V(\omega)(u,v)^{\rm t} \!=\! (f,0)^{\rm t}$ or, \vspace{-0.5mm}equivalently,
\begin{align*}
 (-\I\sigma - \omega \eps) u  + \I\curl  v = f, \\
 -\I\curl_0 u - \omega \mu v = 0.
\\[-6mm]
\end{align*}
Because $\mu$ is strictly positive and $\omega\ne 0$, we can solve the second equation for $v$ to obtain $v\!=\! -\I\omega^{-1} \mu^{-1} \curl_0 u$.
Since $v\!\in\!  H(\curl,\Omega)$, the latter yields $u \!\in\! \dom \cL(\omega)$ and, inserted in the first \vspace{-0.5mm} equation,
\[
   \omega^{-1} \cL(\omega) u = \big( -\I\sigma - \omega \eps +  \omega^{-1}\curl \mu^{-1} \curl_0 \big) u = f.
\vspace{-0.5mm}	
\]
Since $f \in  L^2(\Omega)^3$ and $u$ was unique, it follows that $\omega \in \rho(\cL) \setminus \{0\}$.

If we set $f\!=\!0$ in the above reasoning, it follows that $\dim \ker V(\omega) \le \dim \ker \cL(\omega)$. Conversely, if $u \in \ker \cL(\omega)$ and we set $v\!:=\! -\I\omega^{-1} \mu^{-1} \curl_0 u$, then the above relations show that $(u,v)^{\rm t} \in \ker V(\omega)$. Altogether this proves that $\dim \ker V(\omega) = \dim \ker \cL(\omega)$ for $\omega\ne 0$ and hence, in particular, the identity for the point spectra.

The claim on the residual spectra follows from \cite[Lemma III.5.4]{EE} since $V(\omega)$ and $\cL(\omega)$ are $J$-self-adjoint with respect to complex conjugation  $J$ in  $L^2(\Omega)^3\oplus L^2(\Omega)^3$ and $ L^2(\Omega)^3\!$, respectively. Then $\sigma_c(V) \!\setminus\! \{0\}
\!=\! (\sigma(V) \!\setminus\! \sigma_p(V)) \!\setminus\! \{0\} \!=\! (\sigma(\cL) \!\setminus\! \sigma_p(\cL)) \!\setminus\! \{0\} \!=\! \sigma_c(\cL) \!\setminus\! \{0\}$.

Due to \cite[Thm.\ IX.1.6]{EE}, the $J$-self-adjointness also implies that all $\sigma_{ek}(V)$, $k=1,2,3,4$, and all $\sigma_{ek}(\cL)$, $k=1,2,3,4$, coincide. The last claim is proved if we show that $\sigma_{ek}(V) = \sigma_{ej}(\cL)$ for any $j$, $k\in \{1,2,3,4\}$. Here we show that $\sigma_{e2}(V) \supset \sigma_{e2}(\cL)$ and $\sigma_{e4}(V) \subset \sigma_{e4}(\cL)$. First we consider $\omega \in \C\setminus \{0\}$.

To show $\sigma_{e2}(V) \!\supset\! \sigma_{e2}(\cL)$, suppose that $\omega \in \sigma_{e2}(\cL) \setminus\{0\}$. Then, by \cite[Thm.\ IX.1.3]{EE} there exists a singular sequence $(u_n)_{n\in\N} \subset \dom \cL(\omega)$ of $\cL(\omega)$ in~$0$, i.e.\ $\|u_n\|=1$, $n\in\N$, $u_n \rightharpoonup 0$ and $\cL(\omega)u_n \to 0$ for $n\to \infty$. If we set $v_n\!:=\! - \I \omega^{-1} \mu^{-1} \curl_0 u_n$, $n\in\N$, then $v_n \!\in\! H(\curl,\Omega)$ and
the sequence with elements $w_n\!:=\!(u_n,v_n)^{\rm t}/\sqrt{\|u_n\|^2\!+\!\|v_n\|^2} $ $\in \dom V(\omega)$ satisfies $\|w_n\|=1$, $n\in\N$, and $V(\omega) w_n = (\omega^{-1} \cL(\omega) u_n, 0)^{\rm t} \to 0$ for $n\to \infty$.
In addition, for any $\omega_0\in\rho(\cL)$,
\begin{align*}
v_n&=-\I \omega^{-1}\mu^{-1}\curl_0 \cL(\omega_0)^{-1}\cL(\omega_0)u_n\\
&=-\I \omega^{-1}\mu^{-1}\curl_0 \cL(\omega_0)^{-1}\left(\cL(\omega)u_n+(\omega(\omega+\I\sigma)-\omega_0(\omega_0\eps+\I\sigma))u_n\right)\rightharpoonup 0;
\end{align*}
here we have used $\cL(\omega)u_n\!\to\! 0$, $u_n\!\rightharpoonup\! 0$ and that $\curl_0\cL(\omega_0)^{-1}$ is bounded by~Lem\-ma~\ref{lemma: boundedness}. Now $\sqrt{\|u_n\|^2\!+\!\|v_n\|^2}\!\geq\! \|u_n\|\!=\!1$ yields $w_n\!\rightharpoonup\! 0$.
This proves~$\omega \!\in\! \sigma_{e2}(\cL) \setminus \{0\}$.

To show $\sigma_{e4}(V) \!\subset\! \sigma_{e4}(\cL)$, assume that $\omega \!\notin\! \sigma_{e4}(\cL) \setminus\{0\}$. Then, by \cite[Thm.~IX.1.4]{EE} there exists a compact operator $K$ in $ L^2(\Omega)^3$ such that $0 \!\in\! \rho (\cL(\omega)+K)$. If we set $\cK\!:=\! {\rm diag} (K,0)$, then $\cK$ is compact in $L^2(\Omega)^3 \oplus L^2(\Omega)^3$ and, using Remark \ref{rem-for-e4}, we conclude that the operator matrix obtained from the right-hand side of \eqref{eq:res-V} by replacing $\cL(\omega)^{-1}$ by $(\cL(\omega)\!+\!K)^{-1}$ is bounded and a two-sided inverse for $V(\omega)+\cK$ and hence $0 \!\in\! \rho(V(\omega)\!+\!\cK)$. Now \cite[Thm.\ IX.1.4]{EE} yields that $0\!\notin\! \sigma_{e4}(V(\omega))$,~as~required.

Finally, it remains to consider $\omega=0$. It is not difficult to see that $V(0)(0,H)^{\rm t}=0$ for all $H\!\in\! \ker \curl$ and hence
$\{0\} \oplus \nabla \dot{H}^1_0(\Omega) \!\subset\! \ker V(0)$. This proves  $0 \!\in\! \sigma_{e2}(V)$. Further,
$\cL(0)= \curl \mu^{-1} \curl_0$ is self-adjoint with $\nabla\dot{H}^1_0(\Omega) \!\subset\! \ker \curl_0\!=\!\ker \cL(0)$ and hence also $0 \in \sigma_{e2}(\cL)$.
\end{proof}

\section{The essential spectrum of the Maxwell problem}
\label{section:5}

In this section we determine the essential spectrum of $V(\cdot)$ via the essential spectrum of the quadratic operator pencil  $\cL$. Here we assume that $\Omega$ is an infinite domain and that $\sigma$, $\mu$, $\eps$ have limits $0$, $ \eps_\infty\!\id$, $\mu_\infty\!\id$ in the sense of~\eqref{eq:coeffs-infty} at infinity, as in Theorem \ref{theorem: main res}; note that $\eps_\infty$, $\mu_\infty >0$ by assumption~\eqref{conditions eps mu sigma}.

To this end, we work in the Helmholtz decomposition $L^2(\Omega)\!=\!\nabla \dot H^1_0(\Omega) \oplus H(\Div 0,\Omega)$, see e.g.\ \cite[Lemma 11]{MR3942228}, and denote by $P_{\ker(\Div)}$ the corresponding orthogonal projection from onto $H(\Div 0,\Omega)$.  We begin with a general result which applies in a wider context.

\begin{prop}
\label{thm: compactness}
Let $m:\Omega\to \C^{3\times 3}$ be a tensor-valued function \vspace{-1mm} with
\begin{align}
\label{eq:limit-gen}
\lim_{R\to\infty}\sup_{\|x\|>R} \|m(x)\|=0.
\\[-7.5mm]
\nonumber
\end{align}
Then $mP_{\ker(\Div)}$ is compact from $(H(\curl, \Omega), \norma{\cdot}_{H(\curl, \Omega)})$ to  $( L^2(\Omega)^3, \norma{\cdot}_{ L^2(\Omega)^3})$.
\end{prop}

\begin{proof}
For any $\delta>0$ we can write $m = m_c + m_\delta$ where $m_\delta$ is a bounded {multiplication} operator with $\|m_\delta\|<\delta$ and
$m_c$ has compact support in some domain $\Omega_R := \Omega\cap B(0,R)$ for sufficiently large $R>0$. We show that $m_cP_{\ker(\Div)}$ is compact for every $\delta > 0$. Since $\norma{m P_{\ker(\Div)} \!-\! m_c P_{\ker(\Div)}}_{\cB(H(\curl,\Omega),  L^2(\Omega)^3)} \leq  \delta$ vanishes as $\delta \to 0$, we deduce that $mP_{\ker(\Div)}$ is the norm limit of the compact operators $m_c P_{\ker(\Div)}$ and hence compact.

Let $\chi_R$ be a smooth cut-off function  with $\chi_R \!=\! 1$ on $ \supp(m_c)\subset \Omega_R$ and $\chi\!=\!0$ outside $\Omega_R$.
Then there exists a constant $C_R>0$ such that, for all \vspace{-1mm} $u\in H(\curl,\Omega)$,
$$
\|(\chi_R P_{\ker(\Div)} u)|_{\Omega_R}\|_{H(\curl, \Omega_R) \cap H(\Div, \Omega_R)}
\leq C_R\|u\|_{H(\curl, \Omega)},
$$
where we use that $\Div (\chi_R P_{\ker(\Div)}u)=\nabla \chi_R\cdot P_{\ker(\Div)}u$ and
$\curl (\chi_R P_{\ker(\Div)}u)=\nabla \chi_R\times P_{\ker(\Div)}u+\chi_R \curl u$ since $\curl P_{\ker(\Div)}u=\curl u$.
The compactness of $m_c P_{\ker(\Div)}$ follows from the compactness of the composition
$$m_c P_{\ker(\Div)} u = m_c \iota(\chi_R P_{\ker(\Div)} u)|_{\Omega_R};$$
here $\iota$ is the compact embedding of ${H(\curl, \Omega_R)} \cap H(\Div, \Omega_R)$ in $L^2(\Omega_R)^3\!$, see~\cite{MR561375}.
\end{proof}

\vspace{-3mm}

\begin{definition}
\label{Lmidef}
We define quadratic pencils of closed operators acting in the Hilbert space 
$H(\Div 0,\Omega)$ equipped with the $ L^2(\Omega)^3$-norm \vspace{-1mm} by
\begin{align*}
&\begin{array}{ll}
L_\mu(\omega) \!:=\!  \curl \mu^{-1} \curl_{0} - \eps_\infty \omega^2 \id,  \\[1mm]
\dom(L_\mu(\omega)) \!:=\! \{u \in H_0(\curl,\Omega){\cap H(\Div 0,\Omega)} : \mu^{-1}\curl u \!\in\! H(\curl,\Omega) \},
\end{array}\hspace{-2mm} & \omega\!\in\!\C,\\[-1mm]
\intertext{
and}
&\begin{array}{l}
\!L_\infty(\omega) \!:=\!  \curl \mu_{\infty}^{-1}\!\curl_{0} -  \eps_\infty\omega^2\id,  \\[1mm]
\!\dom(L_\infty(\omega)) \!:=\! \{u \in H_0(\curl,\Omega) {\cap H(\Div 0,\Omega)}: \curl u \!\in\! H(\curl,\Omega) \},
\end{array}\hspace{-2mm} & \omega\!\in\!\C;
\vspace{-1mm}
\end{align*}
note that $L_\infty$ can be regarded as a special case of $L_\mu$, namely when $\mu\!=\!\mu_\infty\! \id$.
\end{definition}

\begin{lemma}
\label{lemma: boundedness2}
The following are true.
\begin{myenum}
\item[{\rm (i)}]
The operator $L_\mu(\omega)^{-1}\curl$ is closable and bounded from $ L^2(\Omega)^3$ to $H(\Div 0,\Omega)$.
\item[{\rm (ii)}]  For $\omega\!=\!\I t$ with $t\!\geq\! \eps_{\min}^{-1/2}$,
the operator  $\curl_{0}L_\infty(\omega)^{-1}$ is bounded in $H(\Div 0,\Omega)$, and also as an operator from $H(\Div 0,\Omega)$ to $H(\curl,\Omega)$ \vspace{-2mm} with
\begin{equation}
\label{eq:opnorm}
  \| \curl_{0}L_\infty(\omega)^{-1}\|_{\cB(H(\Div 0,\Omega),H(\curl,\Omega))}
  \leq 
\left(\frac{\mu_\infty}{\eps_\infty |\omega|^2}+\mu_\infty^2\right)^{1/2}.
\end{equation}
\end{myenum}
\end{lemma}

\begin{proof}
The boundedness claims follow analogously as in Lemma \ref{lemma: boundedness}, using that $u\in H(\Div 0,\Omega)$
satisfies $\|u\|_{H(\Div,\Omega)}=\|u\|$. It remains to prove \eqref{eq:opnorm}.
Noting that $\|L_\infty(\omega)^{-1}\|\leq 1/(\eps_\infty t^2) \leq 1/(\eps_\infty|\omega^2|)$, we estimate, for $u\in  L^2(\Omega)^3$,
\begin{align*}
\|\curl_{0}L_\infty(\omega)^{-1}u\|^2
&= \mu_\infty \big\langle \curl \mu_\infty^{-1} \curl_0 L_\infty(\omega)^{-1}u,L_\infty(\omega)^{-1}u \big\rangle|\\
&= \mu_\infty \big\langle (I-\eps_\infty t^2 L_\infty(\omega)^{-1})u ,L_\infty(\omega)^{-1}u \big\rangle
\\[-1mm]
&\leq \mu_\infty \big\langle u,L_\infty(\omega)^{-1} u \big\rangle \leq \mu_\infty \frac{1}{\eps_\infty|\omega|^2}\|u\|^2\\[-7mm]
\end{align*}
and, since  $0 \le I-\eps_\infty t^2 L_\infty(\omega)^{-1} \le I$,
\begin{align*}
\|\curl \curl_0L_\infty(\omega)^{-1}\|_{\cB(L^2(\Omega)^3,L^2(\Omega)^3)}
\!&\!=\! \mu_\infty
\!\!\! \sup_{u \in  L^2(\Omega)^3 \atop \|u\|=1}\!\!\! \big\langle (I-\eps_\infty t^2 L_\infty(\omega)^{-1})u ,  u \big\rangle \!\leq\! \mu_\infty\!.
\\[-7mm]
\end{align*}
Together, this implies the claimed resolvent norm bound.
\end{proof}

Note that unless $\mu$ is differentiable, the intersection between the (operator) domains of {the pencils $L_\mu,L_\infty$} could be trivial. Nevertheless they have the same form domain, and the following result holds.

\begin{prop}
\label{thm: difference res}
If $\sigma$, $\eps$ and $\mu$ satisfy the limiting assumption \eqref{eq:coeffs-infty} and
$L_\mu$, $L_\infty$ are as in Definition {\rm \ref{Lmidef}},  then $\sigma_{ek}(L_\mu) \!=\! \sigma_{ek}(L_\infty) \subset \R$
for $k=1,2,3,4,5$, and \vspace{-1mm} hence	
\begin{align*}
   \sigma_{ek}(L_\mu) \setminus \{0\} \!&=\!
	 \Big( \Big( \!-\! \frac 1{ \eps_\infty\mu_\infty}\sigma_{ek}(\curl\curl_0)^{1\!/2} \Big) \!\cup
	   \Big(  \frac 1{ \eps_\infty\mu_\infty} \sigma_{ek}(\curl \curl_0)^{1\!/2}\Big) \Big) \setminus \{0\}.
\\[-6mm]	
\end{align*}
\end{prop}

\begin{proof}
Let $\omega\!\in\!\C$ and set $z\!:=\eps_\infty \omega^2\!$. Then $\omega\!\in\!\sigma_{ek}(L_\mu)$ if and only if $z\!\in\! \sigma_{ek}(C_\mu)$ and
$\omega\!\in\!\sigma_{ek}(L_\infty)$ if and only if $z\!\in\! \sigma_{ek}(C_\infty)$
where $C_\mu\!:=\!\curl \mu^{-1} \curl_0$ and $C_\infty\!:=\!\curl \mu_\infty^{-1} \curl_0$ are self-adjoint in $H(\Div 0,\Omega)$.
Thus it suffices to show that $\sigma_{ek}(C_\mu)=\sigma_{ek}(C_\infty)$ for some $k\!\in\!\{1,2,3,4,5\}$.
Since the associated quadratic forms ${\mathfrak c}_\mu$ and ${\mathfrak c}_\infty$ have the same domain,
$\dom {\mathfrak c}_\mu \!=\! \dom {\mathfrak c}_\infty \!=\! H_0(\curl,\Omega)$, the second resolvent identity takes the~form
\begin{equation}
\label{eq:form-2nd-res.-id}
   (C_\mu\!-\!z)^{-1} \!- (C_\infty\!-\!z)^{-1} \!=\!
   \big( \curl_0 (C_\mu\!-\!\overline{z})^{-1} \big)^{\!*} (\mu_\infty^{-1} \!\!-\mu^{-1}) \curl_0 (C_\infty\!-\!z)^{-1}
\end{equation}
for $z\!\in\! \C\setminus \R$. In fact, for arbitrary $u$, $v \in  L^2(\Omega)^3$ and $z\in\C\setminus \R$, we can write
\begin{align*}
  &\!\big\langle \big( (C_\mu\!-\!z)^{-1} \!- (C_\infty\!-\!z)^{-1} \big) u,v \big\rangle
	\!=\! \big\langle  (C_\mu\!-\!z)^{-1}  u,v \big\rangle - \big\langle u, (C_\infty\!-\!\overline z)^{-1} v \big\rangle\\
	\!&=\! \big\langle  (C_\mu\!-\!z)^{-1}  u, (C_\infty\!-\!\overline z)(C_\infty\!-\!\overline z)^{-1}v \big\rangle
	       - \big\langle  (C_\mu\!-\!z)  (C_\mu\!-\!z)^{-1}  u, (C_\infty\!-\!\overline z)^{-1} v \big\rangle\\
	\!&=\! ( {\mathfrak c}_\infty - {\mathfrak c}_\mu ) \big[	(C_\mu\!-\!z)^{-1}  u, (C_\infty\!-\!\overline z)^{-1} v  \big];
\end{align*}
together with ${\mathfrak c}_\mu \!=\! \langle\mu^{-1} \!\curl_0 \cdot, \curl_0 \cdot\rangle$ and analogously for ${\mathfrak c}_\infty$,
the identity \eqref{eq:form-2nd-res.-id}~follows.
The first factor on the right-hand side of \eqref{eq:form-2nd-res.-id} is bounded since $\dom C_\mu \!\subset\! \dom \curl_0$.
By assumption \eqref{eq:coeffs-infty}, the ten\-sor-valued function $(\mu^{-1}\!-\!\mu_\infty^{-1})\id$ satisfies condition \eqref{eq:limit-gen} of Proposition \ref{thm: compactness} and thus the operator $(\mu^{-1}\!-\!\mu_\infty^{-1})\id P_{\ker\Div}$ is com\-pact from $H(\curl,\Omega)$ to $H(\Div 0,\Omega)\!\subset\! L^2(\Omega)^3$. By Lemma \ref{lemma: boundedness2} (ii), $\curl_0 (C_\infty\!-\!z)^{-1}$ $=\!\curl_0 L_\infty(\omega)^{-1}$ is bounded from $H(\Div 0, \Omega)$ to $H(\curl,\Omega)$. Altogether, we see~that
 \begin{align*}
 (\mu_\infty^{-1} \!\!-\mu^{-1}) \curl_0 (C_\infty\!-\!z)^{-1}  = (\mu_\infty^{-1}\!-\!\mu^{-1})\id P_{\ker\Div}  \curl_0 (C_\infty\!-\!z)^{-1}
\end{align*}
is compact.
Hence, by \eqref{eq:form-2nd-res.-id}, the resolvent difference of $C_\mu$ and $C_\infty$ is compact and, by \cite[Thm.\ IX.2.4]{EE},
$\sigma_{ek}(C_\mu)=\sigma_{ek}(C_\infty)$ follows for all $k=1,2,3,4$, and for $k=5$ since $C_\mu$,$C_\infty$ are self-adjoint.
\end{proof}

Now we can characterise the essential spectrum of the Maxwell pencil $V(\cdot)$ and show that it lies on the real axis and on some bounded purely imaginary interval below $0$.

\begin{theorem}
\label{sigma-ess}
Suppose that $\sigma$, $\eps$ and $\mu$ satisfy the limiting assumption \eqref{eq:coeffs-infty}.
Let $P_\nabla := \id - P_{\ker\Div}$ be the orthogonal projection from $ L^2(\Omega)^3\!=\!\nabla \dot H^1_0(\Omega) \oplus H(\Div 0,\Omega)$ onto $\nabla \dot{H}^1_0(\Omega)$ and recall that $\cW(\omega):=- \omega(\omega\eps + \I \sigma)$, $\omega\!\in\!\C$, in $ L^2(\Omega)^3$.
Then
\[
   \sigma_{ek}(\cL) = \sigma_{ek}(L_\infty) \cup \sigma_{ek}(P_\nabla \cW(\cdot) |_{\nabla \dot H^1_0(\Omega)}), \quad k=1,2,3,4,
\]
with $\sigma_{ek}(L_\infty)\!\subset\! \R$ given in Proposition~{\rm\ref{thm: difference res}} and $ \sigma_{ek}(P_\nabla \cW(\cdot) |_{\nabla \dot H^1_0(\Omega)}) \!\subset\! \I [-\frac{\sigma_{\max}}{\eps_{\min}},0]$.
\end{theorem}

\begin{proof}
Let $\omega\!\in\!\C$.
By Proposition~\ref{thm: compactness},  the operator $M(\omega)\!:=\!(\omega(\omega\epsilon\!+\!\I\sigma)\!-\!\omega^2)P_{\ker(\Div)}$ in $ L^2(\Omega)^3$ is $\curl_0$-compact and hence $T_0$-compact with $T_0=\mu^{-1/2}\curl_0$. Since $\cL(\omega)=T_0^*T_0 + \cW(\omega)$ where $\cW(\omega)=-\omega(\omega\eps+\I \sigma)$, bounded
sequences whose $\cL(\omega)$ graph norms are bounded have bounded $T_0$ graph norms. Hence $M(\omega)$ is $\cL(\omega)$-compact which
yields $\sigma_e(\cL(\omega))=\sigma_e(\cL(\omega)+M(\omega))$.

Since $\nabla \dot{H}^1_0(\Omega)\subset\ker(\curl_0) =\ker T_0$ and hence $T_0 P_\nabla \!=\! P_\nabla T_0^*\!=\!0$,
$\nabla \dot{H}^1_0(\Omega)$ is a reducing subspace for $T_0^*T_0$.%
Therefore the \vspace{-1mm}operator
\begin{align} \nonumber
  \!\cT(\omega)\!:=\! \cL(\omega)\!+\!M(\omega)
	& \!=\! T_0^*T_0\!-\!\omega(\omega\eps\!+\!\I\sigma) (P_\nabla\!+\! P_{\ker\Div})\!+\!(\omega(\omega\epsilon\!+\!\I\sigma)\!-\!\eps_\infty\omega^2)P_{\ker(\Div)} \\
	& \!=\! T_0^*T_0\!-\!\omega(\omega\eps\!+\!\I\sigma) P_\nabla - \omega^2 \eps_\infty P_{\ker(\Div)}
\label{eq:omcT}	
\end{align}
which is a bounded perturbation of $T_0^*T_0$ admits an operator matrix representation with respect to the decomposition
$ L^2(\Omega)^3\!=\! \nabla \dot{H}^1_0(\Omega) \oplus H(\Div0, \Omega)$ given \vspace{-1mm} by
\begin{align}
\cT(\omega) \!&=\!
\begin{pmatrix}
\hspace{5.7mm} P_\nabla \cT(\omega) |_{\nabla \dot H^1_0(\Omega)} \!&\!  \hspace{6.5mm} P_\nabla \cT(\omega) |_{H(\Div0, \Omega)}\\
P_{\ker\Div} \cT(\omega) |_{\nabla \dot H^1_0(\Omega)} \!&\! P_{\ker\Div} \cT(\omega) |_{H(\Div0, \Omega)} \!&\!
\end{pmatrix}
\nonumber
\\
\!&=\!
\begin{pmatrix}
\hspace{5.5mm} P_\nabla (-\!\omega(\omega\eps\!+\!\I\sigma)) |_{\nabla \dot H^1_0(\Omega)}  \!&\! 0 \\
P_{\ker\Div}(-\!\omega(\omega\eps\!+\!\I\sigma))  |_{\nabla \dot H^1_0(\Omega)} \!&\! P_{\ker\Div} (T_0^*T_0\!-\!\omega^2 \eps_\infty) |_{H(\Div0, \Omega)}\\
\end{pmatrix}
\nonumber
\\
\!&=\!
\begin{pmatrix}
\hspace{5.7mm} P_\nabla \cW(\omega) |_{\nabla \dot H^1_0(\Omega)} \!&\!  0  \\
P_{\ker\Div} \cW(\omega) |_{\nabla \dot H^1_0(\Omega)} \!&\!  L_\mu(\omega)
\end{pmatrix}.
\label{eq: A^0}
\\[-7mm] \nonumber
\end{align}
with domain $\dom(\cT(\omega))=\nabla \dot{H}^1_0(\Omega)\oplus \dom(L_\mu(\omega))$.
Apart from $L_\mu(\omega)$, the other two matrix entries in $\cT(\omega)$ are bounded and everywhere defined,
and $\sigma_{e2}(L_\mu(\omega)) = \sigma^*_{e2}(L_\mu(\omega))$.
Thus Theorem \ref{thm: ess spec} in Section \ref{sec: triangular} below and Proposition~\ref{thm: difference res} yield
\vspace{-1mm} that
$$
    \sigma_{e2}(\cT\!(\omega))
		\!=\! \sigma_{e2}(L_\mu(\omega)) \cup \sigma_{e2}(P_\nabla \cW(\omega) |_{\nabla \!\dot H^1_0(\Omega)})
		\!=\! \sigma_{e2}(L_\infty(\omega)) \cup \sigma_{e2}(P_\nabla \cW(\omega) |_{\nabla \!\dot H^1_0(\Omega)})
$$
and hence, since $\omega\in\C$ was arbitrary,
\begin{align*}
\sigma_{e2}(\cL)&=\sigma_{e2}(\cL+M)=\sigma_{e2}(\cT) = \sigma_{e2}(L_\infty) \cup \sigma_{e2}(P_\nabla \cW(\cdot) |_{\nabla \dot H^1_0(\Omega)}).
\qedhere
\\[-6mm]
\end{align*}
\end{proof}

\begin{rem} \label{rem:spec-eq}
Theorem \ref{sigma-ess} generalises \cite[Thm.\ 6]{MR3942228} since we do not suppose \cite[Ass.\ 14]{MR3942228}
on $\Omega$, which requires the subspaces $K_N(\Omega)$ of $H(\Div 0, \Omega)$ and $K_T(\Omega)$ of $H_0(\Div 0, \Omega)$ to be finite dimensional. 
If the latter holds, see \cite[Prop.\ 15]{MR3942228} for a list of sufficient conditions,  
then both Theorem \ref{sigma-ess} and  \cite[Thm.\ 6]{MR3942228} apply and
we obtain the interesting equality
\begin{equation} \label{eq:spec_eq}
\sigma_{ek}\big(P_\nabla \cW(\cdot) |_{\nabla \dot H^1_0(\Omega)}\big) = 
\sigma_{ek}(\Div ( \cW(\cdot) \nabla))
\quad k=1,2,3,4,
\end{equation}%
where $P_\nabla \cW(\omega) |_{\nabla \!\dot H^1_0(\Omega)}$ is a bounded operator in $L^2(\Omega; \C^3)$, while \vspace{-1mm} 
$\Div (\cW(\omega) \nabla) $ is defined as a bounded operator from $\dot{H}^1_0(\Omega)$ to $\dot H^{-1}(\Omega)$ in \cite{MR3942228}.  In fact, \eqref{eq:spec_eq} follows from the identity
\[
\sigma_{ek}(L_\infty) \cup \sigma_{ek}\big(P_\nabla \cW(\cdot) |_{\nabla \dot H^1_0(\Omega)}\big) = \sigma_{ek}(\cL) 
= \sigma_{ek}(V^0_{(\cdot)}) \cup \sigma_{ek}(\Div ( \cW(\cdot) \nabla))
\]%
where $V^0_{(\cdot)}$ is the Maxwell pencil $\I V(\cdot)$ with constant coefficients $\eps_\infty$, $\mu_\infty$ and $\sigma \equiv 0$ defined in \cite[Thm.\ 6]{MR3942228}, if we observe $\sigma_{ek}(V^0_{(\cdot)}) = \{0\} \cup \sigma_{ek}(L_\infty)\subset\!\R$, $k=1,2,3,4$, and that the sets in \eqref{eq:spec_eq} lie on $\I\R$ and both contain $\{0\}$. 

Note that, in concrete examples, identity \eqref{eq:spec_eq} is useful to explicitly determine the purely imaginary part of the essential spectrum of the Maxwell pencil. 
\vspace{-1mm}
In fact, 
\begin{align*}
  \sigma_{ek}(\Div ( \cW(\cdot) \nabla)) 
  &= \{ \omega \in \C: 0 \in  \sigma_{ek}(\Div ( (\omega\eps + \I \sigma) \nabla)) \} \cup \{0\} \\
  &= \{ \I \nu \in \I \R: 0 \in  \sigma_{ek}(\Div ( (\nu\eps + \sigma) \nabla)) \} \cup \{0\}
\end{align*}
and Fredholm properties of operators $\Div (a \nabla )$ with non-definite coefficients $a$ also arise when studying Maxwell equations in media
with dielectric permittivity and/or magnetic permeability, see e.g.\ \cite{MR3200087}, \cite{MR3515306}, \cite{MR3661549} or  \cite{MR3960264}, \cite{MR3837172} for relations to spectra of Neumann-Poincar\'e operators.
\end{rem}


\section{ \bf Abstract results for polynomial pencils}
\label{section:6}

Before proceeding with the analysis of the spectral pollution for the domain truncation method applied to $\cL$ we need some abstract results providing an enclosure for the set of spectral pollution of sequences of polynomial pencils. 

Let $H_0$ be a Hilbert space, $H, H_n \subset H_0$ be closed subspaces. Let $P: H_0 \to H$, $P_n: H_0 \to H_n$ be the corresponding orthogonal projections and assume that $P_n \to P$ strongly in $H_0$, which we write as $P_n\s P$. For fixed $M \in \N$, let $A_j$, $j=0, \dots, M$, be densely defined operators in $H$ and, for $n \in \N$, let $A_{j,n}$, $j=0, \dots, M$,  be densely defined operators in $H_n$.
We assume that $A_j$, $j\ne 0$, are bounded and $A_{j,n}$, $j\ne 0$, are uniformly bounded in $n\in\N$; in particular, only $A_0$ and $A_{0,n}$ may be unbounded.

In addition, we assume that there exists a ray $ \e^{\I \gamma} (-\infty,c) \subset\bigcap_{n\in\N}\rho(A_{0,n})\cap\rho(A_0)$
with $c\in\R$, $\gamma \in (-\pi,\pi]$ such \vspace{-1mm} that
\begin{equation}
\label{eq:decay}
  \lim_{t\in\e^{\I \gamma}\R, t\to -\infty}\|(A_0-t)^{-1}\|\to 0,
  \quad
	\lim_{t\in\e^{\I \gamma}\R, t\to -\infty}\,\sup_{n\in\N}\|(A_{0,n}-t)^{-1}\|\to 0.
\vspace{-1mm}	
\end{equation}
This assumption is satisfied e.g.\ if $A_0$ and $A_{0,n}$, $n\in\N$, are $m$-accretive (then with $\gamma=0$)
or self-adjoint (then with $\gamma=\frac \pi 2$ or $-\frac \pi 2$). In the sequel we assume, without loss of generality, that $\gamma=0$.

Consider the \vspace{-1.2mm}  pencils of operators acting in $H$ and $H_n$, respectively, given by
\begin{alignat*}{2}
T(\lambda) &:= \sum^M_{j=0} \lambda^j A_j,  & \quad \dom T(\lambda)&:=\dom(A_0) \subset H, \\[-1mm]
T_n(\lambda) &:= \sum^M_{j=0} \lambda^j A_{j,n}, \quad & \dom T_n (\lambda) &:=\dom(A_{0,n}) \subset H_n, \quad n\in \N.
\\[-7mm]
\end{alignat*}
The boundedness of all higher order coefficient operators implies  that all derivatives $T^{(k)}(\la)$,  $T_n^{(k)}(\la)$, $n\in\N$, $k=1,2,\dots,M$, are bounded operators
\vspace{-1.2mm}
and that 
\begin{alignat*}{3}
T^*(\la)&:=\ T(\overline{\la})^* \!&&=\! \sum^M_{j=0} \lambda^j A_j^*, \quad  & \dom T^*(\la) &=\! \dom A_0^*, \\[-1mm]
T_n^*(\la)&:=T_n(\overline{\la})^*\!&&=\! \sum^M_{j=0} \lambda^j A_{j,n}^*, \quad &  \dom T_n^*(\la) &=\! \dom A_{0,n}^*, \ n\!\in\!\N.
\\[-8mm]
\end{alignat*}

We define the \emph{region of boundedness} of the sequence $(T_n)_{n\in\N}$ by
\[
  \Delta_b((T_n)_{n\in \N}) \!:=\!
  \big\{ \la \!\in\! \C :\,\exists \,n_0 \!\in\! \N \mbox{ with } \la\!\in\! \rho(T_n), n \!\geq\! n_0, \,\sup_{n \geq n_0} \norma{T_n(\la)^{-1}} \!<\! \infty\big\};
\]
note that, for the case of monic linear operator pencils $T_n(\la)\!:=\!\lambda - A_{0,n}$, $\lambda\in\C$, with unbounded $A_{0,n}$, this notion coincides with the region of boundedness of the operator sequence $(A_{0,n})_{n\in\N}$, see \cite[Def.~2.1~(iii)]{MR3694623}.

\begin{lemma}\label{lemma.tildeDelta}
{\rm i)}
Let $\la\in  \Delta_b\big((T_n)_{n\in\N}\big)$ with $\la\in\rho(T_n)$ for $n\geq n_\la$.
Then there exist $r_{\lambda},m_{\lambda}>0$ such that  $B_{r_{\lambda}}(\la)\subset \Delta_b\big((T_n)_{n\in\N}\big)$ with
$$
\forall\,\mu\in B_{r_{\lambda}}(\la):\quad  \mu\in\rho(T_n), \quad \|T_n(\mu)^{-1}\|\leq m_{\la}, \quad n\geq n_\la.
$$
{\rm ii)}
Let $K\!\subset\!\Delta_b\big((T_n)_{n\in\N}\big)$ be a compact subset. Then there exist $n_K\!\in\!\N$, $m_K\!>\!0$~with
$$
\forall\,\mu\in K:\quad \mu\in\rho(T_n), \quad \|T_n(\mu)^{-1}\|\leq m_K, \quad n\geq n_K.
$$
\end{lemma}

\begin{proof}
i) Let $\la$ satisfy the assumptions and let $n\geq n_\la$.
By a Neumann series argument, the \vspace{-1.2mm} operator
$$
T_n(\mu)=\left(I+\sum_{k=1}^{M}\frac{(\mu-\la)^k}{k!}T_n^{(k)}(\la)T_n(\la)^{-1}\right)T_n(\la), \quad n\geq n_0,
$$
is boundedly invertible if $\mu\in B_{r_\la}(\la)$ and $r_\la>0$ is so small \vspace{-1.2mm} that
$$
c_{\la}:=\sum_{k=1}^{M}\frac{r_\lambda^k}{k!}\sup_{n\geq n_\la}\|T_n^{(k)}(\la)T_n(\la)^{-1}\|<1.
$$
Note that, for every $k\!=\!1,\dots,M$, the operators $T_n^{(k)}(\la)\!=\!\sum_{j=1}^M \frac{j!}{(j-k)!}\la^{j-k}A_{j,n}$,  are bounded uniformly in $n\in\N$.  We obtain that $B_{r_\la}(\la)\!\subset\!\rho(T_n)$ for every $n\!\geq\! n_\la$,~with
$$
\|T_n(\mu)^{-1}\|\leq \frac{\sup_{n\geq n_\la}\|T_n(\la)^{-1}\|}{1-c_\la}, \quad \mu\in B_{r_\la}(\la).
$$

ii)
By i), the compact set $K$ can be covered by open disks (around each $\la\!\in\!K$) on which $\mu\mapsto \sup_{n\geq n_\la}\|T_n(\mu)^{-1}\|$ is uniformly bounded. Since $K$ is compact, there exists a finite covering of such disks. Now the claim is easy to see.
\qedhere
\end{proof}

\begin{prop}\label{prop.pollution}
No spectral pollution occurs in $ \Delta_b\big((T_n)_{n\in\N}\big)$.
\end{prop}

\begin{proof}
Let $\la\!\in\! \Delta_b\big((T_n)_{n\in\N}\big)$.Lemma~\ref{lemma.tildeDelta}~i) implies that $B_{r_{\la}}(\la)\!\subset\!\rho(T_n)$ for $n\!\geq\! n_\la$, and so, in the limit $n\!\to\!\infty$, points in $\sigma(T_n)$ cannot accumulate at~$\la$.
\end{proof}

\begin{lemma}\label{lemmagsrpencil}
Assume that there exists $\la_0\in\bigcap_{n\in\N}\rho(T_n)\cap\rho(T)$ with
\begin{equation}\label{eq.convatlm01}
T_n(\la_0)^{-1}P_n\s T(\la_0)^{-1}P, \quad n\to\infty,
\end{equation}
and that $A_{j,n}P_n\s A_jP$ for $j=1\dots,M$.
Then for every $\la\in \Delta_b\big((T_n)_{n\in\N}\big)\cap\rho(T)$,
$$
T_n(\la)^{-1}P_n\s T(\la)^{-1}P, \quad n\to\infty.
$$
\end{lemma}

\begin{proof}
Let $\la\in \Delta_b\big((T_n)_{n\in\N}\big)\cap\rho(T)$. Define the bounded \vspace{-1mm} operators
\begin{align*}
S(\la)&:=T(\la)-T(\la_0)=\sum_{j=1}^M (\la^j-\la_0^j)A_{j},
\\[-1mm]
S_n(\la)&:=T_n(\la)-T_n(\la_0)=\sum_{j=1}^M (\la^j-\la_0^j)A_{j,n}, \quad n\in\N.
\\[-7mm]
\end{align*}
Assumption \eqref{eq:decay} together with the boundedness of the operators $A_{j,n}$, $n\!\in\!\N$, $j\!=\!1,\dots,M$, imply that, by a Neumann series argument, there exists $t_0\in\R$ such that $(-\infty,t_0)$ is contained in the (operator) region of boundedness $\Delta_b\left((T_n(\la_0)\right)_{n\in\N})$, see \cite[Def.~2.1~(iii)]{MR3694623}, and in $\rho(T(\la_0))$, with
\begin{equation}
\label{eq:conv-res-seq}
\lim_{t\in\R,t\to-\infty}\|(T(\la_0)-t)^{-1}\|=0, \quad \lim_{t\in\R,t\to-\infty}\,\sup_{n\in\N}\|(T_n(\la_0)-t)^{-1}\|=0.
\end{equation}
Then \eqref{eq.convatlm01} and \cite[Prop.\ 2.16 i)]{MR3694623} imply that, for $t\in (-\infty,t_0)$,
$$(T_n(\la_0)-t)^{-1}P_n\s (T(\la_0)-t)^{-1}P, \quad n\to\infty.$$
By the assumptions, $S_n(\la)P_n\!\s\! S(\la)P$ as $n\!\to\!\infty$.
This and \eqref{eq:conv-res-seq} show that~the perturbation result \cite[Cor.\ 3.5]{MR3694623}, applies to $T(\la)\!=\!T(\la_0)+S(\la)$, $T_n(\la)\!=\!T_n(\la_0)+S_n(\la)$,  $n\in\N$, and yields that, for all sufficiently negative $t\in (-\infty,t_0)$,
$$(T_n(\la)-t)^{-1}P_n\s (T(\la)-t)^{-1}P, \quad n\to\infty.$$
By the choice of $\la$ we have $0\!\in\!\Delta_b\left((T_n(\la)\right)_{n\in\N})\cap\rho(T(\la))$, and hence another appli\-cation of \cite[Prop.\ 2.16 i)]{MR3694623}
implies the claim.
\end{proof}

\begin{prop}\label{prop.inclusion}
Suppose that the assumptions of Lemma~{\rm \ref{lemmagsrpencil}} are satisfied.
Then, for each $\la\in\sigma_p(T)$ such that for some $\eps>0$ we have
\beq
  B_{\eps}(\la)\backslash\{\la\}\subset   \Delta_b\big((T_n)_{n\in\N}\big)\cap \rho(T),
\label{eqneighbourhoodinregionofbdd}
\eeq
there exists a sequence of elements $\la_n\in\sigma(T_n)$, $n\in\N,$ with $\la_n\to\la$, $n\to\infty$.
\end{prop}

\begin{proof}
Let $\la\!\in\!\sigma_p(T)$ and $\eps\!>\!0$ satisfy~\eqref{eqneighbourhoodinregionofbdd}.
Assume the claim does not hold. Then there exists a
$\delta\!\in\!(0,\eps)$ and an infinite subset $I\!\subset\!\N$ with ${\rm dist}(\la,\sigma(T_n))\!\geq\! 2\delta$, $n\!\in\! I$.
Define bounded operators $Q$ and $Q_n$, $n\in\N$, by the contour \vspace{-1mm} integrals
\begin{align*}
  Q&:=\frac{1}{2\pi\I}\int_{|z|=\delta} T(\la+z)^{-1}\sum_{k=0}^{M-1} \frac{z^k}{(k+1)!}T^{(k+1)}(\la)\,\rd z, \\
  Q_n&:=\frac{1}{2\pi\I}\int_{|z|=\delta} \!T_n(\la+z)^{-1}\sum_{k=0}^{M-1} \frac{z^k}{(k+1)!}T_n^{(k+1)}(\la)\,\rd z, \quad n\in I;
\end{align*}
recall that the sums on the right-hand side are bounded operators since all higher order coefficients of $T$ were assumed to be bounded.
Since $z\!\mapsto\! T_n(\la+z)^{-1}$ is holomorphic in $B_{2\delta}(0)$, we have $Q_n\!=\!0$, $n\!\in\! I$.
Since $\la\!\in\!\sigma_p(T)$, there exists $x\in\dom(T)$ with $\|x\|=1$ and $T(\la)x=0$.
Using this in the Taylor expansion of $T$ in $\lambda$, we conclude \vspace{-2mm} that
\begin{align*}
   T(\la+z)x=\sum_{k=1}^M \frac{z^k}{k!}T^{(k)}(\la) x, & \quad z\in B_{2\delta}(0),\\[-3mm]
\intertext{and \vspace{-1mm} hence}
 \frac 1 z x=T(\la+z)^{-1}\sum_{k=0}^{M-1} \frac{z^k}{(k+1)!} T^{(k+1)}(\la)x, &\quad z\in B_{2\delta}(0)\setminus \{0\}.
\end{align*}
Now Cauchy's integral formula implies \vspace{-1.5mm}that
\begin{equation}
\label{eq:Qx=x}
  Qx=\left(\frac{1}{2\pi\I}\int_{|z|=\delta} \frac 1 z \,\rd z\right)x=x\ne 0.
\end{equation}

For every $n\in I$, define the function $f_n:\{z\in\C:\,|z|=\delta\}\to [0,\infty)$ by
$$
   f_n(z)\!:=
	 \!\Big\|T(\la\!+\!z)^{-1}\!\sum_{k=0}^{M-1} \!\frac{z^k}{(k\!+\!1)!}T^{(k+1)}(\la)Px
	 \!-\!T_n(\la\!+\!z)^{-1}\!\sum_{k=0}^{M-1} \!\frac{z^k(\la)}{(k\!+\!1)!}T_n^{(k+1)}P_nx\Big\|.
$$
Then $$\|QPx-Q_nP_nx\|\leq \frac{1}{2\pi}\int_{|z|=\delta} f_n (z )\,\rd |z |, \quad n\in I.$$
The assumptions together with Lemma~\ref{lemmagsrpencil} imply that $f_n(z )\to 0$, $n\to\infty$, for every $z\!\in\!\C$ with $|z|\!=\!\delta$.
Note that $f_n$, $n\!\in\!\N,$ are uniformly bounded by the compactness of the circle $\{z\!\in\!\C:|z|\!=\!\delta\}$ and by Lemma~\ref{lemma.tildeDelta}~ii).
Lebesgue's dominated convergence theorem implies $\|QPx-Q_nP_nx\|\to 0$ as $n\in I$, $n\to\infty$.
Since $Q_n=0$, $n\in I$, it follows that
$QPx=0$. {However $Px\!=\!x$ since $x\!\in\! \dom(T) \!\subset\! H$ and $P$ is a projection onto $H$. Thus} $Qx\!=\!0$, a contradiction to $Qx\!=\!x\!\neq\! 0$, see~\eqref{eq:Qx=x}.
\end{proof}

Next we define the \emph{limiting approximate point spectrum} by
\begin{alignat*}{2}
\sigma_{\rm app}\big((T_n)_{n\in\N}\big) \!&:=\!
\big\{\la\!\in\!\C:\,\exists\, I\!\subset\!\N,\,I\ \mbox{infinite},\,& \exists\,&x_n\!\in\!\dom(T_n), \, \|x_n\|\!=\!1,\,n\!\in\! I, \\
&&& \hspace{1.3cm}\text{with } \|T_n(\la)x_n\|\to 0\big\},
\\[-3mm]
\intertext{the \emph{limiting essential spectrum} by}
\sigma_{e}\big((T_n)_{n\in\N}\big)\!&:=\!
\big\{\la\!\in\!\C:\,\exists\, I\!\subset\!\N,\,I\ \mbox{infinite},\,& \exists\,&x_n\!\in\!\dom(T_n), \, \|x_n\|\!=\!1,\,n\!\in\! I, \\
 &&& \text{with }x_n \rightharpoonup 0,\,\|T_n(\la)x_n\|\to 0\big\},
\end{alignat*}
It is easy to see that, as in the operator case, see \cite[Lemma~2.14 ii)]{MR3831156},
\begin{equation}\label{eq:sigmaapp}
   \C\backslash \Delta_b\big((T_n)_{n\in\N}\big)=\sigma_{\rm app}\big((T_n)_{n\in\N}\big)\cup \sigma_{\rm app}\big((T_n^*)_{n\in\N}\big)^*.
\end{equation}

\begin{prop}\label{prop.sigmae}
Suppose that the assumptions of Lemma~{\rm\ref{lemmagsrpencil}} are satisfied.
Then
$$
  \sigma_{\rm app}\big((T_n^*)_{n\in\N}\big)^*\subset \sigma_e\big((T_n^*)_{n\in\N}\big)^*\cup\sigma_p(T^*)^*.
$$
\end{prop}

\begin{proof}
Let $\la\in \sigma_{\rm app}\big((T_n^*)_{n\in\N}\big)^*$.
 By definition, there exist  an infinite subset $I\!\subset\!\N$ and $x_n\!\in\!\dom(T_n^*)$, $n\!\in\! I$, with
 $\|x_n\|\!=\!1$ and $\|T_n(\la)^*x_n\|\!\to\! 0$ as $n\!\to\!\infty$. The sequence $(x_n)_{n\in\N}\subset H_0$ is bounded and thus has
 a weakly convergent subsequence $(x_n)_{n\in I_2}$ with infinite $I_2\subset I$; denote its weak limit by $x\in H_0$.
If $x\!=\!0$, then $\la\!\in\! \sigma_e\big((T_n^*)_{n\in\N}\big)^*$.

Now assume that $x\neq 0$. Define $y_n:=T_n(\la)^*x_n$, $n\in I_2$. Then $y_n\to 0$ as $n\!\to\!\infty$.
Note that, if $z\in\Delta_b((T_n)_{n\in\N})\cap\rho(T)$,  then Lemma~{\rm\ref{lemmagsrpencil}} implies $T_n(z)^{-1}P_n\s T(z)^{-1}P$,
\vspace{-2mm} and
$$
  T_n(z)^*x_n=\sum_{k=1}^M\frac{(\overline{z}-\overline{\la})^k}{k!}T_n^{(k)}(\la)^*x_n+y_n, \quad n\in I_2.
$$
\vspace{-1mm}Thus
\begin{equation}\label{eq.xn}
x_n=T_n(z)^{-*}\sum_{k=1}^M\frac{(\overline{z}-\overline{\la})^k}{k!}T_n^{(k)}(\la)^*x_n+T_n(z)^{-*}y_n, \quad n\in I_2.
\end{equation}
Let $w\in H_0$ be arbitrary.
The convergence assumptions, $y_n\to 0$ as $n\!\to\!\infty$ and $T^{(k)}(\la)=P T^{(k)}(\la)$ imply that
\begin{align*}
\langle x_n,w\rangle&=\sum_{k=1}^M\frac{(\overline{z}-\overline{\la})^k}{k!}\left\langle x_n,T_n^{(k)}(\la)T_n(z)^{-1}P_nw\right\rangle
 +\left\langle y_n,T_n(z)^{-1}P_nw\right\rangle\\
&\to \sum_{k=1}^M\frac{(\overline{z}-\overline{\la})^k}{k!}\left\langle x,P T^{(k)}(\la)T(z)^{-1}Pw\right\rangle \\ &
= \Big\langle T(z)^{-*}\sum_{k=1}^M\frac{(\overline{z}-\overline{\la})^k}{k!}T^{(k)}(\la)^*Px,w\Big\rangle
\end{align*}
as $n\!\to\!\infty$. By the uniqueness of the weak limit, we obtain \vspace{-1mm}that
$$
  x=T(z)^{-*}\sum_{k=1}^M\frac{(\overline{z}-\overline{\la})^k}{k!}T^{(k)}(\la)^*Px\in\dom(T(z)^*) \subset H
$$
hence $Px=x$ \vspace{-1mm} and
$$
   T(z)^*x=\sum_{k=1}^M\frac{(\overline{z}-\overline{\la})^k}{k!}T^{(k)}(\la)^*x.
$$
The uniqueness of the Taylor expansion of $T(\cdot)^*$ in $\la$ implies that $0=T(\la)^*x$ $=T^*(\overline{\la})x$.
Since $x\neq 0$, we conclude that $\la\in \sigma_p(T^*)^*$.
\end{proof}

Now we prove the main result of this section.

\begin{theorem}\label{thm: bad set pencils}
Assume that there exists $\la_0\in\bigcap_{n\in\N}\rho(T_n)\cap\rho(T)$ with
$$T_n(\la_0)^{-1}P_n\s T(\la_0)^{-1}P, \quad T_n(\la_0)^{-*}P_n\s T(\la_0)^{-*}P.$$
If also $A_{j,n}P_n\s A_jP$ and $A_{j,n}^*P_n\s A_j^*P$ for every $j=1,\dots,M$, then spectral pollution is contained \vspace{-2mm} in
\be
\sigma_e\big((T_n)_{n\in\N}\big)\cup \sigma_e\big((T_n^*)_{n\in\N}\big)^*,\label{eq.bad1}
\ee
and for every isolated $\la\!\in\!\sigma_p(T)$ not belonging to
$\sigma_e\big((T_n)_{n\in\N}\big)\cup \sigma_e\big((T_n^*)_{n\in\N}\big)^*$
there exist $\la_n\in\sigma(T_n)$, $n\in\N$, with $\la_n\to\la$.
\end{theorem}

\begin{proof}
First note that $\Delta_b\big((T_n)_{n\in\N}\big)=\Delta_b\big((T_n^*)_{n\in\N}\big)^*$, see \eqref{eq.bad1}. The latter
and Proposition~\ref{prop.sigmae} imply that
\begin{align*}
\big(\C\backslash\Delta_b\big((T_n)_{n\in\N}\big)\big)\cap\rho(T)
& = \big(\sigma_{\rm app}\big((T_n)_{n\in\N}\big)\cup \sigma_{\rm app}\big((T_n^*)_{n\in\N}\big)^*\big)\cap\rho(T)\\
&\subset \sigma_e\big((T_n)_{n\in\N}\big)\cup \sigma_e\big((T_n^*)_{n\in\N}\big)^*;
\end{align*}
note that $\la \!\in\! \sigma_{p}(T^*)^*$ implies that $\{0\} \!\ne\! \ker T(\la)^* \!=\! \ran T(\la)^\perp$ and hence $\la \!\notin\! \rho(T)$.
Now the claims follow from Propositions~\ref{prop.pollution} and \ref{prop.inclusion}.
\end{proof}


\section{{\bf Limiting essential spectrum}}
\label{sec: limiting sigmae}

In this section, along with
the linear Maxwell pencil $V(\cdot)$ in $L^2(\Omega)^3\oplus L^2(\Omega)^3$, see \eqref{eq:Vdef0},
the associated operator matrix $\cA$ in  $L^2(\Omega)^3\oplus L^2(\Omega)^3$, see \eqref{def:calA}, and the quadratic operator pencil $\cL(\cdot)$ in $ L^2(\Omega)^3$, see \eqref{eq:defcL},
we now consider their analogues $V_n(\cdot)$ and $\cA_n$ in $L^2(\Omega_n)^3\oplus L^2(\Omega_n)^3$ and $\cL_n(\cdot)$ in $L^2(\Omega_n)^3$, respectively.

Note that all our results in Sections \ref{sec:proof1} on spectral enclosures and resolvent estimates for $V(\cdot)$ and $\cA$ as well as in Section \ref{sec:V-L} on the relations between the spectral properties of $V(\cdot)$ and $\cL(\cdot)$ hold for both bounded and unbounded domains, and thus cover, when applied on the domains $\Omega_n$, $n\in \N$,  equally $V_n(\cdot)$, $\cA_n$ and~$\cL_n(\cdot)$.

For convenience, we briefly recall that,
in line with \eqref{Vfact1}, \eqref{def:calA} \vspace{-1mm} and~\eqref{eq:defcL},
\begin{equation}
\begin{aligned}
\label{Vnfact1}
  &V_n(\omega) =
  \begin{pmatrix} \eps^{1/2} & 0 \\ 0 & \mu^{1/2}\end{pmatrix}(\mathcal{A}_n-\omega I) \begin{pmatrix} \eps^{1/2} & 0 \\ 0 & \mu^{1/2}\end{pmatrix},\\
  & \dom V_n(\omega)=H_0(\curl,\Omega_n)\oplus H(\curl,\Omega_n),
\end{aligned}
\end{equation}
in which
\begin{equation}
\label{deficalAn}
\begin{aligned}
   &\cA_n  :=
  \begin{pmatrix}
	-\I \epsilon^{-\frac 12} \sigma \epsilon^{-\frac 12}  & - \I \epsilon^{-\frac 12} \curl \mu^{-1/2} \\
	\I \mu^{-1/2}{ \curl_0} \epsilon^{-\frac 12} &  0
	\end{pmatrix}, \\
  &\dom \cA_n := \eps^{1/2} H_0(\curl, \Omega_n) \oplus \mu^{1/2}H(\curl, \Omega_n),
\end{aligned}	
\end{equation}
and
\begin{equation}
\label{eq:defcLn}
\begin{aligned}
 &\cL_n(\omega):=\curl\mu^{-1}{\curl_0}-\omega(\omega\eps+\I\sigma), \\
 &\dom(\cL_n(\omega)):=\{E\in H_0(\curl,\Omega_n):\,\mu^{-1}\curl E\in H(\curl,\Omega_n)\}.
\end{aligned}
\end{equation}

In the sequel, we define the orthogonal projection $P_n:L^2(\Omega)^3\to L^2(\Omega_n)^3$ by $P_nu=\chi_{\Omega_n}u$ for $u \in L^2(\Omega)^3$. Note that $L^2(\Omega_n)$ is understood as a subspace of $L^2(\Omega)$ by extending each function by zero.

\begin{prop}
\label{prop:gsrL}
Let $\omega\!=\!\I t$ with $t\!\geq\! \epsilon_{\min}^{-1/2}$.  Then $\mathcal L_n(\omega)^{-1}P_n\!\s\! \mathcal L(\omega)^{-1}$ as $n\!\to\!\infty$.
\end{prop}

\begin{proof}
In the sequel we use Lemma \ref{lem:squareroot} applied to both $\cL(\cdot)$ and to its truncated analogues $\cL_n(\cdot)$;
the truncated analogues of $T_0\!=\!\mu^{-1/} \curl_0$, $\dom T_0\!=\! H_0(\curl, \Omega)$, and of
$\cW(\omega)\!=\!-\omega(\omega\epsilon+\I\sigma)$, $\omega\!\in\!\C$,  in $ L^2(\Omega)^3$, are operators in $L^2(\Omega_n)^3$
which we denote by $T_{0,n}\!=\!\mu^{-1/2} \curl_0$, $\dom T_{0,n}\!=\! H_0(\curl, \Omega_n)$, and 
$\cW_n(\omega)\!=\!-\omega(\omega\epsilon+\I\sigma)$. 


Because $\Omega\!=\!\bigcup_{n\in\N} \Omega_n$ and $C_c^{\infty}(\Omega)^3$ is a core of $(T_0^*T_0\!+\!I)^{1/2}\!\!$, see Lemma \ref{lem:squareroot}, it follows that for every $u$ there exists $N_u\in\N$ such that ${\rm supp}\,u\subset\Omega_n$ for all $n\!\geq\! N_u$. Then $(T_0^*T_0\!+\!I)^{1/2}u\!=\!(T_{0,n}^*T_{0,n}\!+\!I)^{1/2}P_nu$ for $n\!\geq\! N_u$. \vspace{-0.5mm} By~Lemma \ref{lem:squareroot}
 $\sup_{n\in\N} \|(T_{0,n}^*T_{0,n}\!+\!I)^{-1/2}\|\!\leq\! 1\!<\!\infty$ and hence \cite[Thm.\ 3.1]{MR3694623} yields that
$$
  (T_{0,n}^*T_{0,n}+I)^{-1/2}P_n\s (T_0^*T_0+I)^{-1/2}, \quad n\to\infty.
$$

It is easy to see that $(\cW_n(\omega)\!-\!I)P_n\s \cW(\omega)\!-\!I$ as $n\to\infty$ for all $\omega\!\in\!\C$. Since the product and sum of strongly convergent operators are strongly convergent, we obtain \vspace{-1mm} that
\begin{equation}
\label{eq:for-isc}
\begin{aligned}
&\left(I+(T_{0,n}^*T_{0,n}+I)^{-1/2}(\cW_n(\omega)-I)(T_{0,n}^*T_{0,n}+I)^{-1/2}\right)P_n \\
&\s I+(T_0^*T_0+I)^{-1/2}(\cW(\omega)-I)(T_0^*T_0+I)^{-1/2}, \quad n\to\infty.
\end{aligned}
\end{equation}
Now let $\omega= \I t$ with $t\geq \epsilon_{\min}^{-1/2}$.
Then Lemma \ref{lem:squareroot} implies that
\begin{equation}\label{eq:resest2}
\begin{aligned}
&\left\|\left(I+(T_0^*T_0+I)^{-1/2}(\cW(\omega)-I)(T_0^*T_0+I)^{-1/2}\right)^{\!-1}\right\|\leq 1, \\[-1mm]
&\sup_{n\in\N}\,\left\|\left(I\!+\!(T_{0,n}^*T_{0,n}\!+\!I)^{-1/2}(\cW_n(\omega)\!-\!I)(T_{0,n}^*T_{0,n}\!+\!I)^{-1/2}\right)^{\!-1}\right\|\leq 1<\infty.
\end{aligned}
\end{equation}
Hence, by \cite[Lemma 3.2]{MR3694623}, the inverses in \eqref{eq:for-isc}, converge strongly as well,
\begin{align*}
&\left(I+(T_{0,n}^*T_{0,n}+I)^{-1/2}(\cW_n(\omega)-I)(T_{0,n}^*T_{0,n}+I)^{-1/2}\right)^{\!-1}P_n \\[-1mm]
&\s \left(I+(T_0^*T_0+I)^{-1/2}(\cW(\omega)-I)(T_0^*T_0+I)^{-1/2}\right)^{\!-1}, \quad n\to\infty.
\end{align*}
Altogether, we arrive at
\begin{alignat*}{2}
\!\mathcal L_n(\omega)^{-1}\!P_n\!
&=\!(T_{0,n}^*T_{0,n}\!+\!I)^{-\frac{1}{2}}\!\!\left(I\!+\!(T_{0,n}^*T_{0,n}\!+\!I)^{-\frac{1}{2}}\!(\cW_n(\omega)\!-\!I)(T_{0,n}^*T_{0,n}\!+\!I)^{-\frac{1}{2}}\right)^{\!-1}\!\!\!\cdot \\[-1mm]
&\hspace{7.6cm} \cdot (T_{0,n}^*T_{0,n}\!+\!I)^{-\frac{1}{2}}P_n\hspace{-5mm}\\
&\s \!(T_0^*T_0\!+\!I)^{-\frac{1}{2}}\!\!\left(\! I\!+\!(T_0^*T_0\!+\!I)^{-\frac{1}{2}}(\cW(\omega)\!-\!I)(T_0^*T_0\!+\!I)^{-\frac{1}{2}}\!\right)^{\!\!-1}\!\!\!\!\!
(T_0^*T_0\!+\!I)^{-\frac{1}{2}}\\
&=\mathcal L(\omega)^{-1}. \hspace{8.7cm}
\qedhere
\end{alignat*}
\end{proof}

\vspace{2mm}

Applying Theorem \ref{thm: bad set pencils}  to the quadratic pencils $\cL_n$
and using that $\cL_n$ is $J$-self-adjoint with respect to conjugation for all $n \in \N$ so that $\sigma_e((\cL^*_n)_{n\in\N})^* = \sigma_e((\cL_n)_{n\in\N})$,
we immediately obtain
\begin{equation} \label{eq: inclusion spectral poll}
\sigma_{\rm poll}((\cL_n)_{n\in\N}) \subset \sigma_e((\cL_n)_{n \in \N}) \cup \sigma_e((\cL^*_n)_{n\in\N})^*
= \sigma_e((\cL_n)_{n\in\N}).
\end{equation}

\begin{prop} \label{prop: equality lim ess spectra1}
Suppose that $\sigma$, $\eps$ and $\mu$ satisfy the limiting assumption \eqref{eq:coeffs-infty}.
Denote by ${\cT_n}$ the triangular operator matrices given by \eqref{eq: A^0} with $\Omega$ replaced by $\Omega_n$,
i.e.\ acting in $L^2(\Omega_n)^3\!=\! \nabla \dot{H}^1_0(\Omega_n) \oplus H(\Div0, \Omega_n)$.
Then the  limiting essential spectra of 
$(\cL_n)_{n\in\N}$ and $({\cT_n})_{n\in\N}$ are equal,
\[
   \sigma_e((\cL_n)_{n\in\N}) = \sigma_e((\cT_n)_{n\in\N}).
\]
\end{prop}

\begin{proof}
The proof is closely modelled on the proofs of  Theorem \ref{sigma-ess}  and Proposition~\ref{thm: compactness}. Let $\omega\!\in\!\C$ be fixed.
Let $M_n(\omega) \!:=\! (\omega(\omega \eps\! +\! \I \sigma)\!-\! \omega^2)P_{\ker(\Div, \Omega_n)}$ in $L^2(\Omega_n)^3$. Then $\cT_n(\omega)$ is the operator matrix representation of $\cL_n(\omega)\!+\!M_n(\omega)$ in $L^2(\Omega_n)^3\!=\! \nabla \dot{H}^1_0(\Omega_n) \oplus H(\Div0, \Omega_n) $.
First note that, for any $u_n\!\in\!\dom(\cL_n)$, $\|u_n\|\!=\!1$, the sequence $(\|\cL_n(\omega) u_n\|)_{n\in\N}$ is bounded if and only if the sequence $(\|(\cL_n(\omega)\!+\!M_n(\omega))u_n\|)_{n\in\N}$ is bounded.

Now we argue that it suffices to show the following claim: if any of the above two sequences is bounded, then for any infinite subset $I\subset\N$ the sequence $(M_n(\omega)u_n)_{n\in I}$ $\subset  L^2(\Omega)^3$ has a convergent subsequence.
To see that this claim proves the theorem, assume that $u_n \rightharpoonup 0$ and $\cL_n(\omega)u_n\!\to\! 0$ as $n\!\to\!\infty$, i.e.\ $\omega\!\in\!\sigma_e((\cL_n)_{n\in\N})$.
Then, by the claim together with $u_n\rightharpoonup 0$, and the uniqueness of the weak limit, we get $M_n(\omega) u_n\to 0$ as $n\!\to\!\infty$, whence $\omega\!\in\! \sigma_e((\cL_n\!+\!M_n)_{n\in\N})$. The proof is analogous if we start with $\omega\in \sigma_e((\cL_n\!+\!M_n)_{n\in\N})$.

To prove the claim, let $(\|\cL_n(\omega) u_n\|)_{n\in\N}$ be bounded. Then $(\|u_n\|_{H(\curl,\Omega_n)})_{n\in\N}$ is bounded as well,
and thus the property that, for any infinite subset $I\subset\N$, $(M_n(\omega)u_n)_{n\in I}\subset  L^2(\Omega)^3$ has a convergent subsequence
means that
$$
	M_n(\omega) P_{\ker\Div}
	\!:\!(H(\curl,\Omega_n),\|\cdot\|_{H(\curl,\Omega_n)})\!\to\! (L^2(\Omega_n)^3,\|\cdot\|_{L^2(\Omega_n)^3}),
  \ n\!\in\!\N,	
$$
form a \emph{discretely compact} sequence, see \cite[Def.\ 3.1.(k)]{MR291870} or \cite[Def.\ 2.5]{MR3694623}.
 As in the proof of Proposition~\ref{thm: compactness},
for any $\delta>0$ we can write $M_n(\omega) = M_{c,n}(\omega) + M_{\delta,n}(\omega)$ where
$M_{\delta,n}(\omega)$ is a bounded multiplication operator with $\|M_{\delta,n}(\omega)\|<\delta$
vanishing uniformly in~$n$ as $\delta\to 0$ and
$M_{c,n}(\omega)$ has compact support in some domain $\Omega_{R,n} := \Omega_n \cap B(0,R)\subset \Omega\cap B(0,R)=\Omega_R$ for sufficiently large $R>0$.
Since the uniform limit of a discretely compact sequence is discretely compact, see \cite[Prop.\ 2.9]{MR3694623}, the sequence $(M_n(\omega))_{n\in\N}$ is discretely compact if each sequence $(M_{c,n}(\omega))_{n\in\N}$, $\delta>0$, is discretely compact.
To show the latter, let $I\subset\N$ be an infinite subset. Let $\chi_R$ be the same cut-off function as in the proof of Proposition~\ref{thm: compactness} and let $\iota$ be the compact embedding of $H(\curl, \Omega_{R}) \cap H(\Div, \Omega_{R})$
 in $L^2(\Omega_{R})^3$, see~\cite{MR561375}.
Then, for all sufficiently large $n\in I$, $\supp M_{c,n}(\omega) \subset \Omega_{R,n} \subset \Omega_n$  and
$$
  M_{c,n}(\omega)  P_{\ker\Div}u_n=M_{c,n}(\omega) \iota(\chi_R P_{\ker(\Div, \Omega_n)}u_n)|_{\Omega_{R,n}}.
$$
As in the proof of Proposition~\ref{thm: compactness}, we now deduce that $(M_{c,n}(\omega)u_n)_{n\in I}\subset  L^2(\Omega)^3$ has a convergent subsequence.
\end{proof}

\begin{prop}
\label{prop: sigma_e cL}
Suppose that $\sigma$, $\eps$ and $\mu$ satisfy the limiting assumption \eqref{eq:coeffs-infty}.
Let $L_{\mu,n}$ and $L_{\infty,n}$ be defined in the same way as $L_\mu$ and $L_\infty$, see Definition {\rm \ref{Lmidef}} with $\Omega$ replaced by $\Omega_n$. Then
\[
  \sigma_e((L_{\mu,n})_{n\in\N}) = \sigma_e((L_{\infty,n})_{n\in\N}).
\]
\end{prop}

\begin{proof} Recall that $L_{\mu,n}(\omega)=\! C_{\mu,n} \!-\! \eps_\infty \omega^2 \id$,
$L_{\infty,n}(\omega)\!=\!C_{\infty,n} \!-\! \eps_\infty \omega^2 \id$, $n \in \N$, $\omega \in \C$, are closed operators acting in the Hilbert space $H(\Div 0, \Omega_n) \subsetneq L^2(\Omega_n)^3$, endowed with the $L^2(\Omega_n)^3$-norm,
and $C_{\mu,n} = \curl\mu^{-1}\curl_0$, $C_{\infty,n} = \curl\mu_\infty^{-1}\curl_0$ are self-adjoint therein.

The proof is modelled on that of Proposition~\ref{thm: difference res}. Here 
it suffices to prove
$\sigma_e((L_{\mu,n}(\omega))_{n\in\N})=\sigma_e((L_{\infty,n}(\omega))_{n\in\N}$ for only one $\omega\in\C$, which we choose
as  $\omega=\I t$ with $t\geq \eps_{\min}^{-1/2}$, or equivalently $\sigma_e((C_{\mu,n})_{n\in\N})=\sigma_e((C_{\infty,n})_{n\in\N})$.
By \cite[Thm.\ 2.5]{MR3831156}  the limiting essential spectrum has the spectral mapping property for the resolvent. Due to \cite[Thm.\ 2.12\,(ii)]{MR3831156} it is then enough to show that, for $z=\eps_\infty \omega^2 \le -1$,
\[
  K_n(z) := (L_{\mu,n}(\omega))^{-1} - (L_{\infty,n}(\omega))^{-1}
	= (C_{\mu,n}-z)^{-1} - (C_{\infty,n}-z)^{-1}
\]
is such that $(K_n(z))_{n\in\N}$ is discretely compact and $(K_n(z)^*P_n)_{n\in\N}$ is strongly convergent.
The strong convergence follows from Proposition \ref{prop:gsrL} which yields that
$$
   K_n(z)^*P_n
	 =   (\cL_{\mu,n}(\overline{\omega}))^{-1}P_n - (\cL_{\infty,n}(\overline{\omega}))^{-1}P_n
   \s  (\cL_{\mu}(\overline{\omega}))^{-1} - (\cL_{\infty}(\overline{\omega}))^{-1}.
$$

Applying \eqref{eq:form-2nd-res.-id} in the proof of Proposition~\ref{thm: difference res} on $\Omega_n$, we deduce that
\begin{equation}
\label{eq:Kn}
  K_n(z) =
	 \big( \curl_0 (C_{\mu,n}\!-\!\overline{z})^{-1} \big)^{\!*} (\mu_\infty^{-1} \!\!-\mu^{-1}) \curl_0 (C_{\infty,n}\!-\!z)^{-1}.
\end{equation}

By Lemma \ref{lemma: boundedness2} (ii) on $\Omega_n$,
the operators $\curl_0 (C_{\infty,n}\!-\!z)^{-1}\!=\!\curl_{0} (\cL_{\infty,n}(\omega))^{-1}$~are bounded from {$H(\Div 0,\Omega_n)$} to $H(\curl,\Omega_n)$ with uniformly bounded operator~\vspace{-1.5mm}norms,
\[
  \sup_{n\in\N} \|\curl (\cL_{\infty,n}(\omega))^{-1}\|_{\cB( H(\Div 0,\Omega_n),H(\curl,\Omega_n))}
  \leq 
\left(\frac{\mu_\infty}{\eps_\infty |\omega|^2}+\mu_\infty^2\right)^{1/2}
  <\infty.
\]
By Lemma \ref{lemma: boundedness2} (i) on $\Omega_n$, the  operators
$(\curl_0 (C_{\mu,n}\!-\!\overline{z})^{-1} \big)^{\!*}\!\!=
\overline{(\cL_{\mu,n}(\omega))^{-1}\curl} $ are bounded from $L^2(\Omega_n)^3$ to $H(\Div 0, \Omega_n)$. Moreover, they are strongly convergent, $\overline{(\cL_{\mu,n}(\omega))^{-1}\curl }P_n \!\s\! \overline{(\cL_{\mu}(\omega))^{-1}\curl }$ as $n\!\to\!\infty$, since for every $u\!\in\! H(\curl,\Omega)$ we have $P_n u\!\in\! H(\curl,\Omega_n)$ with $\curl P_n u = P_n \curl u$ as $\curl$ is a local operator,
and since $(\cL_{\mu,n}(\omega))^{-1}P_n \!\s\! (\cL_{\mu}(\omega))^{-1}$  as $n\!\to\!\infty$, which follows by analogy with the proof of Proposition \ref{prop:gsrL}.
Analogously to the proof of Proposition \ref{prop: equality lim ess spectra1} for $M_n(\omega)$, one can show \vspace{-0.5mm} that
$$
   (\mu^{-1} \!-\! \id) P_{\ker(\Div,\Omega_n)}:
	 (H(\curl, \Omega_n), \norma{\cdot}_{H(\curl, \Omega_n)}) \!\to\! (L^2(\Omega_n)^3, \norma{\cdot}_{L^2(\Omega_n)^3}),
	 \quad n\!\in\!\N,
$$
form a discretely compact sequence of operators. Now  \eqref{eq:Kn} and \cite[Lemma 2.8 i), ii)]{MR3694623} imply that $(K_n(z))_{n\in\N}$ is a discretely compact sequence.
\end{proof}

\begin{lemma}
\label{lemma: core}
For every $n\in\N$, the closure of $\sV_{n} = C^{\infty}_c(\Omega_{n} )^3 \cap H(\Div 0, \Omega_{n})$,
with respect to the $H(\curl, \Omega_{n})$-norm equals $\cH_{n} = H_0(\curl, \Omega_{n}) \cap H(\Div 0, \Omega_{n})$.
\end{lemma}

\begin{proof}
The subspace $\cH_{n}$ of $H_0(\curl, \Omega_{n})$ equipped with the norm $\|\cdot\|_{H(\curl,\Omega_n)}$ is closed since
$H_0(\curl, \Omega_{n})\cap H(\Div 0,\Omega_{n})$ with its norm $\|\cdot\|_{H(\curl,\Omega_n)}+\|\cdot\|_{H(\Div,\Omega_n)}$ is closed and
the norms $\|u\|_{H(\curl,\Omega_n)}+\|u\|_{H(\Div,\Omega_n)}$ and $\|u\|_{H(\curl,\Omega_n)}$ 
are equivalent for
 $u\in H_0(\curl, \Omega_n)\cap H(\Div 0,\Omega_{n})$.
Consequently, $\cH_{n}$ is a Hilbert space. Since $\sV_{n} \subset \cH_{n}$, the statement is equivalent to {proving} that
$\cH_{n} \cap \sV_{n}^\perp = \{0\}$ where the orthogonal complement is taken with respect to the inner product
$\langle \cdot,\cdot\rangle+\langle \curl\cdot,\curl\cdot\rangle$. Let $h \in \cH_{n} \cap \sV_{n}^\perp$. \vspace{-1mm}Then
\begin{equation}
\label{eq:perp}
   \langle h, v\rangle + \langle\curl h, \curl v\rangle = 0, \quad v \in \sV_n.
\end{equation}
First we claim that every $\varphi \!\in\! C^\infty_c(\Omega_{n})^3$ can be represented as $\varphi \!=\! \nabla \xi \!+\! v$ with $\xi \!\in\! C^\infty_c(\Omega_{n})$, $v \!\in\! \sV_n$. Indeed, the Dirichlet problem
\[
  -\Delta \xi =- \Div \varphi \quad \textup{in $\Omega_{n}$},  \quad
  \xi = 0 \quad \textup{on $\p \Omega_{n}$}
\]
has a unique solution $\xi \in C^\infty_c(\Omega_{n})$ and we can set $v = \varphi - \nabla \xi \in \sV_n$.
Using $\varphi = \nabla \xi + v$, $\curl \nabla \xi = 0$, $\langle h, \nabla \xi\rangle{=-\langle\Div h,\xi\rangle} = 0$
and \eqref{eq:perp}, we conclude
\begin{equation}
\label{eq:Ccinf}
    \langle h, \varphi\rangle + \langle\curl h, \curl \varphi\rangle = \langle h, v\rangle + \langle \curl h, \curl v\rangle = \,0, \quad \varphi \in C^\infty_c(\Omega_{n})^3.
\end{equation}
Since $C_c^{\infty}(\Omega_{n})^3$ is dense in $(H_0(\curl,\Omega_{n}),\|\cdot\|_{H(\curl,\Omega_n)})$, equality \eqref{eq:Ccinf} also holds
for all $\varphi \!\in\! H_0(\curl, \Omega_{n})$. Thus we can choose $\varphi=h\in H_0(\curl,\Omega_{n})$ in \eqref{eq:Ccinf} to obtain
\[
    0 \leq \norma{\curl h}^2 = - \norma{h}^2 \leq 0,
\]
so all the inequalities are equalities and hence $h = 0$.
\end{proof}

\begin{theorem}
\label{thm: final inclusion}
Suppose that $\sigma$, $\eps$ and $\mu$ satisfy the limiting assumption \eqref{eq:coeffs-infty}.
Let
$\cW(\omega):=- \omega(\omega\eps + \I \sigma)$, $\omega\!\in\!\C$, in $ L^2(\Omega)^3$ and $\cW_n(\omega)$ correspondingly in  $L^2(\Omega_n)^3$.
Then the limiting essential spectrum of $({\cL}_n)_{n\in\N}$ satisfies
\begin{align*}
\sigma_e(({\cL}_n)_{n\in\N}) \!&\subset\! \sigma_e((L_{\infty,n})_{n\in\N}) \cup \sigma_e((P_\nabla \cW_n(\cdot)|_{\nabla \!\dot H^1_0(\Omega_n)})_{n\in\N})
\\
\!&\subset\! W_{\!e}(L_\infty) \cup \sigma_e(P_\nabla \cW(\cdot) |_{\nabla \!\dot H^1_0(\Omega)}).
\end{align*}
\end{theorem}

\begin{proof}
By Proposition \ref{prop: equality lim ess spectra1} we have $\sigma_e((\cL_n)_{n\in\N})  = \sigma_e(({\cT_n})_{n\in\N})$.
Since ${\cT_n}(\omega)$ is a diagonally dominant operator matrix of order 0 for all $n \in \N$, $\omega \in \C$,
with bounds $a=\|\cW(\omega) |_{\nabla \dot H^1_0(\Omega)}\|$, $b=0$ in \eqref{limiting diag dom} uniform in $n$,
Theorem \ref{theorem: lim ess spec BOM} in Section \ref{sec: triangular} below implies that its limiting essential spectrum is the union of the limiting essential spectra of its diagonal entries,
\[
   \sigma_e(({\cT_n})_{n\in\N}) \subset \sigma_e((L_{\mu,n})_{n\in\N}) \cup \sigma_e((P_\nabla \cW_n(\cdot) |_{\nabla \dot H^1_0(\Omega_n)})_{n\in\N})
\]
By Proposition \ref{prop: sigma_e cL} it follows that $\sigma_e((L_{\mu,n})_{n\in\N}) \!=\! \sigma_e((L_{\infty,n})_{n\in\N})  \!\subset\! \R$.
Next \vspace{-1mm} we~show
\[
    \sigma_e((L_{\infty,n})_{n\in\N}) \subset W_e(L_\infty). 
\]

If $\omega \!\in\! \sigma_e((L_{\infty,n})_{n\in\N})$, by definition there exist $w_n \!\in\! \dom L_{\infty,n}(\omega) \subset
H_0(\curl, \Omega_n)$ $\cap H(\Div0, \Omega_n)$, $\norma{w_n} = 1$, $n\in\N$, $w_n \rightharpoonup 0$ and
$L_{\infty,n}(\omega)w_n \to 0$ as $n\to\infty$. Taking the scalar product with $w_n$, we find that
\[
    \langle L_{\infty,n}(\omega)w_n,w_n\rangle =
    \norma{ \mu_\infty^{-1/2} \curl_{0} w_n}^2 - \eps_\infty^{-1} \omega^2 \to 0
\]
 as $n\to\infty$. By Lemma \ref{lemma: core}, for each $n \in \N$ there exists $v_n \in C^\infty_c(\Omega_n)^3 \cap H(\Div 0, \Omega_n)$ with $\norma{v_n\!-\!w_n}^2 \!\leq\! 1/n$, $\norma{\curl(v_n \!-\! w_n)}^2 \!\leq\! 1/n$.
Let $v_n^{0} \in H_0(\curl, \Omega) \cap H(\Div 0, \Omega)$ be the extension of $v_n$ to $\Omega$ by zero for $n\in\N$. \vspace{-1mm} Then
\begin{align*}
  |\norma{ \mu_\infty^{-1/2}\!\curl v_n^{0}}^2 \!\!- \!\eps_\infty^{-1} \omega^2 \norma{v_n^{0}}^2|
	&\!\leq\! |\norma{ \mu_\infty^{-1/2}\!\curl w_n}^2 \!\!-\!  \eps_\infty^{-1} \omega^2 \norma{w_n}^2| \!+\! \frac{\!\!1  \!+\!  \eps_\infty^{-1}\omega^2\!\!}{n} \!\to\! 0
\end{align*}
 as $n\to\infty$.	Since $\norma{v_n^{0}} \to 1$ as $n \to \infty$, upon renormalisation of the elements $v_n^{0}$, we obtain $\omega \in W_e(L_\infty)$.

Finally, we prove that $\sigma_e((P_\nabla \cW_n(\cdot) |_{\nabla \dot H^1_0(\Omega_n)})_{n\in\N}) \!\subset\! \sigma_e(P_\nabla \cW(\cdot) |_{\nabla \dot H^1_0(\Omega)})$.  If $\omega \!\in\! \sigma_e((P_\nabla \cW_n(\cdot) |_{\nabla \dot H^1_0(\Omega)})_{n\in\N})$, there exist $u_n \!\in\! \dot H^1_0(\Omega_n)$, $\norma{\nabla u_n} \!=\! 1$, $n\!\in\!\N$, such that $\nabla u_n \!\rightharpoonup\! 0$ \vspace{-1mm} and
\[
     \norma{P_{\nabla \dot H^1_0(\Omega_n)} \omega(\omega \eps + \I \sigma) \nabla u_n} \to 0, \quad n \to \infty.
\]
Let $u_n^{0} \in \dot H^1_0(\Omega)$ be the extension of $u_n \in \dot H^1_0(\Omega_n)$ to $\Omega$ by zero for $n\in\N$. By standard properties of Sobolev spaces, $\nabla u_n^{0} = (\nabla u_n)^{0}$. Hence the sequence ${(u_n^{0})_{n\in\N}} \subset \dot H^1_0(\Omega)$ is such that $\norma{\nabla u_n^{0}} = 1$, $n\in\N$ , $\nabla u_n^{0} \rightharpoonup 0$ and
\[
   \norma{P_{\nabla \dot H^1_0(\Omega_n)} \omega(\omega \eps \!+\! \I \sigma) \nabla u_n^{0}} \to 0, \quad n \to \infty.
\]
Now the claim follows if we observe that $P_{\nabla \dot H^1_0(\Omega)} f = P_{\nabla \dot H^1_0(\Omega_n)} f$ for all $f \in  L^2(\Omega)^3$ with $\supp f \subset \Omega_n$.
\end{proof}

\begin{rem}
In fact, $\sigma_e((P_\nabla \cW_n(\cdot) |_{\nabla \dot H^1_0(\Omega_n)})_{n\in\N}) \!=\! \sigma_e(P_\nabla \cW(\cdot) |_{\nabla \dot H^1_0(\Omega)})$; here~the  inclusion `$\supset$' follows by \cite[Prop.~2.7]{MR3831156} since \vspace{-1mm} $P_\nabla \cW(\cdot) |_{\nabla\!\dot H^1_0(\Omega)}$, $P_\nabla \cW_n(\cdot)|_{\nabla\!\dot H^1_0(\Omega_n)}$,~$n\!\in\!\N$, are bounded and $P_\nabla \cW_n(\cdot) |_{\nabla\! \dot H^1_0(\Omega_n)}\!\!\sto\! P_\nabla \cW(\cdot) |_{\nabla\! \dot H^1_0(\Omega)}$ as $n\!\to\!\infty$, see
 \cite[Lemma~3.2]{MR3694623}. 
\end{rem}

\begin{proof}[Proof of Theorem {\rm \ref{theorem: main res}}.]
Due to Theorem \ref{Prop: spectra S_1 and L}, we have $0 \in \sigma_e(V)=\sigma_e(\cL)$ and hence $0 \notin \sigma_{\rm poll}((V_n)_{n\in\N})$,
$0 \notin \sigma_{\rm poll}((\cL_n)_{n\in\N})$.
Then, by \eqref{eq:spec-V-L} and \eqref{eq:poll}, it follows that
\begin{equation}
\label{eq:se1}
   \sigma_{\rm poll}((V_n)_{n\in\N}) \!=\! \sigma_{\rm poll}((V_n)_{n\in\N}) \setminus \{0\} \!=\! \sigma_{\rm poll}((\cL_n)_{n\in\N})  \setminus \{0\}
   \!=\! \sigma_{\rm poll}((\cL_n)_{n\in\N}).
\end{equation}
Now \eqref{eq: inclusion spectral poll} and Theorem \ref{thm: final inclusion} imply that
\begin{equation}
\label{eq:se2}
  \sigma_{\rm poll}((\cL_n)_{n\in\N}) \subset
  \sigma_e((\cL_n)_{n\in\N}) \subset W_e(L_\infty) \cup \sigma_e(P_\nabla \cW(\cdot) |_{\nabla \!\dot H^1_0(\Omega)}).
\end{equation}
Since Theorems \ref{sigma-ess} and \ref{Prop: spectra S_1 and L} yield that
\[
   \sigma_e(P_\nabla \cW(\cdot) |_{\nabla \!\dot H^1_0(\Omega)}) \subset \sigma_e(\cL) =\sigma_e(V) \subset \sigma(V),
\]
we easily deduce that $\sigma_{\rm poll}((V_n)_{n\in\N}) \cap \sigma_e(P_\nabla \cW(\cdot) |_{\nabla \!\dot H^1_0(\Omega)}) = \emptyset$.
This, together with \eqref{eq:se1}, \eqref{eq:se2} shows that $\sigma_{\rm poll}((V_n)_{n\in\N})\subset W_e(L_\infty)$, as required.

The approximation of isolated eigenvalues outside of $\sigma_e((\cL_n)_{n\in\N}) \cup \sigma_e((\cL_n^*)_{n\in\N})^*$ $=\sigma_e((\cL_n)_{n\in\N})$, and hence outside of $ W_e(L_\infty) \cup \sigma_e(P_\nabla \cW(\cdot) |_{\nabla \!\dot H^1_0(\Omega)})$
by \eqref{eq:se2}, is a consequence of Theorem~\ref{thm: bad set pencils}.
\end{proof}

If $\sigma = 0$, we can improve the spectral inclusion part in Theorem  \ref{theorem: main res} to \emph{all} spectral points in $\sigma(V)$.

\begin{theorem} \label{thm: main res self-adj}
Assume that $\sigma = 0$. In addition to the conclusions of Theorem~{\rm \ref{theorem: main res}}, for every $\omega \!\in\! \sigma(V)$ there exists a sequence  $\omega_n \!\in\! \sigma(V_n)$, $n\!\in\!\N$, with $\omega_n \!\to\! \omega$ as~$n \!\to\! \infty$.
\end{theorem}

\begin{proof}
When $\sigma \!=\! 0$,
the spectral problems for $V$ and $\cL$ re\-duce to  classical spectral problems for the self-adjoint operator matrix $\cA$ in \eqref{def:calA}.
We therefore have a domain truncation problem for a sequence of self-adjoint operators converging in strong resolvent sense, $(\cA_n\!-\!\omega)^{-1}\cP_n\!\s\! (\cA\!-\!\omega)^{-1}\!$,
where $\mathcal P_n\!:=\!{\rm diag}(P_n,P_n)$. In fact, the strong convergence
$V_n(\omega)^{-1}\mathcal P_n\!\s\! V(\omega)^{-1}$
follows from \eqref{eq:res-V} and Proposition~\ref{prop:gsrL};
here we need that $ L^2(\Omega)^3\oplus\mu^{1/2}H(\curl,\Omega)$ is dense in $L^2(\Omega)^3\oplus L^2(\Omega)^3$ and that,~for $u\!\in\! \mu^{1/2}H(\curl,\Omega)$, $P_n u\!\in\! \mu^{1/2}H(\curl,\Omega_n)$ with $\curl \mu^{1/2} P_n u\!=\!P_n \curl \mu^{1/2}u$ since $\curl$ is a local operator.
 Then $(\cA_n\!-\!\omega)^{-1}\cP_n\!\s\! (\cA\!-\!\omega)^{-1}$ by~\eqref{Vfact1}~and~\eqref{Vnfact1}.
 The~spectral inclusion
now follows from classical results, see e.g.\,\cite[Thm.\ VIII.24~(a)]{MR751959}.
\end{proof}


\section{\bf Abstract results for essential spectra and limiting essential spectra of triangular operator matrices}
\label{sec: triangular}

In this section we prove the abstract results on essential spectra and limiting essential spectra of triangular operator matrices used in Theorems~\ref{sigma-ess} and \ref{thm: final inclusion} and employed to prove our main result on spectral approximation, Theorem~\ref{theorem: main res}.
The results below are more general than what we needed there since we also admit~un\-bounded off-diagonal entries. Thus we decided to present them in a  separate~section.

In a product Hilbert space $\cH = \cH_1 \oplus \cH_2$ we consider lower triangular $2 \times 2$ operator
\vspace{-2mm} matrices
 \begin{equation} \label{def: cA}
  \cA = \begin{pmatrix}
  A & 0 \\
  C & D
  \end{pmatrix}
\end{equation}
such that $A$, $D$ are densely defined, $C$, $D$ are closable, $\dom(A) \subset \dom(C)$ and $\rho(A) \neq \emptyset$.
Then, e.g.\ by \cite[Thm.\ 2.2.8]{TreB},
$\cA$ is closable with closure
$$
 \overline{\cA} = \begin{pmatrix} A & 0 \\ C & \overline{D} \end{pmatrix}.
$$
The Schur Frobenius factorisation \cite[(2.2.10)]{TreB} of $\overline{\cA}$ simplifies to
\begin{equation} \label{eq: FSfact}
  \overline{\cA} - \la = \begin{pmatrix}
  I & 0 \\
  C(A- \la)^{-1} & I
  \end{pmatrix}
  \begin{pmatrix}
  A - \la & 0 \\
  0 & \overline{D} - \la
  \end{pmatrix},
	\quad \la \in\rho(A),
\end{equation}
and the first factor therein is bounded and boundedly invertible since $C$ is closable and $A$ is closed. Therefore,
$$
  \sigma_{ek}(\overline{\cA}) \setminus \sigma(A) = \sigma_{ek}(\overline{D}), \quad k=1,\dots,5.
$$
In the sequel we study the relation between $\sigma_{ek}(\overline{\cA})$
and the union $\sigma_{ek}(A) \cup \sigma_{ek}(\overline{D})$, mainly for $k\!=\!2$.
Here we denote the set of semi Fredholm oper\-ators with~finite nullity and finite defect by $\Phi_+$ and $\Phi_-$, respectively,
see \cite[Sect.~I.3]{EE}.

Note that even for diagonal operator matrices $\cA\!=\!{\rm diag}\,(A,D)$,
i.e.\ $C\!=\!0$, equal\-ity does not prevail for every $k\in\{1,\dots,5\}$;
in fact, by \cite[IX.\,(5.2)]{EE},
\begin{align}
\nonumber
  &\sigma_{e1}({\rm diag}\,(A,\overline{D})) \!\supset\! \sigma_{e1}(A) \cup \sigma_{ek}(\overline{D}), \\
\label{ek}	
  &\sigma_{ek}({\rm diag}\,(A,\overline{D})) \!=\! \sigma_{ek}(A) \cup \sigma_{ek}(\overline{D}), \quad k\!=\!2,3, \\
\nonumber	
	&\sigma_{ek}({\rm diag}\,(A,\overline{D})) \!\subset\! \sigma_{ek}(A) \cup \sigma_{ek}(\overline{D}), \quad k\!=\!4,5.
\end{align}

\vspace{-1mm}	

It is well-known that, for $C\ne 0$, the assumption $\dom(A) \subset \dom(C)$ is essential to have the inclusion
$\sigma_{ek}(\overline{\cA}) \subset \sigma_{ek}(A) \cup \sigma_{ek}(\overline{D})$, $k\!=\!1,\dots, 5$. In fact, if $A$, $D = 0$ and $C$ is boundedly invertible with dense domain $\dom(C) \subsetneq \cH_1$, then
$\sigma_{ek}(\overline{\cA}) \!=\! \sigma_{e1}(\overline{\cA}) \!=\! \C \ne \{0\} \!=\! \sigma_{ek}(A) \!=\! \sigma_{ek}(D)$ for $k\!=\!1,\dots,5$.

On the other hand, certain relative compactness assumptions may ensure equality; e.g.\
if for some $\mu \in \rho(A) \cap \rho(\overline{D})$ the operator $(\overline{D} \!-\! \mu)^{-1}C(A\!-\!\mu)^{-1}$ is compact, \vspace{-1mm} then,
by \cite[Thm.\ 2.4.8]{TreB},
\[
    \sigma_{e3}(\overline{\cA}) = \sigma_{e3}(A) \cup \sigma_{e3}(\overline{D}).
\]

In the following, for the case $k\!=\!2$, we characterise the difference between $\sigma_{e2}(\overline{\cA})$ and the union
$\sigma_{e2}(A) \cup \sigma_{e2}(\overline{D})$ and establish criteria for equality.
Here, for a closed linear operator $T$, we set $\sigma_{e2}^*(T)\!:=\!\{\la\!\in\!\C: \ran(T\!-\!\la) \mbox{ closed, codim  } \ran (T\!-\!\la) \!<\! \infty \}$;
note that then $\la \in \sigma_{e2}^*(T)$ if and only if $\overline{\la} \in \sigma_{e2}(T^*)$; see \cite[Sect.\ IX.1]{EE}.


\begin{theorem}
\label{thm: ess spec}
Let $\cA$ be as in \eqref{def: cA}, i.e.\
$A$, $D$ are densely defined, $C$, $D$ are closable, $\dom(A) \subset \dom(C)$ and $\rho(A) \neq \emptyset$. Then
\begin{equation}
\label{se2}
  \big( \sigma_{e2}(A) \setminus \sigma_{e2}^*(\overline{D} ) \big) \cup \sigma_{e2}(\overline{D} ) \subset \sigma_{e2}(\overline{\cA})
	\subset \sigma_{e2}(A)  \cup \sigma_{e2}(\overline{D} ),
\vspace{-1mm}	
\end{equation}
and hence
\[
  \sigma_{e2}(\overline{\cA}) \cup \big( \sigma_{e2}(A) \cap \sigma_{e2}^*(\overline{D} ) \big)
	= \sigma_{e2}(A)  \cup \sigma_{e2}(\overline{D} );
\]
in particular, if  $\sigma_{e2}^*(\overline{D})  = \sigma_{e2}(\overline{D})$ or if $ \sigma_{e2}(A) \cap \sigma_{e2}^*(\overline{D})=\emptyset$, then
\[
  \sigma_{e2}(\overline{\cA}) = \sigma_{e2}(A)  \cup \sigma_{e2}(\overline{D} ).
\]
\end{theorem}

\begin{proof}
First we prove the left inclusion in \eqref{se2}.
The enclosure $\sigma_{e2}(\overline{D} ) \subset \sigma_{e2}(\overline{\cA})$ is trivial; we just add a zero first component to a singular sequence
coming from $\overline{D}$. Now let $\la \in \sigma_{e2}(A) \setminus \sigma_{e2}^*(\overline{D} )$.
Then $\overline{D}-\la \in \Phi_-$ and hence $\overline{D}-\la$ has an approximate right inverse $R_\la \in \cB(\cH_2)$,
see \cite[Thm.\ I.3.11]{EE}, i.e.\
$(\overline{D}-\la) R_\la = I_{\cH_2} + F_\la$ with $F_\la \in \cB(\cH_2)$ of finite rank.  Since $\la \in \sigma_{e2}(A)$, there exists $(x_n)_{n\in\N} \subset \dom A$, $\|x_n\|\!=\!1$, $x_n \!\rightharpoonup\! 0$, $(A-\la)x_n \!\to\! 0$, $n\!\to\! \infty$. This implies that $(Ax_n)_{n\in\N}$ is bounded.
Since $C$ is closable and $\dom A \!\subset\! \dom C$, $C$~is $A$-bounded and hence $(C x_n)_{n\in\N}$ is bounded as~well.

Now set $y_n:=-R_\la Cx_n$, $n\in\N$. Then $(y_n)_{n\in\N}$ is bounded and, for $n\in\N$,
\[
  Cx_n+(\overline{D}-\la)y_n = Cx_n - (\overline{D}-\la) R_\la Cx_n = Cx_n- (I_{\cH_2}+F_\la) Cx_n = - F_\la Cx_n.
\]
Since $(Cx_n)_{n\in\N}$ is bounded and $F_\la\in \cB(\cH_2)$ has finite rank,
upon choosing a subsequence, we may assume that
\[
  Cx_n+(\overline{D}-\la)y_n = - F_\la Cx_n \to 0, \quad n\to\infty.
\]
It remains to be shown that $y_n\!=\!-R_\la Cx_n\!\!\rightharpoonup\! 0$~for $n\!\to\!\infty$.
To this end, let $\mu\!\in\!\rho(A)$  $(\ne\! \emptyset)$. Then $C(A\!-\!\mu)^{-1}$ is bounded since $C$ is closable and $A$ is~closed. Thus
\[
  Cx_n = C(A -\mu)^{-1} \big( \underbrace{(A-\lambda) x_n}_{\to 0} + (\lambda-\mu) \underbrace{x_n}_{\rightharpoonup 0}  \big) \rightharpoonup 0,
	\quad n\to\infty,
\vspace{-2mm}	
\]
and hence, since $R_\la$ is bounded, $y_n\!\!\rightharpoonup\! 0$~for $n\!\to\!\infty$, as required.
Finally, if we set $v_n:=( x_n, y_n)$, $n\in\N$, and normalise $v_n$,
we obtain a singular sequence for $\cA$ at $\la$ and hence $\la \in \sigma_{e2}(\overline{\cA})$.

In order to prove the second inclusion in \eqref{se2}, let $\la\notin \sigma_{e2}(A) \cup  \sigma_{e2}(\overline{D})$, i.e.\
$A-\la$, $\overline{D}-\la \in \Phi_+$.
For arbitrary $\mu\!>\!0$, set
\begin{equation}
\label{mu-trick}
  \cA_\mu:= M_\mu^{-1} \cA M_\mu = \begin{pmatrix} A & 0 \\ \mu C & D \end{pmatrix},
	\quad
	M_\mu:= \begin{pmatrix} \mu I & 0 \\ 0 & I \end{pmatrix}.
\end{equation}
Then $\sigma_{e2}(\overline{\cA})\!=\!\sigma_{e2}(\overline{\cA_\mu})$ because $M_\mu$ is bounded and boundedly invertible.
Due~to the stab\-ility of semi-Fredholmness, see \cite[Thm.\ I.3.22, Rem.\ I.3.27] {EE}, and since diag\,$(A\!-\!\la,\overline{D}\!-\!\la) \!\in\! \Phi_+$,
we can choose $\mu\!>\!0$ so small that $\overline{\cA_\mu}\!-\!\la \!\in\! \Phi_+$ and thus $\la\!\notin\! \sigma_{e2}(\overline{\cA})$.%

Finally, the last two claims are obvious from \eqref{se2}.
\end{proof}

\begin{rem}
For the second inclusion in \eqref{se2}, in the same way as in the proof of Theorem \ref{thm: ess spec}, one can also show that
$\sigma_{ek}(\overline{\cA}) \subset \sigma_{ek}(A) \cup \sigma_{ek}(\overline{D})$ for $k=3,4,5$.
Here the Fredholm stability results [EE, Thm.\ I.3.22 and Rem. I.3.27] for $\Phi_\pm$ and hence $\Phi$, together with the stability of the index therein, give the inclusions for $k=3,4$, while for $k=5$ the stability of bounded invertibility \cite[Thm.\ IV.1.16]{MR1335452} is used.

The first inclusion in \eqref{se2} also holds for $k\!=\!3$, i.e.\
$\sigma_{e3}(\overline{\cA}) \cup \big( \sigma_{e3}(A) \cap \sigma_{e3}^*(\overline{D} ) \big)
\!=\! \sigma_{e3}(A)  \cup \sigma_{e2}(\overline{D})$, whereas for $k\!=\!4$
the difference between $\sigma_{e4}(\overline{\cA})$ and $\sigma_{e4}(A) \cup \sigma_{e4}(\overline{D})$
has a much less elegant description.
\end{rem}

\begin{corollary}
\label{cor: J-self ess}
Let $\cA$ be as in \eqref{def: cA}. If $D$ is $\cJ$-self-adjoint
for some conjugation~$\cJ$ in $\cH_2$, i.e.\ $\cJ^2\!=\!I_{\cH_2}$, $(\cJ x,\cJ y)=(x,y)$ for $x,y \in \cH_2$, then
$$\sigma_{ek}(\overline{\cA}) = \sigma_{ek}(A) \cup \sigma_{ek}(D), \quad k= 2,3,4.$$
\end{corollary}

\begin{proof}
We prove the claim for $k=2$; the proof for $k=3,4$ is left to the reader.
Since $D$ is $\cJ$-self-adjoint, $\dim \ker (D-\la) = \dim \ker (D^*-\overline{\la})$ for $\la\in\C$ by \cite[Lemma III.5.4]{EE}.
Hence, either $\ran(D-\la)$ is not closed  or $\dim\ker (D-\la)=\dim\ran(D-\la)^\perp$ so that $D-\la$ is a Fredholm of index $0$
for $\la\in\C$. This proves that $\sigma_{e2}(D) = \sigma_{e2}^*(D)$ and hence Theorem \ref{thm: ess spec} yields the claim.
\end{proof}

The following counter-examples show that, in general,
neither of the inclusions in \eqref{se2} is an equality, 
even when all entries of $\cA$ are bounded.

\begin{example} \label{ex: sharp e2}
Let $D\in \cB(\cH_2)$ be a bounded linear operator in some Hilbert space $\cH_2$ such that $0 \in \sigma^*_{e2}(D) \setminus \sigma_{e2}(D)$,
e.g.\ $\dim \ker D <\infty$ and $\dim (\ran D)^ \perp = \infty$. For example, we can choose $D : \ell^2(\N) \to \ell^2(\N)$
given by $D e_{k} := e_{2k}$, $k \in \N$.

i) If $A$ in $\cH_1\!=\!\cH_2$ is compact with $\sigma_{e2}(A)\!=\!\{0\}$, $C\!=\!A$ and $D$ is as above,~then,
for $\cA$ as in \vspace{-1mm} \eqref{def: cA},
\[
 0 \notin \big( \sigma_{e2}(A) \setminus \sigma_{e2}^*(\overline{D} ) \big) \cup \sigma_{e2}(\overline{D} ) = \sigma_{e2}(\overline{D}), \quad
 0 \in \sigma_{e2}(\overline{\cA});
\]
for the latter note that since $0\in \sigma_{e2}(A)$, there exists a singular sequence $(x_n)_{n\in\N}$ for $A$ and then, since $C=A$,
it follows that $((x_n,0))_{n\in\N}$ is a singular sequence for~$\cA$. This example shows that the first inclusion in \eqref{se2} is not an~equality.

ii) If $P_{(\ran D)^\perp}$ is the orthogonal projection on $(\ran D)^\perp\!\!=\!\ker D^*\!$, $A\!=\!I\!-\!P_{(\ran D)^\perp}$ in $\cH_1\!=\!\cH_2$
and $C=P_{(\ran D)^\perp}$, then $\ker A = (\ran P)^\perp$ so that $0\in\sigma_{e2}(A)$ and, for $\cA$ as in  \vspace{-1mm} \eqref{def: cA},
\[
  0 \notin \sigma_{e2}(\overline{\cA}), \quad
	0 \in \sigma_{e2}(A) \cap \big(\sigma^*_{e2}(D) \!\setminus\! \sigma_{e2}(D) \big) \subset \sigma_{e2}(A) {\cup} \sigma_{e2}(D).
\]
To prove the former, suppose to the contrary that $0\in \sigma_{e2}(\overline{\cA})$. Then there would exist a sequence $h_n \!=\! (x_n, y_n)^t \!\in\!
\cH_1 \oplus \cH_1$ such that $\norma{h_n} \!=\! 1$, $h_n \!\rightharpoonup\! 0$ for  \vspace{-1mm} $n\!\to\!\infty$,~and
\begin{equation}
\begin{aligned}
\label{ex: limits2}
&(I-P_{(\ran D)^\perp})x_n &\to 0, \\
&P_{(\ran D)^\perp} x_n + D y_n \!\!\! &\to 0,
\end{aligned} \qquad n\to\infty.
\end{equation}
The second relation in \eqref{ex: limits2} implies $P_{(\ran D)^\perp} x_n \!\to\! 0$ and $D y_n \!\to\! 0$ as $n \!\to\! \infty$.
To\-gether with the first relation in \eqref{ex: limits2}, we conclude that  $x_n \!\to\! 0$ and hence $\|y_n\|\!\to\! 1$ for $n\!\to\!\infty$;
hence upon choosing a subsequence we can assume that $y_n\!\ne\! 0$, $n\!\in\!\N$.
Since $h_n \!\rightharpoonup\! 0$ for $n\!\to\!\infty$ implies that $y_n \!\rightharpoonup\! 0$ for $n\!\to\!\infty$, we conclude that
$\widehat{y}_n \!:=\! y_n/\norma{y_n}$, $n\!\in\!\N$, is a singular sequence for $D$, a contradiction, since $0 \!\notin\! \sigma_{e2}(D)$.
This example proves that the second inclusion in \eqref{se2} is not an~equality.
\end{example}

\begin{rem}
We mention that Theorem \ref{thm: ess spec}, see also Corollary \ref{cor: J-self ess}, provides a direct proof of \cite[Prop.~25]{MR3942228} on Maxwell's equations. Indeed, our results apply to the lower triangular operator matrix $\widetilde{V_\omega}$ in \cite[(26)]{MR3942228} therein
whose entries $A_\omega$, $C_\omega$, $D_\omega$ are $2\times 2$ operator matrices themselves with unbounded entries.
Standard computations show that $\sigma_{e2}(D_\omega) = \sigma^*_{e2}(D_\omega)$ and hence Theorem \ref{thm: ess spec} yields
the equality $\sigma_{e2}(\widetilde{V_\omega}) = \sigma_{e2}(\cA_\omega) \cup \sigma_{e2}(\cD_\omega)$, which had to be proved
in \cite[Prop.~25]{MR3942228} for the concrete operators therein.
\end{rem}

Finally, we provide some results on the limiting essential spectrum of sequences of lower triangular operator matrices.
The first results of this kind were established in the thesis \cite[Sect.\ 2.3]{SThesis} without the assumption of triangularity
for bounded off-diagonal corners.

Let $\cH_0=\cH_{1,0}\oplus \cH_{2,0}$ be a Hilbert space, $\cH_i$, $\cH_{i,n} \subset \cH_{i,0}$ be closed subspaces
for $n\in\N$ and $i=1,2$. Let $P_i: \cH_{i,0} \to \cH_i$, $P_{i,n}: \cH_{i,0} \to \cH_{i,n}$, $n\in\N$,
be the corresponding orthogonal projections and assume that $P_{i,n} \s P_i$  in $\cH_{i,0}$.

In addition to $\cA$ as in \eqref{def: cA}, in the subspaces $\cH_n=\cH_{1,n} \oplus \cH_{2,n}$ of $\cH_0=\cH_{1,0}\oplus \cH_{2,0}$, $n\in\N$, we
consider the lower triangular operator  \vspace{-1mm} matrices 
\begin{equation} \label{def: cA^n}
\cA_n =
\begin{pmatrix}
A_n & 0 \\
C_n & D_n
\end{pmatrix}
\end{equation}
satisfying analogous assumptions as $\cA$, i.e.\
$A_n$, $D_n$ are densely defined, $C_n$, $D_n$ are closable, $\dom(A_n) \subset \dom(C_n)$ and $\rho(A_n) \neq \emptyset$.

While the assumptions ensure that each $C_n$ is $A_n$-bounded, we suppose
that the operator sequence $(C_n)_{n\in\N}$ is uniformly $(A_n)_{n\in\N}$-bounded,  i.e.\ there exist $a$, $b > 0$ and $N \in \N$ such that 
\begin{equation}
\label{limiting diag dom}
   \norma{C_n x}^2 \leq a \norma{x}^2 + b \norma{A_n x}, \quad x \in \dom(A_n), \ n\in \N, \ n\ge N.
\end{equation}


\begin{theorem}
\label{theorem: lim ess spec BOM}
Let $\cA_n$ be defined as in \eqref{def: cA^n}, $n \in \N$, and assume that
$(C_n)_{n\in\N}$ is uniformly $(A_n)_{n\in\N}$-bounded, i.e.~\eqref{limiting diag dom} holds. Then
\[
  \sigma_{e2}((D_n)_{n \in \N}) \subset \sigma_{e2}((\overline{\cA_n})_{n \in \N}) \subset \sigma_{e2}((A_n)_{n \in \N}) \cup \sigma_{e2}((D_n)_{n \in \N}).
\]
\end{theorem}

\begin{proof}
The proof is similar to the proof of the respective parts of the proof of Theorem  \ref{thm: ess spec}; note that, due to assumption
\eqref{limiting diag dom}, we can choose $\mu>0$ in the transformation of $\cA_n$, see \eqref{mu-trick}, independently of $n\in I \subset \N$.
The proof is also analogous to the proof of \cite[Prop.\ 2.3.1 i)]{SThesis} if we observe that the sequence $B_n=0$ of zero operators is discretely compact and we replace the uni\-form boundedness property of $(\|C_n\|)_{n\in\N}$ therein by \eqref{limiting diag dom}. We leave the details to the reader.
\end{proof}


\section*{Appendix - Computations for Example \ref{ex:waveguide}}
\label{sec:appendix}

In this appendix we provide the computations for Example \ref{ex:waveguide} where we considered the semi-infinite cylinder $\Omega\!=\!(0, \infty)\!\times\!(0,L_2)\!\times\!(0,L_3)$ and supposed that $\eps\!=\!\mu \!=\! \id$ everywhere, and $\sigma \!=\! \id$ if $x_1\!\in \!(0,1)$, else $\sigma\!=\!0$, i.e.\ $\sigma \!=\! \chi_K \id$ with $K:=(0,1)\!\times\!(0,L_2)\!\times\!(0,L_3)$. 

With this choice of the coefficients the Maxwell system in $\Omega$ in \eqref{intro: Max1} becomes
\begin{alignat*}{2}
\curl^2 E &= \omega(\omega+\I)E, \;\;\; &&0 < x_1 < 1, \\
\curl^2 E &= \omega^2 E, \;\;\; &&1 < x_1 < X_n,
\end{alignat*}
with the condition that $E$ and $\curl E$ are continuous across the interface $x_1=1$.
The boundary condition in \eqref{intro: Max1} was $\nu\times E =0$ on the boundary $\partial \Omega$. We use the notation ${\bf{n}} = (n_2, n_3)
\in \N_0^2$.

\noindent
{\bf \underline{Case 1, {\boldmath $x_1\in (0,1)$:}}}
For this range of $x_1$, if we \vspace{-1mm} set
 \[ \alpha_{\bf n}(\omega) := \sqrt{\pi^2n_2^2/L_2^2 + \pi^2n_3^2/L_3^2 - \omega(\omega+\I)}, \]
the correct ansatz to use for the solution of this problem by Fourier expansions is
\begin{align*}
E_1(x_1,x_2,x_3) &= \sum_{{\bf n}\in {\mathbb N}^2} \hat{E}_1({\bf n}) \sin\left(\frac{\pi n_2}{L_2}x_2\right) \sin\left(\frac{\pi n_3}{L_3}x_3\right)
 \frac{\cosh(\alpha_{\bf n}(\omega)x_1)}{\cosh(\alpha_{\bf n}(\omega))}, \\[-0.7mm]
E_2(x_1,x_2,x_3) &= \!\!\sum_{{\bf n}\in {\mathbb N}_0\times\N} \!\hat{E}_2({\bf n}) \cos\left(\frac{\pi n_2}{L_2}x_2\right) \sin\left(\frac{\pi n_3}{L_3}x_3\right)
 \frac{\sinh(\alpha_{\bf n}(\omega)x_1)}{\sinh(\alpha_{\bf n}(\omega))}, \\[-0.7mm]
E_3(x_1,x_2,x_3) &= \!\!\sum_{{\bf n}\in {\mathbb N}\times\N_0} \hat{E}_3({\bf n}) \sin\left(\frac{\pi n_2}{L_2}x_2\right) \cos\left(\frac{\pi n_3}{L_3}x_3\right)
 \frac{\sinh(\alpha_{\bf n}(\omega)x_1)}{\sinh(\alpha_{\bf n}(\omega))}.
\end{align*}
{\bf \underline{Case 2, {\boldmath$x_1\in (1,X_n)$:}}}
For this range of $x_1$, if we \vspace{-1mm} set
 \[ \beta_{\bf n}(\omega) := \sqrt{\pi^2n_2^2/L_2^2 + \pi^2n_3^2/L_3^2 - \omega^2}, \]
the correct ansatz to use for the solution of this problem by Fourier expansions is
\begin{align*}
E_1(x_1,x_2,x_3) &= \sum_{{\bf n}\in {\mathbb N}^2} \hat{E}_1({\bf n}) \sin\left(\frac{\pi n_2}{L_2}x_2\right) \sin\left(\frac{\pi n_3}{L_3}x_3\right)
 \frac{\cosh(\beta_{\bf n}(\omega)(X_n-x_1))}{\cosh(\beta_{\bf n}(\omega)(X_n-1))}, \\[-0.7mm]
E_2(x_1,x_2,x_3) &= \!\!\sum_{{\bf n}\in {\mathbb N}_0\times \N} \!\hat{E}_2({\bf n}) \cos\left(\frac{\pi n_2}{L_2}x_2\right) \sin\left(\frac{\pi n_3}{L_3}x_3\right)
 \frac{\sinh(\beta_{\bf n}(\omega))(X_n-x_1))}{\sinh(\beta_{\bf n}(\omega)(X_n-1))}, \\[-8mm]
\intertext{\phantom{I}}  
E_3(x_1,x_2,x_3) &= \!\!\sum_{{\bf n}\in {\mathbb N}\times\N_0} \!\hat{E}_3({\bf n}) \sin\left(\frac{\pi n_2}{L_2}x_2\right) \cos\left(\frac{\pi n_3}{L_3}x_3\right)
 \frac{\sinh(\beta_{\bf n}(\omega)(X_n-x_1))}{\sinh(\beta_{\bf n}(\omega)(X_n-1))}.
\end{align*}
In the definition of $\alpha_{\bf n}(\omega)$ and $\beta_{\bf n}(\omega)$ we choose the branch of the square root with non-negative real part.
The above two ans\"atze ensure the continuity of $E$ across the interface $x_1=1$. To ensure continuity of $\curl E$ across this interface, a direct calculation shows that the first component of $\curl E$ is automatically
continuous across the interface $x_1=1$; it is therefore $\nu\times \curl E$ for $\nu=(1,0,0)$ that gives rise to non-trivial conditions. Direct calculations using the formulae above yield the condition that for some ${\bf n}\in{\mathbb N}_0^2$ with $|{\bf n}|>0$,
\[
  \alpha_{\bf n}(\omega)\coth(\alpha_{\bf n}(\omega))  + \beta_{\bf n}(\omega)\coth(\beta_{\bf n}(\omega)(X_n-1)) = 0.
\]
Next we prove equation \eqref{eq:sigmaessV}, namely
\[
  \sigma_{e}(V) =  (-\infty, -\pi/L] \cup [\pi/L, + \infty) \cup (-\I \{0,1/2,1\}), \quad L = \max\{L_2, L_3\}.
\]
Indeed, due to \cite[Thm.\ 6]{MR3942228}, see also Remark \ref{rem:spec-eq}, we have
\[
  \sigma_{e}(V) =  \sigma_{e}(V^0) \cup \sigma_{e}(\Div(\cW(\cdot)\nabla)), \quad \cW(\omega):=-\omega(\omega + \I \chi_K), \ \omega\in\C,
\]
where $V^0$ is the Maxwell pencil $\I V(\cdot)$ with $\sigma \!=\! 0$ and $\Div(\cW(\omega)\nabla)$ acts from
$\dot{H}^1_0(\Omega)$ to its dual $\dot H^{-1}(\Omega)$ for each $\omega\in\C$.
%
Clearly, $\sigma_{e}(\Div\cW(\cdot)\nabla))\!=\! \{0\} \cup \sigma_{e}(\Div(\cU(\cdot)\nabla))$ with $\cU(\omega)\!:=\! -(\omega + \I \chi_K)$, $\omega\!\in\!\C$. %
We start by \vspace{-1mm} showing
\begin{equation}
\label{mmextra} 
\sigma_{e}(\Div(\cU(\cdot)\nabla)) = -\I\{0,1/2,1\}. 
\end{equation}
By inspection, one has the inclusion $\sigma_{e}(\Div\cU(\cdot))\nabla)) \subset -\I[0,1]$. The values $\omega = 0$
and $\omega = -\I$ are both easily seen to be eigenvalues of infinite multiplicity, with eigenfunctions which are $C_c^\infty$-functions supported
entirely outside $K$ (for $\omega = 0$ where $\cU(0)\!=\!-\I \chi_K$) or in the interior of $K$ (for $\omega = -\I$  where $\cU(-\I)\!=\!-\I \chi_{\Omega\setminus K}$). It remains to examine whether any other $\omega
\in -\I[0,1]$ have the property that $0$ lies in the essential spectrum of the Dirichlet operator $-\Div((\omega + \I \chi_K)\nabla)$.
Since the coefficient $\omega + \I\chi_K$ takes only the values $\omega+\I$ and $\omega$, 
whose ratio is  $1+\I/\omega$, the results in \cite{MR4041099} suggest that the only value of $\omega$ for which this may happen is $\omega = -\I/2$, which has the property
that $1+\I/\omega = -1$. Unfortunately the hypotheses in \cite{MR4041099} do not quite cover our case, so we outline a proof by direct calculation.
 By Glazman decomposition, one shows that
\[ 
\mbox{$-\Div((\omega + \I \chi_K)\nabla)$ is invertible} \iff 
\mbox{$(-\I\omega+1)\Lambda_{L}  - \I \omega\Lambda_R$ is invertible}, \]
where $\Lambda_L$ and $\Lambda_R$ are the left- and
right-hand Dirichlet to Neumann maps on the interface $x_1= 1$. Take a basis of transverse eigenfunctions $(\psi_n(x_2,x_3))_{n\in\N}$, 
e.g.\ some ordering of $\sin\left(\frac{n_2 \pi}{L_2}x_2\right) \sin\left(\frac{n_3 \pi}{L_3}x_3\right)$, with strictly positive eigenvalues $(\kappa_n^2)_{n\in\N}$. In such a basis,
 both $\Lambda_L$ and $\Lambda_R$ are represented by diagonal matrices,
\[ 
\Lambda_L = \mbox{diag}((\kappa_n\coth(\kappa_n))_{n\in\N}), \quad \Lambda_R = \mbox{diag}((\kappa_n)_{n\in\N}). 
\]
Putting $\omega = -\I \nu$ with $\nu\in(0,1)$, we find
\[ 
(-\I\omega+1)\Lambda_{L} - \I\omega\Lambda_R = \mbox{diag}((\kappa_n((1-\nu)\coth(\kappa_n)-\nu))_{n\in\N}). 
\]
If $0<\nu<1/2$, then this infinite matrix has a bounded, positive inverse, so no $\omega \in -\I(0,1/2)$ lies in the essential spectrum.
If $1/2<\nu<1$, then the matrix has (at worst) a finite-dimensional kernel, but is still a finite-rank perturbation of a matrix with
bounded inverse. From this fact and the Glazman decomposition, one is able to argue that $0$ lies outside the essential
spectrum of $-\Div((\omega + \I \chi_K)\nabla)$ for $\omega\in -\I(1/2,1)$. It remains only to show that $\omega = -\I/2$ does
indeed have the property that $0$ lies in the essential spectrum of $-\Div((\omega + \I \chi_K)\nabla)$. We prove this by
directly verifying that the functions
\[ 
u_{n}(x_1,x_2,x_3) := 
\begin{cases}
 (1\!-\!(x_1\!-\!1)\kappa_n(\coth(\kappa_n)\!-\!1))\psi_n(x_2,x_3) { \frac{\sinh(\kappa_n x_1)}{\sinh(\kappa_n)}}, & x_1 \!\in\! (0,1), \\
 \psi_n(x_2,x_3)\exp(-\kappa_n(x_1\!-\!1)), & x_1 > 1,
 \end{cases} 
 \]
form a Weyl singular sequence for $-\Div((-\I/2 \!+\! \I \chi_K)\nabla)$ acting from $\dot{H}^1_0(\Omega)$ to $\dot H^{-1}(\Omega)$.
They satisfy the compatibility conditions across $x_1\!=\!1$ and, by direct calculation,
\[ 
 -\Div((-\I/2+\I\chi_K)\nabla u_n) =  \pm (\I/2) \Delta u_n, 
\]
with $-$ for $x_1<1$ and $+$ for $x_1>1$. Since $\Delta u_n =0 $ for $x_1>1$,  we have
\[
-\mbox{div}((-\I/2+\I\chi_K)\nabla u_n) = -(\I/2)\Delta u_n 
\] 
in all cases, and it suffices to show that
\begin{equation}
\label{ccextra}
\frac{\|\Delta u_n\|_{\dot{H}^{-1}(\Omega)}}{\| u_n \|_{\dot{H}^1_0(\Omega)}} \to 0, \quad  n\to\infty. 
\end{equation}
Since the Dirichlet Laplacian in $\Omega$ has spectrum $[\kappa_1^2,\infty)$, we have $\Delta \!\geq\! \kappa_1^2$ and  thus,
by testing with $C_c^\infty(\Omega)$ functions, one may show \vspace{-1mm} that
\begin{align*}
\label{ccextra2}
	 \| \Delta u_n \|_{\dot{H}^{-1}(\Omega)} 
	  \!\leq\! \frac 1 {\kappa_1} \| \Delta u_n \|_{L^2(\Omega)}.
\end{align*}
By direct calculation,  $\Delta u_n$ is non-trivial only for $x_1 \leq 1$, and
\[ 
\Delta u_n = -2\kappa_n^2(\coth(\kappa_n)-1)\psi_n(x_2,x_3)\frac{\cosh(\kappa_n x_1)}{\sinh(\kappa_n)}. 
\]%
It follows from elementary estimates that, for $n\to\infty$,
\begin{align*}
&\| \Delta u_n \|_{L^2(\Omega)} \leq \mbox{O}(\kappa_n^{2}(\coth(\kappa_n)\!-\!1)) \,\|\psi_n\|_{L^2((0,L_2)\!\times(0,L_3))}, \\
&\| \nabla u_n \|_{L^2(\Omega)} \geq \| \nabla u_n \|_{L^2((1,\infty)\!\times\!(0,L_2)\!\times(0,L_3))} \geq \mbox{O}(\kappa_n^{1/2}) \,\|\psi_n\|_{L^2((0,L_2)\!\times(0,L_3))},
\end{align*}
and hence
\[ 
\frac{\| \Delta u_n \|_{\dot{H}^{-1}(\Omega)} }{\| u_n \|_{\dot{H}^1_0(\Omega)}} 
\leq \frac 1{\kappa_1} \frac{\| \Delta u_n \|_{L^2(\Omega)} }{\| \nabla u_n \|_{L_2(\Omega)}} 
= \mbox{O}(\kappa_n^{3/2}(\coth(\kappa_n)-1)). 
\]
This completes the proof of \eqref{ccextra} and hence of \eqref{mmextra}.

Regarding $\sigma_{e}(V^0)$, we can use the Fourier expansion
\begin{align*}
 E_1(x_1,x_2,x_3) &= \sum_{{\bf n}\in {\mathbb N}^2} \hat{E}_1({\bf n}) \sin\left(\frac{\pi n_2}{L_2}x_2\right) \sin\left(\frac{\pi n_3}{L_3}x_3\right)
 \cos(\xi x_1), \\[-0.7mm]
 E_2(x_1,x_2,x_3) &= \!\!\sum_{{\bf n}\in {\mathbb N}_0\times\N} \!\hat{E}_2({\bf n}) \cos\left(\frac{\pi n_2}{L_2}x_2\right) \sin\left(\frac{\pi n_3}{L_3}x_3\right)
 \sin(\xi x_1), \\[-0.7mm]
 E_3(x_1,x_2,x_3) &= \!\!\sum_{{\bf n}\in {\mathbb N}\times\N_0} \!\hat{E}_3({\bf n}) \sin\left(\frac{\pi n_2}{L_2}x_2\right) \cos\left(\frac{\pi n_3}{L_3}x_3\right)
 \sin(\xi x_1).
\\[-8mm]
\end{align*}
In the new Fourier coordinates the matrix differential \vspace{-1.5mm} expression
$\begin{pmatrix}
- \I \omega & \curl \\
\curl & \I \omega
\end{pmatrix}$
corresponds \vspace{-1mm} to
\[
{\cV}(\omega, \xi,n_2,n_3) =
\begin{pmatrix}
-\I \omega  & 0 & 0 &   0 &  \frac{\pi n_3}{L_3} & -\frac{\pi n_2}{L_2} \\
0 & -\I \omega &  0 &   -\frac{\pi n_3}{L_3} & 0 & -\xi \\
0 & 0 & -\I \omega  &   \frac{\pi n_2}{L_2} & \xi & 0  \\
0 &  -\frac{\pi n_3}{L_3} & \frac{\pi n_2}{L_2} & \I \omega & 0 & 0 \\
-\frac{\pi n_3}{L_3} & 0 & -\xi & 0 & \I \omega & 0 \\
-\frac{\pi n_2}{L_2} & \xi & 0  & 0 & 0 & \I \omega
\end{pmatrix}.
\vspace{-2mm}
\]
Then we \vspace{-2mm}have
\[
  \det {\cV}(\omega, \xi,n_2,n_3) = \omega^2 \bigg( \xi^2 + \frac{\pi^2 n_2^2}{L_2^2} + \frac{\pi^2 n_3^2}{L_3^2} - \omega^2\bigg)^2.
\]
As in \cite[Ex.\ 10]{MR3942228} the essential spectrum is the set of $\omega \in \C$ such that for some $\xi\in\R$ and
$(n_2,n_3)\neq (0,0)$, one has $\det \cV(\omega, \xi,n_2,n_3) =0$. This yields
\[
   \sigma_{e}(V^0) = \{0\} \cup \bigg\{\omega \in \R : \omega^2 \geq \frac{\pi^2}{\max \{L_2^2,L_3^2\}} \bigg\}.
\]

\vspace{3mm}

\noindent
{\small
{\bf Acknowledgements.}
The second and third authors are grateful for the support of the `Engineering and Physical Sciences Research Council'\,(EPSRC) under  grant EP\!/T000902/1, \textit{`A new paradigm for spectral localisation of operator pencils and analytic operator-valued functions'}.
The last author gratefully acknowledges the support of Schweizer Nationalfonds (SNF) through grant $200021\_204788$, \textit{`Spectral approximation beyond normality'}.
}

{\small

\bibliographystyle{abbrv}

\bibliography{Maxbib}

}
\end{document}